\documentclass{amsart}
\pdfoutput=1
\usepackage[utf8]{inputenc} 

\usepackage[T1]{fontenc}    
\usepackage{ae,aecompl}     
\usepackage{url}

\usepackage{amsmath}
\usepackage{amsfonts}
\usepackage{amssymb}
\usepackage{amsthm} %
\usepackage{mathrsfs} %
\usepackage{enumerate} 
\usepackage{verbatim}	

\usepackage{tikz}
\usepackage{pgfplots}
\usetikzlibrary{calc}
\usetikzlibrary{shapes}

\usetikzlibrary{intersections}
\usetikzlibrary{patterns}

\usepackage{caption}
\usepackage[hyperfootnotes]{hyperref}

\theoremstyle{plain} 
\newtheorem{thm}{Theorem}[section]
\newtheorem{cor}[thm]{Corollary}
\newtheorem{prop}[thm]{Proposition}
\newtheorem{lem}[thm]{Lemma}
\newtheorem{conject}{Conjecture}

\theoremstyle{definition}
\newtheorem{defi}[thm]{Definition}

\newtheorem{remark}[thm]{Remark}

\newtheorem{ex}[thm]{Example}

\usepackage{color}
\definecolor{Ccolor}{rgb}{0,0.5,0}
\definecolor{Mcolor}{rgb}{1,0,0}
\definecolor{lightgray}{rgb}{0.6,0.6,0.6}

\newcommand{\h}{\widehat}
\newcommand{\cone}{\operatorname{cone}}
\newcommand{\conv}{\operatorname{conv}}

\newcommand{\Gar}{\operatorname{Gar}}
\newcommand{\dep}{\textnormal{dp}}
\newcommand{\dom}{\operatorname{Dom}}
\newcommand{\mpair}[1]{\langle\, #1\,\rangle}
\newcommand{\supp}{\operatorname{supp}}

\newcommand{\mor}{\operatorname{mor}}
\newcommand{\ob}{\operatorname{ob}}

\author[C. Hohlweg]{Christophe~Hohlweg$^{\diamond}$}
\address[Christophe Hohlweg]{Universit\'e du Qu\'ebec \`a Montr\'eal\\
LaCIM et D\'epartement de Math\'ematiques\\ CP 8888 Succ. Centre-Ville\\
Montr\'eal, Qu\'ebec, H3C 3P8\\ Canada}
\email{hohlweg.christophe@uqam.ca}
\urladdr{http://hohlweg.math.uqam.ca}
\thanks{$^\diamond$supported by NSERC Discovery grant {\em Coxeter groups and related structures}. A significant part of this work was done while CH was on sabbatical leave at Institut de Recherche Math\'ematique Avanc\'ee (IRMA), Universit\'e de Strasbourg, from January 2013 to June 2014.}

\author[M.~Dyer]{Matthew Dyer}
\address{Department of Mathematics 
\\ 255 Hurley Building\\ University of Notre Dame \\
Notre Dame, Indiana 46556, U.S.A.}
\email{dyer.1@nd.edu}

\keywords{Root systems, Coxeter groups,  Garside families, inversion sets, low elements, weak order, Bruhat order, root poset, dominance order, small roots}
\subjclass[2010]{Primary 20F55; secondary 17B22; 05E15; 06F99}

\title[Small roots and low elements in Coxeter groups]{Small roots, low elements, and the weak order in Coxeter groups}

\begin{document}

\begin{abstract}   In this article we provide a  new finite class of elements in any Coxeter system $(W,S)$ called {\em low elements}. They are  defined from Brink and Howlett's small roots, which are strongly linked to the automatic structure of $(W,S)$. Our first main result is to show that they form a Garside shadow in $(W,S)$, i.e., they contain $S$ and are  closed under join (for the right weak order) and by taking suffixes. These low elements are the key to prove that all   finitely generated Artin-Tits groups have a {\em finite} Garside family. This result was announced in a note with P.~Dehornoy in {\em Comptes rendus math\'ematiques}~\cite{DDH14} in which the present article was referred to under the following working title: {\em Monotonicity of dominance-depth on root systems and applications}.

The proof is based on a fundamental property enjoyed by small roots and which is our second main result; the set of small root is {\em bipodal}.

For a natural number $n$, we define similarly $n$-low elements from $n$-small roots and conjecture that the set of $n$-small roots is bipodal, implying the set of $n$-low elements is a Garside shadow; we prove this conjecture for  affine Coxeter groups and Coxeter groups whose graph is labelled by $3$ and $\infty$. To prove the latter, we extend the root poset on positive roots to a {\em weak order on the root system} and  define a {\em Bruhat order on the root system}, and  study the paths in those orders in order to establish a criterion to prove bipodality involving only finite dihedral reflection subgroups.
\end{abstract}

\date{\today}


 \maketitle

\section{Introduction}

In this article, we introduce and investigate the notion of a Garside shadow in a Coxeter system $(W,S)$:  a {\em Garside shadow in $(W,S)$} is a  subset of $W$ that contains~$S$  and closed under join (taken in the right weak order) and suffix. For instance $W$ itself is a Garside shadow;  see~\S\ref{sse:Garside}. The notion of Garside shadow  is analogous to the notion of a Garside family in a monoid~\cite{DDM13,Gars15}. 

We prove the existence of a {\em finite} Garside shadow in every Coxeter group by introducing the notion of a low element in $(W,S)$ and proving that the finite family of low elements is a Garside shadow.

An element in $w\in W$ is {\em low} if its (left) inversion set $N(w)$ --- the set of positive roots that are sent to negative roots under $w^{-1}$ --- is the conic hull of some {\em small roots}; see \S\ref{se:Low} for a precise definition of these notions. Small roots were introduced by B.~Brink and B.~Howlett~\cite{BrHo93} in their work on the regularity of the language of reduced words in $W$, in which they show that the set of small roots is finite. We state now our first main result. 

\begin{thm}\label{thm:Main1} For any Coxeter system $(W,S)$ with $S$ finite, the set of low elements of $W$ is a Garside shadow in $(W,S)$. 
\end{thm}

Since the intersection of a family of Garside shadows is a Garside shadow (c.f. Proposition~\ref{prop:Gars}), there is a smallest Garside shadow in $(W,S)$  and the next  result follows by Theorem~\ref{thm:Main1}.

\begin{cor}\label{cor:Main1} For any Coxeter system $(W,S)$ with $S$ finite, the smallest Garside shadow in $(W,S)$ is finite.
\end{cor}

\smallskip

The initial motivation for the paper was a question by P.~Dehornoy in the context of the associated finitely generated Artin-Tits  group $G$: {\em is there a finite Garside family in $G$?} Finite Garside families are important  as they provide normal forms with nice properties and are potentially linked to the problem of decidability of the Word problem and the Conjugacy Problem in finitely generated  Artin-Tits braid group;  see~\cite[Questions~26~and~27]{Gars15} or \cite{DDM13} for more details.    In particular, they provide an affirmative answer to the problem of decidability of the conjugacy problem in the Artin-Tits monoid.  P.\,Dehornoy not only provided the initial translation of the problem in Artin-Tits braid groups into a problem about Coxeter groups, but he also provided many partial results, including examples in affine and right angled cases, which were essential for the work described here. Our Theorem~\ref{thm:Main1} (and Corollary~\ref{cor:Main1}) answers Dehornoy's question in the positive. Indeed, $\sigma$ is the canonical lifting of a Coxeter group $W$ into the associated Artin-Tits monoid $M$, then $A$ is a Garside shadow in $(W,S)$ if, and only if, $\sigma(A)$ is a Garside family in $M$; see \cite[\S3]{DDH14}. Therefore, the smallest Garside family is the copy in $M$ of the smallest Garside shadow in $(W,S)$. Theorem~\ref{thm:Main1}, for which only a sketch of the proof was given in~\cite{DDH14}, appears as~\cite[Theorem~1.2]{DDH14}. 

\smallskip

The proof that the set of low elements is finite and closed under join is given in Proposition~\ref{prop:Low}. The difficult part of the proof lies in the stability by taking suffixes of the set of low elements. In order to lift up this difficulty we proceed as follows in \S\ref{se:Suffix}.  On the one hand,  we provide for a prefix $w'$ of $w\in W$ a complete description involving the {\em Bruhat order}  of the rays of the cone over the inversion set of $w'$ in function of the rays of the inversion set of $w$; see \S\ref{sse:InvSuff}. On the other hand, we describe a fundamental property enjoyed by small roots and which is our second main result; the set of small root is {\em bipodal}, see Theorem~\ref{thm:Main3}. We conclude then the proof of Theorem~\ref{thm:Main1} in \S\ref{sse:ProofMain1}.
 
\smallskip
 
Along the way, in \S\ref{sse:Low}, we discuss the question of the existence of an infinite filtration of $W$ by finite Garside subsets of $W$ called the $n$-low elements, which would provide, using the dictionary in~\cite{DDH14}, an infinite filtration of {\em bounded Garside families in Artin-Tits groups}.  More precisely, the set of positive roots may be ranked by the {\em dominance order}, which was introduced by B.~Brink and B.~Howlett together with small roots in~\cite{BrHo93}:  for any positive root $\beta$ there is a nonnegative integer $n\in\mathbb N$ such that  $\beta$  strictly dominates exactly $n$ positive roots; we call $n$ the {\em dominance depth} of $\beta$. Define the {\em $n$-small roots} to be those of dominance depth at most $n$; see \S\ref{sse:Small}. With this definition, the small roots are precisely the $0$-small roots.  X.~Fu~\cite{Fu12} and the first author (unpublished, see~\cite{Ed09}) show that the set of $n$-small roots is always finite for any $n\in\mathbb N$. Extending the definition of low elements, we call {\em $n$-low elements} the elements in $W$ whose left inversion sets are the conic hull of some $n$-small roots. We prove in Proposition~\ref{prop:Low} that the set of $n$-low elements is finite and closed under join (taken in the right weak order). We conjecture that they are also stable  under suffix.

\begin{conject}\label{conj:1}   If  $n\in\mathbb N$, the set of $n$-low elements is a finite Garside shadow in~$(W,S)$. 
 \end{conject} 
  
 This conjecture would be true if, for instance, the set of $n$-low elements is bipodal (Conjecture~\ref{conj:3}): this is discussed in~\S\ref{se:BruhatRoot}. We show that Conjecture~\ref{conj:1} is true if $(W,S)$ is finite, is a dihedral Coxeter system, is an affine Coxeter system (Theorem~\ref{thm:AffineBip}), if $n=0$  (Theorem~\ref{thm:Main1}) or by the following last  main result.

\begin{thm}\label{thm:Main2} Let $n\in\mathbb N$ and assume $S$ finite. Suppose that  every entry of the Coxeter matrix of $(W,S)$ is either $1$, $2$, $3$ or $\infty$. Then   the set of $n$-low elements of $W$ is a finite  Garside shadow in $(W,S)$.
\end{thm}
 
This theorem is restated  as Theorem~\ref{thm:W3Infty} and \S\ref{se:BruhatRoot} is dedicated 
to its proof.  There,  we  extend the {\em root poset on positive roots} (see~\cite[Chapter 4]
{BjBr05}) to a {\em weak order on the root system} and  define a {\em Bruhat order on the root 
system}, and  study the paths in those orders in order to establish a criterion to prove 
bipodality involving only finite dihedral reflection subgroups of $(W,S)$  
(Proposition~\ref{prop:Increase} and Corollary~\ref{cor:Increase}). The proof 
resolves a conjecture raised in \cite{Ed09} that the set of  $n$-small roots is {\em balanced} 
for all $n$, which may be phrased as a monotonicity property of the dominance depth for 
positive roots in any maximal  dihedral reflection subgroup. We  prove that this monotonicity 
property fails in general for   finite dihedral reflection subgroups but is true in a stronger form 
for infinite dihedral reflection subgroups. The key part of the argument applies to  a more general family of length 
functions on positive roots, with  both  the  standard  depth on positive roots  and the dominance depth  on roots as special 
cases,  to show they are monotonic non-decreasing in the Bruhat order  on   roots. 

\section{Weak order and Garside shadows}\label{se:Garside}

Fix $(W,S)$ a Coxeter system with length function $\ell:W\to\mathbb N$.  The {\em rank} of $W$ is the cardinality of $S$. The {\em standard parabolic subgroup $W_I$} is the subgroup of $W$ generated by $I\subseteq S$. It is well-known that $(W_I,I)$ is itself a Coxeter system and that the length function $\ell_I:W_I\to\mathbb N$ is the restriction of $\ell$ to $W_I$.   Moreover, $W_I$ is finite if and only if it  contains   a {\em longest element}, which is then unique, denoted by $w_{\circ,I}$. In~\S\ref{sse:RefSub}, we discuss more general facts about reflection subgroups for which standard parabolic subgroups are an instance. We refer the reader to \cite{Hu90,BjBr05} for general definitions and properties of Coxeter groups.  

\subsection{Weak order and reduced words} We say that $s_1\dots s_k$ $(s_i\in S)$ is a {\em reduced word for $w\in W$} if $w=s_1\dots s_k$ and $k=\ell(w)$. For $u,v,w\in W$, we adopt the following terminology: 

\begin{itemize}
\item {\em $w=uv$ is reduced} if $\ell(w)=\ell(u)+\ell(v)$; 
\item {\em $u$ is a prefix of $w$} if a reduced word for $u$ is a prefix of a reduced word for~$w$;
\item {\em $v$ is a suffix of $w$} if a reduced word for $v$ is a suffix of a reduced word for~$w$.
\end{itemize}
Observe that if $w=uv$ is reduced then the concatenation of any reduced word for~$u$ with any reduced word for~$v$ is a reduced word for $w$; so in this case $u$ is a prefix of $w$ and $v$ is a suffix of $w$.

The {\em (right) weak order}  is the order on~$W$ defined by $u\leq_Rv$ if $u$ is a prefix of $v$. Since we only consider the right weak order in this article, we only use from now on the term {\em weak order}. The weak order gives a natural orientation of the Cayley graph of  $(W,S)$: we orient an edge $w\to ws$ if $w\in W$ and $s\in S$ such that $w\leq_Rws$.  Moreover, Bj\"orner~\cite[Theorem~8]{Bj84} shows that the poset $(W,\leq_R)$ is a complete meet semilattice: for any $A\subseteq W$, there exists an infimum $\bigwedge A\in W$, also called the {\em meet} of $A$, see~\cite[Chapter~3]{BjBr05} for more details.  A subset $X\subseteq W$ is  {\em bounded in $W$} if there is $g\in W$ such that $x\leq_Rg$ for any $x\in X$. Therefore any bounded subset $X\subseteq W$ admits a least upper bound $\bigvee X$ called {\em the join of $X$}: 
$$
\bigvee X=\bigwedge\{g\in W\,|\, x\leq_R  g,\, \forall x\in X\}.
$$ 
The example of the infinite dihedral group is illustrated in Figure~\ref{fig:InfiniteDi}.
When $W$ is finite, any element $w\in W$ is a prefix of the  longest element~$w_\circ$. So in this case $W$ itself is bounded and  $(W,\leq_R)$ turns out to be a complete ortholattice, see for instance~\cite[Corollary 3.2.2]{BjBr05}.

\subsection{Garside shadows}\label{sse:Garside}  

\smallskip
We are now able  to discuss  the notion of Garside shadows in a Coxeter system mentioned in the introduction. 

\begin{defi}\label{def:Garside} A {\em Garside shadow in $(W,S)$} is a subset $A$ of $W$ containing $S$ such that:
\begin{enumerate}[(i)]
\item $A$ is closed under join in the weak order: if $X\subseteq A$ is bounded, then $\bigvee X\in A$;
\item $A$ is closed under suffix: if $w\in A$, then any suffix of $w$ is also in $A$.
\end{enumerate}
\end{defi}

The definition of a Garside shadow extends naturally to the standard parabolic subgroup $W_I$ generated by $I\subseteq S$: a Garside shadow $A$ of $(W_I,I)$ is a subset of $W_I$ containing $I$ and verifying Conditions~(i)--(ii) above.   Moreover, since a bounded~$X$ is necessarily finite, Condition~(i) above is equivalent to the following condition:
\begin{enumerate}[(i')]
\item if $u,v\in A$ and $u\vee v$ exists then $u\vee v\in A$.
\end{enumerate}

\begin{prop}\label{prop:Gars}
\begin{enumerate}
\item  The intersection of a family of Garside shadows in $(W,S)$ is a Garside shadow in $(W,S)$. In particular, for any $X\subseteq W$ there is a smallest Garside shadow $\Gar_S(X)$ of $(W,S)$ that contains $X$.
\item If $I\subseteq S$ and $X\subseteq W_I$, then $\Gar_I(X)\subseteq \Gar_S(X)$.
\item If $I\subseteq S$ is spherical, that is $W_I$ is finite, then $\Gar_I(I)=W_I$ is the unique Garside shadow in $(W_I,I)$. 
\end{enumerate}
\end{prop}
\begin{proof} The two first statements follow easily from Definition~\ref{def:Garside}. For the third statement, observe first that, since the longest element  $w_{\circ,I}$ of $W_I$ is the unique maximal length element, $w_{\circ,I}$ is the unique element of $W_I$ that has all $s\in I$ as prefix. So $\bigvee I=w_{\circ, I}$ and therefore $w_{\circ,I}\in \Gar_I(I)$ by Condition~(i).  Moreover it is well-known that any element of $W_I$ is prefix of $w_{\circ,I}$. Hence $W_I\subseteq \Gar_I(I)$ by Condition~(ii), which concludes the proof.
\end{proof}

\begin{ex}\label{ex:Inf1} Let $W$ be the infinite dihedral group $\mathcal D_\infty$ generated by $S=\{s,t\}$ with Coxeter graph:
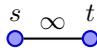
\begin{center}
\begin{tikzpicture}[sommet/.style={inner sep=2pt,circle,draw=blue!75!black,fill=blue!40,thick}]
	\node[sommet,label=above:$s$] (alpha) at (0,0) {};
	\node[sommet,label=above:$t$] (beta) at (1,0) {} edge[thick] node[auto,swap] {$\infty$} (alpha);
\end{tikzpicture}
\end{center}
Then $S$ is not bounded, see~Figure~\ref{fig:InfiniteDi}. So the join of $s$ and $t$ does not exist. Since the identity $e$ is a suffix of $s$ and $t$ we have $\tilde S=\Gar_S(S)=\{e,s,t\}$. 
\end{ex}

\begin{ex}\label{ex:Univ1} Let $W$ be the {\em universal Coxeter group} generated by $S$; its Coxeter graph is the complete graph whose edges are labelled by $\infty$. Then any subset $\{s,t\}\subseteq S$ of cardinality $2$  generates a standard dihedral parabolic subgroup as in Example~\ref{ex:Inf1} and is therefore not bounded. So $\tilde S=\Gar_S(S)=S\cup\{e\}$. 
\end{ex}

\begin{remark}\label{rem:SmallG}
\begin{enumerate}[(a)]
\item The smallest Garside shadow in $(W,S)$ is $\tilde S=\Gar_S(S)$. The existence of a finite Garside shadow in $(W,S)$ stated in Theorem~\ref{thm:Main1} implies that the smallest  Garside shadow $\tilde S$ is finite, which is Corollary~\ref{cor:Main1}. 

 \item We do not have a nice combinatorial characterization of the smallest Garside shadow in general. The main reason is the absence of a suitable combinatorial characterization of existence of the join in general. We give in Example~\ref{ex:AffineA2}  the construction of the smallest Garside shadow in the case of the affine Coxeter group of type $\tilde A_2$ by using a geometric characterization of the existence of the join. 

\item   Let $u,v\in W_I$, where $W_I$ is the standard parabolic subgroup generated by $I\subseteq S$. It is well-known that any reduced words for $u$ or $v$ have all their letters in $I$.  So any suffix of $u$ or $v$ is again an element of $W_I$.  Moreover, it is not difficult to see, with the help of Proposition~\ref{prop:Weak} below for instance, that the join $u\vee v$ is therefore an element of $W_I$.  We deduce then that if $B$ is a Garside set in $(W,S)$, then $B\cap W_I$ is a Garside set in $(W_I,I)$.

We do not know however if  the notion of Garside shadows closure is stable by restriction: let $I\subseteq S$ and $X\subseteq W_I$, is it true that $\Gar_I(X)=\Gar_S(X)\cap W_I ?$ 
 This question has an affirmative answer for standard parabolic subgroups of rank $2$. 
 
\end{enumerate}
\end{remark}

\subsection{Geometric  representation and root system}\label{sse:GeoRep} We recall here useful facts on root systems and reflection subgroups that will be needed to give an interpretation for the join in the  weak order in \S\ref{sse:GeoWeak}.

A Coxeter system can be seen as a discrete reflection subgroup in some quadratic space $(V,B)$, where $V$ is a real vector space endowed with a symmetric bilinear form $B$. The group of linear maps that preserves $B$ is denoted by $O_B(V)$. The {\em isotropic cone of $(V,B)$} is $Q=\{v\in V\,|\, B(v,v)=0\}$. To any non-isotropic vector $\alpha\in V\setminus Q$,  we associate the $B$-reflection $s_\alpha\in O_B(V)$ defined    by $s_\alpha(v)=v-2 \frac{B(\alpha,v)}{B(\alpha,\alpha)}\alpha$. 

Fix a  {\em geometric representation of $(W,S)$}, i.e., a faithful representation of $W$ as a subgroup of $O_B(V)$ such that $S$ is mapped into a set of $B$-reflections associated to a simple system  $\Delta=\{\alpha_s\,|\,s\in S\}$ ($s=s_{\alpha_s}$). 
Recall that a simple system in $(V,B)$ is a finite subset $\Delta$ in $V$ such that:
\begin{enumerate}[(i)]
 \item  $\Delta$ is positively linearly independent:  if $\sum_{\alpha\in \Delta} a_\alpha\alpha$ with $a_\alpha\geq 0$, then all $a_\alpha=0$;
\item for all $\alpha, \beta \in \Delta$ distinct, 
  $\displaystyle{B(\alpha,\beta) \in \ ]-\infty,-1] \cup
    \{-\cos\left(\frac{\pi}{k}\right), k\in \mathbb N_{\geq 2} \} }$;
\item for all $\alpha \in \Delta$, $B(\alpha,\alpha)=1$.
\end{enumerate}
 Note that, since $\Delta$ is positively linearly independant, the cone $\cone(\Delta)$ is pointed: $\cone(\Delta)\cap\cone(-\Delta)=\{0\}$ (here $\cone(A)$ is the set of non-negative linear combinations of vectors in $A$).  Note  also that if the order $m_{st}$ of $st$ is  finite, then $B(\alpha_s,\alpha_t)=-\cos\left(\frac{\pi}{m_{st}}\right)$ and that  $B(\alpha_s,\alpha_t)\leq -1$ if and only if the order of $st$ is infinite. 

Denote by  $\Phi=W(\Delta)$ the corresponding {\em root system} with base $\Delta$, which is partitioned into positive roots $\Phi^+=\cone(\Delta)\cap \Phi$ and negative roots $\Phi^-=-\Phi^+$. The pair $(\Phi,\Delta)$ is called a {\em based root system}. The {\em rank of $(\Phi,\Delta)$} is the rank of $(W,S)$, i.e., $|\Delta|=|S|$. The {\em classical geometric representation} is obtained by assuming that $\Delta$ is a basis of $V$ and that $B(\alpha_s,\alpha_t)=-1$ if the order of $st$ in $W$ is infinite. For more details on geometric representations, the reader may, for instance, consult~\cite[\S1]{HoLaRi14}.

\smallskip

There is a useful statistic on the positive root system $\Phi^+$ called {\em the depth} defined as follows:  
$$
\dep(\beta)=\min\{ \ell(w)\,|\, w\in W, w(\beta)\in \Phi^- \},\quad \beta\in \Phi^+.
$$
The next proposition is well-known, see for instance~\cite[Lemma~4.6.2]{BjBr05}.

\begin{prop}\label{prop:Dep} Let $s\in S$ and $\beta\in\Phi^+\setminus\{\alpha_s\}$.  We have $\dep(\alpha_s)=1$ and 
  $$
  \dep(s(\beta))=
   \left\{
  \begin{array}{cl}
  \dep(\beta)-1&\textrm{if }\ B(\alpha_s,\beta)> 0\\
  \dep(\beta)&\textrm{if }\ B(\alpha_s,\beta)=0\\
  \dep(\beta)+1&\textrm{if }\ B(\alpha_s,\beta)<0\\
  \end{array}
  \right.
  $$
  In particular, $\dep(\alpha)=1$ for any simple root $\alpha\in \Delta$.
\end{prop}

\begin{remark}\label{rem:Proj1}
In order to present examples of rank $|S|=2,3,4$ easily, it is useful to consider the  {\em projective geometric representation} of $W$, as in~\cite{HoLaRi14,DyHoRi13}. Since $\Phi=\Phi^+\sqcup \Phi^-$ is encoded by the set of positive roots $\Phi^+$, we represent $\Phi$ by an `affine cut' $\h\Phi$: there is an affine hyperplane $V_1$ in $V$ {\em transverse to $\Phi^+$}, i.e., for any $\beta\in \Phi^+$, the ray $\mathbb R^+\beta$ intersects $V_1$ in a unique nonzero point $\h\beta$. So $\mathbb R\beta\cap V_1=\{\h\beta\}$ for any $\beta\in\Phi$. The {\em set of normalized roots}  
$\h\Phi=\{\h\beta\,|\, \beta\in \Phi\}$ is contained in the compact  set $\conv(\h\Delta)$ and therefore admits a set $E$ of accumulation points called {\em the set of limit roots}. We have   $E=\varnothing$ if and only if $W$ is finite; $E$ is a singleton if  $(W,S)$ is affine and irreducible. Moreover, limit roots are in the isotropic cone $Q$ of $B$:
$$
E\subseteq \h Q=\{x\in V_1\,|\, B(x,x)=0\}.
$$
The group $W$ acts on $\h\Phi\sqcup E \cup \conv(E)$ componentwise: $w\cdot x=\widehat{w(x)}$.    We refer the reader to~\cite{HoLaRi14,DyHoRi13} for more details.
\end{remark}

\subsection{Weak order and inversion sets}\label{sse:GeoWeak} In this text, we make use of the interplay between reduced words  and their geometric counterparts:  inversion sets.   

\smallskip

The {\em (left) inversion set~$N(w)$ of $w\in W$} is defined by:
$$
N(w)=\Phi^+\cap w(\Phi^-)=\{\beta\in \Phi^{+}\, | \, \ell(s_{\beta}w)<\ell(w)\}. 
$$
Its cardinality is well-known to be $\ell(w)$.  Inversion sets allow a useful geometric interpretation of the weak order. For~$A\subseteq V$ we denote $\cone_\Phi (A)=\cone(A)\cap \Phi$ the set of roots in  $\cone(A)$.

\begin{prop}\label{prop:Weak}
\begin{enumerate} 
\item If $w=s_1\dots s_k$ is a reduced word for $w\in W$, then
$$
N(w)=\{\alpha_{s_1},s_1(\alpha_{s_2}),\cdots , s_1\dots s_{k-1}(\alpha_{s_k})\}.
$$

\item For $u,v\in W$ we have $u\leq_Rv$ if and only if $N(u)\subseteq N(v)$.
\item The map $N:(W,\leq_R)\to (\mathcal P(\Phi^+),\subseteq)$ is a poset monomorphism.
\item If $X\subseteq W$ is bounded, then 
$$
N\left(\bigvee X\right)=\cone_\Phi \left(\bigcup_{x\in X} N(x)\right).
$$ 
\end{enumerate}
\end{prop}
An illustration of this proposition  is given in Figure~\ref{fig:InfiniteDi}. The three first statements are classical and can be found in \cite[Chapter 3]{BjBr05} stated within the language of reflections instead of roots. The last statement follows from Dyer~\cite{Dy11}; a complete proof may also be found in~\cite[\S2.2-\S2.3]{HoLa15}.
\begin{figure}[h!]
\resizebox{.75\hsize}{!}{
\begin{tikzpicture}
	[scale=1,
	 pointille/.style={dashed},
	 axe/.style={color=black, very thick},
	 sommet/.style={inner sep=2pt,circle,draw=blue!75!black,fill=blue!40,thick,anchor=west}]
	 
\node[sommet]  (id)    [label=below:{\small{$e;\  \emptyset$}}]          at (0,0)    {};
\node[sommet]  (1)    [label=left:{\small{$\{\alpha_s\};\ s$}}]         at (-1,1)   {} edge[thick,<-] (id);
\node[sommet]  (2)    [label=right:{\small{$t;\ \{\alpha_t\}$}}]        at (1,1)    {} edge[thick,<-] (id);
\node[sommet]  (12)   [label=left:{\small{$\{\alpha_s, s(\alpha_t)\};\ st$}}]      at (-1,2)   {} edge[thick,<-] (1);
\node[sommet]  (21)   [label=right:{\small{$ts;\ \{\alpha_t, t(\alpha_s)\}$}}]     at (1,2)    {} edge[thick,<-] (2);
\node[sommet]  (121)  [label=left:{\small{$\{\alpha_s\}\sqcup s(N(ts));\ sts$}}]    at (-1,3)    {} edge[thick,<-] (12);
\node[sommet]  (212)  [label=right:{\small{$tst;\ \{\alpha_t\}\sqcup t(N(st))$}}]    at (1,3)    {} edge[thick,<-] (21);

\draw[pointille] (121) -- +(0,1);
\draw[pointille] (212) -- +(0,1);
\end{tikzpicture}}
\caption{The weak order on  the infinite dihedral group $\mathcal D_\infty$: the vertices are labelled by $w\in \mathcal D_\infty$ together with its inversion set $N(w)$.}
\label{fig:InfiniteDi}
\end{figure}
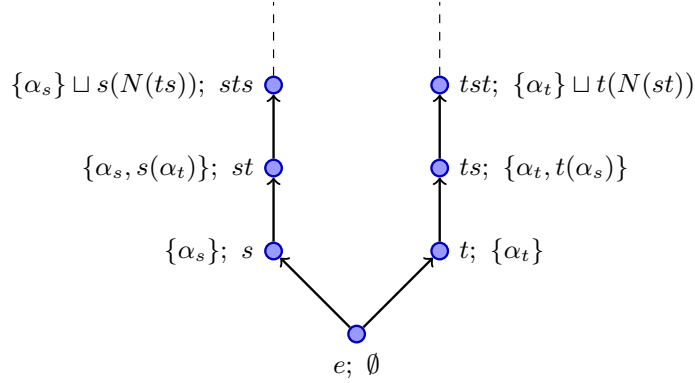
\begin{remark}\label{rem:Proj2}
Within the projective representation, see Remark~\ref{rem:Proj1}, conic closure is replaced by convex hull and $A\subseteq \Phi^+$ is replaced by $\h A \subseteq \h\Phi$. So Proposition~\ref{prop:Weak}(d) translates as follows (see Figure~\ref{fig:AffineA2} for an illustration):  let $X$ be a bounded subset of $W$ and write $\h N(x)$ instead of $\h{N(x)}$. Then the join $\bigvee X$ exists and 
$\h N(\bigvee X)= \conv_\Phi(\h N(X)).$ Moreover, in \cite[Theorem~3.2]{HoLa15}, the authors show that $X\subseteq W$ is bounded if and only if 
$$
\conv \left(\bigcup_{x\in X} \h N(x)\right)\cap \conv(E)=\emptyset.
$$
This is well illustrated in the case of the infinite dihedral group $\mathcal D_\infty$ as in Example~\ref{ex:Inf1}: in this case $E$ is always in the convex hull of $\h\Delta=\{\alpha_s,\alpha_t\}$ and therefore $S=\{s,t\}$ is not bounded. A more general example  of application is given  in Example~\ref{ex:AffineA2}  to justify the construction of the smallest Garside shadow in the case $\tilde A_2$. This point of view provided the intuition behind the definition of low elements. 
\end{remark}

As a direct consequence, we obtain  the following geometric interpretation of our terminology on reduced words.

\begin{cor}\label{cor:RedInv} For $u,v,w\in W$, we have:
\begin{itemize}
\item $w=uv$ is reduced if and only if  $N(w)=N(u)\sqcup u(N(v))$; 
\item $u$ is a prefix of $w$ if and only if $N(u)\subseteq N(w)$;
\item $v$ is a suffix of $w$ if and only if $N(v)=vw^{-1}(N(w)\setminus N(wv^{-1}))$.
\end{itemize}
In particular, $v$ is a maximal proper suffix of $w$ (i.e., $w=sv$ is reduced with $s\in S$) if and only if $N(v)=s(N(w)\setminus \{\alpha_s\})$. 
\end{cor}

\begin{ex}[The smallest Garside shadow in the affine group of type $\tilde A_2$]\label{ex:AffineA2} 

We show now how to obtain the smallest Garside shadow in the case~$\tilde A_2$. In this example, the difficulty is to show the existence or not of the join of a subset $X$ of $W$, i.e., to decide if $X$ is bounded or not in $W$, before taking suffixes. However, in the proof of Theorem~\ref{thm:Main1} the difficulty lies in the stability by taking suffixes.  
\smallskip

Let $(W,S)$ be the affine Coxeter system of type $\tilde A_2$ with Coxeter graph
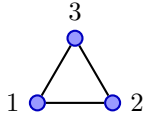
\begin{center}
 \begin{tikzpicture}[sommet/.style={inner sep=2pt,circle,draw=blue!75!black,fill=blue!40,thick},]
\coordinate (ancre) at (0,0);
\node[sommet,label=left:$1$] (alpha) at (ancre) {};
\node[sommet,label=right:$2$] (beta) at ($(ancre)+(1,0)$) {} edge[thick] (alpha);
\node[sommet,label=above:$3$] (gamma) at ($(ancre)+(0.5,0.86)$) {} edge[thick] (alpha) edge[thick] (beta);
\end{tikzpicture}
\end{center}
Our aim is to show that 
$$
 \tilde S= \Gar_S(S)=\{e,1,2,3,12,21,13,31,23,32,121,131,232,1232,2313,3121\}.
 $$
 We build $\tilde S$ recursively:
 \begin{description}
 \item[Initial step] $\tilde S_0=S\cup\{e\}$;
 \item[Inductive step]  If $m$ is in $\mathbb{N}$, then  $\tilde S_{m+1}$  is constituted of all the joins of bounded pairs of  $\tilde S_m$ together with their suffixes;
 \item[Final step] $\tilde S=\bigcup_{m\in\mathbb N} \tilde S_m$.  Actually, Corollary~\ref{cor:Main1} shows that this union is finite.
 \end{description}
Start with $\tilde S_0=\{e,1,2,3\}$. The proper maximal standard parabolic subgroups of $W$ are all finite of rank~$2$ with Coxeter graph the edges of the above graph. So the join of $I=\{1,2\}$ exists and $1\vee 2 = 121=212$, by Proposition~\ref{prop:Gars}. By taking all the suffixes we obtain therefore $\Gar_{\{1,2\}}(\{1,2\})=W_I=\{e,1,2,12,21,121\}$. Repeating the same argument for all other maximal standard parabolic subgroups we obtain:
$$
 \tilde S_1=\{e,1,2,3,12,21,13,31,23,32,121,131,232\}.
 $$
 Now to construct $\tilde S_2$ we consider join of pairs of elements in $W$ that are not contained in a proper standard parabolic subgroup. For instance $$21\vee 23=21.31=23.1 3$$ is the smallest word having both $21$ and $23$ as prefixes. We obtain similarly $$12\vee 13=12.32=13.23\ \textrm{ and }\ 31\vee 32=31.21=32.12.$$ We claim that 
$$
 \tilde S= \tilde S_2=\{e,1,2,3,12,21,13,31,23,32,121,131,232,1232,2313,3121\}.
 $$
 It is easy to check that $\tilde S_{2}$ is closed under suffixes.  
In order to prove that there is no other element of $W$ obtained by joining elements of $\tilde S_1$, we consider the geometric interpretation of the weak order as described  in~Remark~\ref{rem:Proj2}. Let $V$ be a real vector space of dimension $3$ with bilinear form $B$ and basis $\Delta=\{\alpha_1,\alpha_2,\alpha_3\}$ such that for $1\leq i,j\leq 3$ we have $B(\alpha_i,\alpha_j)=-\cos\left(\frac{\pi}{m_{i,j}}\right)$, where $m_{ij}$ is the order of the element $ij$ in $W$. So $\Phi=W(\Delta)$ is a root system with simple system $\Delta$ for $(W,S)$ of type $\tilde A_2$. In Figure~\ref{fig:AffineA2} one finds the first few normalized roots in blue and the unique limit root $\h\delta=\widehat{\alpha_1+\alpha_2+\alpha_3}$ is the red dot and then $E=\h Q=\{\h\delta\}$. From Remark~\ref{rem:Proj2}, we know that $u\vee v$ exists if and only if $\conv(\h N(u)\cup \h N(v))\cap E=\emptyset$, and that in this case $\h N(u\vee v)=\conv_\Phi(\h N(u)\cup \h N(v))$. For instance consider $N(31)=\{\alpha_3,\alpha_1+\alpha_3\}$ and $N(32)=\{\alpha_3,\alpha_2+\alpha_3\}$. Then
$$
\conv(\h N(31)\cup \h N(32))=\conv (\h\alpha_3,\h{\alpha_1+\alpha_3},\h{\alpha_2+\alpha_3})
$$
does not intersect $E$ and we have:
\begin{eqnarray*}
\h N(31\vee 32)&=&\conv_\Phi(\h N(31)\cup \h N(32))\\
&=&\{\h\alpha_3,\h{\alpha_1+\alpha_3},\h{\alpha_2+\alpha_3},\h{\beta}\}\\
&=&\h N(3121),\qquad \textrm{where }\beta=s_{\alpha_3}(\alpha_1+\alpha_2)=\alpha_1+\alpha_2+2\alpha_3.
\end{eqnarray*}
But $\conv(\h N(1)\cup \h N(32))=\conv (\h\alpha_1,\h\alpha_3,\h{\alpha_2+\alpha_3})$ 
does intersect $E$, since the segment $[\h\alpha_1,\h{\alpha_2+\alpha_3}]$ contains $E$. Therefore  $1\vee 32$ does not exist. This is this argument that shows our claim: $\tilde S=\tilde S_2$. Indeed, let $u,v\in \tilde S_1$   be such that  $\{u,v\}$  is not contained in a proper standard parabolic subgroup and is  different from  $\{31,32\}$, $\{12,13\}$ and $\{21,23\}$.    Without loss of generality, we consider $u=1u'$ reduced and $v=32v'$ reduced.  So $E\subseteq [\h\alpha_1,\h{\alpha_2+\alpha_3}]\subseteq \conv(\h N(u)\cup \h N(v))$ and therefore $u\vee v$ does not exists. So no more elements are added to $\tilde S_2$ than the ones we have already found.

\begin{figure}[h!]
\begin{tikzpicture}
	[scale=2,
	 q/.style={red,line join=round,thick},
	 racine/.style={blue},
	 racinesimple/.style={inner sep=2pt,circle,draw=blue!75!black,fill=blue!40,thick},
	 racinedih/.style={blue},
	 rotate=0]
	 
\def\grosseur{0.025}
\def\grosseursimple{0.05}

\def\grosseurdih{0.0075}

\fill[red] (2.00000000000000,1.15470053837925) circle (\grosseursimple);

\draw[blue,line width = 2pt] (2.00000000000000,3.46410161513775) -- (3.00000000000000,1.73205080756888) ;
\draw[red,line width = 2pt] (2.00000000000000,3.46410161513775) -- (1.00000000000000,1.73205080756888);

\fill[red, fill opacity=0.5] (2.00000000000000,3.46410161513775) -- (1.00000000000000,1.73205080756888) -- (3.00000000000000,1.73205080756888) -- cycle;

\node[racinesimple, label=left :{$\alpha_1$}] (a) at (0,0) {};
\node[racinesimple, label=right :{$\alpha_2$}] (b) at (4,0) {};
\node[racinesimple, label=above:{$\alpha_3$}] (g) at (2.00000000000000,3.46410161513775) {};
\draw[green!75!black] (a) -- (b) -- (g) -- (a);
\fill[racine] (3.00000000000000,1.73205080756888) circle (\grosseur);
\node[label=right :{$\alpha_2+\alpha_3$}] at (3.00000000000000,1.73205080756888) {};
\fill[racine] (2.00000000000000,0.01) circle (\grosseur);
\node[label=above :{$\alpha_1+\alpha_2$}] at (2.00000000000000,0.01) {};
\fill[racine] (1.00000000000000,1.73205080756888) circle (\grosseur);
\node[label=left :{$\alpha_1+\alpha_3$}] at (1.00000000000000,1.73205080756888) {};

\fill[racine] (2.50000000000000,0.866025403784439) circle (\grosseur);
\fill[racine] (1.50000000000000,0.866025403784439) circle (\grosseur);
\fill[racine] (2.00000000000000,1.73205080756888) circle (\grosseur);
\node[label=above :{$\beta$}] at (2.00000000000000,1.73205080756888) {};

\fill[racine] (2.00000000000000,0.692820323027551) circle (\grosseur);
\fill[racine] (2.40000000000000,1.38564064605510) circle (\grosseur);
\fill[racine] (1.60000000000000,1.38564064605510) circle (\grosseur);

\fill[racine] (1.71428571428571,0.989743318610787) circle (\grosseur);
\fill[racine] (2.28571428571429,0.989743318610787) circle (\grosseur);
\fill[racine] (2.00000000000000,1.48461497791618) circle (\grosseur);

\fill[racine] (1.75000000000000,1.29903810567666) circle (\grosseur);
\fill[racine] (2.25000000000000,1.29903810567666) circle (\grosseur);
\fill[racine] (2.00000000000000,0.866025403784439) circle (\grosseur);

\fill[racine] (2.20000000000000,1.03923048454133) circle (\grosseur);
\fill[racine] (1.80000000000000,1.03923048454133) circle (\grosseur);
\fill[racine] (2.00000000000000,1.38564064605510) circle (\grosseur);

\fill[racine] (1.81818181818182,1.25967331459555) circle (\grosseur);
\fill[racine] (2.18181818181818,1.25967331459555) circle (\grosseur);
\fill[racine] (2.00000000000000,0.944754985946660) circle (\grosseur);

\fill[racine] (2.00000000000000,1.33234677505298) circle (\grosseur);
\fill[racine] (1.84615384615385,1.06587742004239) circle (\grosseur);
\fill[racine] (2.15384615384615,1.06587742004239) circle (\grosseur);

\fill[racine] (1.85714285714286,1.23717914826348) circle (\grosseur);
\fill[racine] (2.00000000000000,0.989743318610787) circle (\grosseur);
\fill[racine] (2.14285714285714,1.23717914826348) circle (\grosseur);

\node at ($(2.00000000000000,2.43205080756888)$) {$N(3121)$};

\node at (1,2.5) {$N(31)$};
\node at (3,2.5) {$N(32)$};

\coordinate (ancre) at (-0.5,2.6);
\node[racinesimple,label=below left:{$1$}] (alpha) at (ancre) {};
\node[racinesimple,label=below right :{$2$}] (beta) at ($(ancre)+(0.5,0)$) {} edge[thick] (alpha);
\node[racinesimple,label=above:{$3$}] (gamma) at ($(ancre)+(0.25,0.43)$) {} edge[thick]  (alpha) edge[thick] (beta);
\coordinate (mat) at (4, 2.8);

\end{tikzpicture}
\caption{The inversion sets of $31$, $32$ and  $3121=31\vee 32$  pictured in the projective representation (see Remark~\ref{rem:Proj1} and Remark~\ref{rem:Proj2}) of the affine Coxeter group of type $\tilde A_2$. Each blue point represents a normalized root $\h\gamma$ which we simply label by the corresponding root $\gamma$. The red dot represents the unique limit root $\delta$.}
\label{fig:AffineA2}
\end{figure}
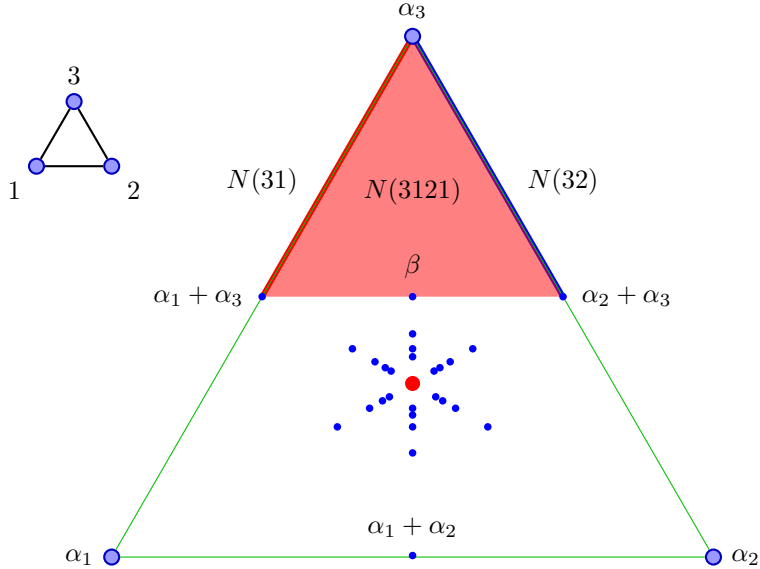
\end{ex}

\subsection{Reflection subgroups}\label{sse:RefSub}  We end this section by recalling some useful facts about reflection subgroups that will be needed to give an interpretation  of suffixes of a word in \S\ref{sse:InvSuff}.

A \emph{reflection subgroup} $W'$ of $W$ is a subgroup $W'=\langle s_\beta\,|\, \beta\in A\rangle$ generated by the reflections associated to the roots in some  $A\subseteq \Phi^+$. Write 
$$
\Phi_{W'}:=\{\beta\in \Phi\,|\, s_{\beta}\in W'\} \ \textrm{ and }\ \Delta_{W'}:=\{\alpha\in \Phi^{+}  \,|\, N(s_\alpha)\cap \Phi_{W'}=\{\alpha\}\}.
$$
Then the first author shows in~\cite{Dy90} that  $\Phi_{W'}$ is a  root system in $(V,B)$ with simple system $\Delta_{W'}$, called  the  {\em canonical simple system of $\Phi_{W'}$}.  Therefore the pair  $(W',S')$ is a Coxeter system, with {\em  canonical simple  reflections}  
$$
S'=\chi(W'):=\{s_{\alpha}\,|\, \alpha\in \Delta_{W'}\},
$$
and with corresponding positive roots  $\Phi_{W'}^{+}=\Phi_{W'}\cap \Phi^{+}  $;  see also~\cite{BoDy10} (both notions depend on $(W,S)$ and not just $W$).  

\begin{remark}  Other characterizations of the canonical simple system $\Delta_{W'}$ are that it is the unique inclusion-minimal subset $\Gamma$ of $\Phi^{+}_{W'}$  such that $\Phi^{+}_{W'}\subseteq \cone(\Gamma)$, and it is the  set consisting of all representatives in $\Phi$ of extreme rays of $\cone(\Phi^{+}_{W'})$ (see \cite{De89}).
 \end{remark}

Various notions attached above and below to the root system $(\Phi,\Delta)$ may be applied to the based root system $(\Phi_{W'},\Pi_{W'})$, and will be denoted by attaching $W'$ as decoration (usually, as subscript) on the corresponding notation used for $(\Phi,\Pi)$. For example, $\ell_{W'}: W'\to \mathbb N$ is the length function of $(W',S')$,
and $N_{W'}(u):=\Phi_{W'}^{+}\cap u(\Phi_{W'}^{-})=N(u)\cap \Phi_{W'}$ for $u\in W'$ where $\ell_{W'}(u)=| N_{W'}{u}|$.

\begin{ex}\label{ex:InfDiRef} Let $\alpha,\beta\in \Phi^+$ such that $B(\alpha,\beta)\leq -1$. Then $\alpha\not = \beta$ and $\Delta_{W'}=\{\alpha,\beta\}$ verifies the condition to be a simple system as in \S\ref{sse:GeoRep}. Let $W':=\langle s_{\alpha},s_{\beta}\rangle$ the  dihedral reflection subgroup associated to $\Delta'$ and note that $W'$ is infinite. 

A useful observation in the context of normalized roots, as in Remark~\ref{rem:Proj1}, is that for any roots $\alpha,\beta\in\Phi^+$, the dihedral reflection subgroup generated by $s_\alpha,s_\beta$ is finite if and only if the line $(\h\alpha,\h\beta)\cap Q=\varnothing$. Otherwise the line $(\h\alpha,\h\beta)$ intersects $Q$ in one or two points and contains an infinite number of normalized roots, see Figure~\ref{fig:Dihedral} for an illustration. We refer the reader to~\cite{HoLaRi14,DyHoRi13} for more details.
  
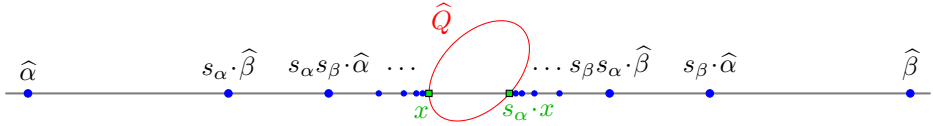
\begin{figure}[h!]
\newcommand{\ccdot}{\!\cdot \!}
\begin{tikzpicture}
	[scale=1,
	 pointille/.style={densely dashed},
	 root/.style={inner sep=1pt,circle,draw=blue,fill=blue,text=blue},
         sroot/.style={inner sep=0.7pt,circle,draw=blue,fill=blue},
         limroot/.style={inner sep=1.2pt,rectangle,draw=black,fill=green!90!black},
	 ]

\begin{scope} [rotate=45]
  
  \draw[red] (0,0) ellipse (0.8cm and 0.5cm);
\end{scope}

\draw[color=red] (-0.5,0.7) node{$\h Q$} ;
\draw[color=gray, thick] (-6.3,-0.3) -- (6,-0.3);

\node[root, label=above :{$\h\alpha$}] at ( -6.00000000000000 ,-0.3) {};
\node[root, label=above :{$s_\alpha \ccdot \h\beta$}] at (
-3.33500000000000 ,-0.3) {};
\node[root, label=above :{$s_\alpha s_\beta \ccdot \h\alpha$}] at (
-2.00250000000000 ,-0.3) {};
\node[sroot] at ( -1.33625000000000 ,-0.3) {};
\node[sroot, label={[shift={(0,0.15)}]$\dots$}] at ( -1.00312500000000 ,-0.3) {};
\node[sroot] at ( -0.836562500000000 ,-0.3) {};
\node[sroot] at ( -0.753281250000000 ,-0.3) {};

\node[root, label=above :{$\h\beta$}] at ( 5.73000000000000 ,-0.3) {};
\node[root, label=above :{$s_\beta \ccdot \h\alpha$}] at (
3.06500000000000 ,-0.3) {};
\node[root, label=above :{$s_\beta s_\alpha \ccdot \h\beta$}] at (
1.73250000000000 ,-0.3) {};
\node[sroot] at ( 1.06625000000000 ,-0.3) {};
\node[sroot, label={[shift={(0.2,0.15)}]$\dots$}] at ( 0.733125000000000 ,-0.3) {};
\node[sroot] at ( 0.566562500000000 ,-0.3) {};
\node[sroot] at ( 0.483281250000000 ,-0.3) {};

\node [limroot] at (-0.67,-0.3) {};
\node[green!70!black] at (-0.77,-0.55) {$x$};
\node[limroot] at (0.4,-0.3) {} ;
\node[green!70!black] at (0.65,-0.55) {$s_\alpha \ccdot x$};
\end{tikzpicture}

\caption{(\cite[Figure~7]{DyHoRi13}) Projective representation of a infinite dihedral reflection subgroup $W'$. In green are the two limit points of the root subsystem associated to $W'$ together with the action of $W'$.}
\label{fig:Dihedral}
\end{figure}
\end{ex}

\begin{ex}[Maximal dihedral reflection subgroups]\label{ex:MaxDi} A {\em maximal rank~$2$ root subsystem of $\Phi$} is a set~$\Phi'$ of the form $\Phi'=P\cap \Phi$ where $P$ is a plane in~$V$ intersecting $\Phi^+$ in at least two roots. The cone  spanned by~$\Phi' \cap\Phi^+$ has then a basis $\Delta'=\{\alpha,\beta\}$ of cardinality $2$  included in~$\Phi'\cap \Phi^+$, and then one has 
$$
\Phi'=P\cap \Phi^{+}=\cone(\Delta')\cap \Phi.  
$$
One can show that $(\Phi',\Delta')$ is a based root subsystem of $(\Phi,\Delta)$ of rank $2$. Write $S'=\{s_\alpha,s_\beta\}$, then the dihedral reflection subgroup $W'=\langle S'\rangle$ is a {\em maximal dihedral reflection subgroup}. 

Any dihedral reflection subgroup $W''$ of $W$ is contained in such a maximal dihedral reflection subgroup since it has a rank $2$ based root subsystem that is contained in such a maximal rank 2 root subsystem.

In the context of normalized roots, as in Remark~\ref{rem:Proj1}, planes correspond to lines so maximal dihedral reflection subgroups corresponds to lines passing through two normalized roots. If $P$ is such a line, then $\h \Phi' =P\cap \h\Phi$ is contained in a segment $[\h\alpha,\h\beta]$ and $\h\Delta'=\{\h\alpha,\h\beta\}$. For instance, in Figure~\ref{fig:AffineA2}, we see that $W'=\langle s_{\alpha_1+\alpha_3},s_{\alpha_2}\rangle$ is a (infinite) maximal dihedral reflection subgroup containing any dihedral reflection subgroup generated by two roots in the segment $[\h{\alpha_1+\alpha_2}, \h\alpha_3]$.
\end{ex}

We state now the following sufficient condition for a reflection subgroup to be finite; this result is used in the proof of Lemma~\ref{lem:SmallInv}.

\begin{prop}\label{prop:FiniteRefl} Let $w\in W$.
\begin{enumerate}
\item If $W'$ is a reflection subgroup of $W$ with $\Delta_{W'}\subseteq N(w)$, then $W'$ is finite.
\item If $\alpha,\beta\in N(w)$, then $B(\alpha,\beta)>-1$.
\end{enumerate}
\end{prop}
\begin{proof} To prove {\em (1)}, note that Proposition~\ref{prop:Weak}{\em (4)} implies in particular that $$\Phi^{+}_{W'}\subseteq \cone_\Phi(\Delta_{W'})\subseteq \cone_\Phi (N(w))=N(w).$$ Since $N(w)$ is finite, so is $\Phi^{+}_{W'}$. Therefore $\Phi_{W'}$ is finite. Since $W'$ acts faithfully as permutation group on $\Phi_{W'}$, $W'$ is finite.

 To prove {\em (2)}, assume for a contradiction that $\alpha,\beta\in \Phi$ with $B(\alpha,\beta)\leq -1$. Then $\alpha\not = \beta$ and $\Delta_{W'}=\{\alpha,\beta\}$ verifies the condition to be a simple system as in \S\ref{sse:GeoRep}. Let $W':=\langle s_{\alpha},s_{\beta}\rangle$ the  dihedral reflection subgroup associated to $\Delta'$ and note that $W'$ is infinite. But   $\Delta_{W'}\subseteq N(w)$ by assumption, contradicting {\em (1)}. 
\end{proof}
%

 
 We end this section by recalling the notion of {\em shortest coset representatives} of reflection subgroups.  Let~$W'$ be a reflection subgroup of $W$ with canonical generators $S'$ and based root subsystem $(\Phi_{W'},\Delta_{W'})$ and  set
 
\begin{equation*}
X_{W'}:=\{v\in W\,|\, \ell(s_\beta v)>\ell(v),\forall \beta\in \Phi_{W'}  \}=\{v\in W\,|\, N(v)\cap \Phi_{W'}=\emptyset \}.
\end{equation*}  
Any coset $W'w$ in $W$ contains an element of minimal length, which must be in~$X_{W'}$.
It is known (see \cite[3.3(ii),3.4]{Dy91}) that

\begin{prop} \label{prop:CosetRep} Let $W'$ be a reflection subgroup of $W$ and $w\in W$. Write $w=uv$ with $(u,v)\in  W'\times X_{W'}$. We have:
\begin{enumerate}
\item  $N(w)\cap \Phi_{W'}=N_{W'}(u)$;
\item $\ell(s_\beta w)<\ell(w)$ if and only if $\ell_{W'}(s_\beta u)<\ell_{W'}(u)$ for all $\beta\in \Phi_{W'}$.
\end{enumerate}
\end{prop}

 This second statement of this last proposition, which will  be used frequently in  \S\ref{se:Suffix} and \S\ref{se:BruhatRoot}, will be referred to  as   {\em ``functoriality of the Bruhat graph''}  (for inclusions of reflection subgroups; see \cite{Dy91,Dy92}). 
It implies that the elements of $X_{W'}$ are precisely  those elements $v$ of $W$ which are of minimum length in their coset $W'v$, and that any element $w\in W$ can be uniquely expressed in the form $w=uv$ where $u\in W'$ and $v\in X_{W'}$.   Accordingly,   $X_{W'}$ is   called the set of \emph{minimal length  left coset representatives} for $W'$ in $W$. Functoriality of the Bruhat graph also implies the following useful alternative characterization of~$X_{W'}$:   
$$X_{W'}=\{v\in W\,|\, \ell(s_\beta v)>\ell(v),\forall \beta\in \Delta_{W'} \}.$$

\section{Small roots and Low elements}\label{se:Low}

Let  $(W,S)$ be a finitely generated Coxeter system together with a root system~$\Phi$ and simple system $\Delta$ as in \S\ref{sse:GeoRep}.  The aim of the next two sections is to prove  Theorem~\ref{thm:Main1}:  the set of low elements  is a finite Garside shadow in $(W,S)$. In order to do this we first review one important partial order  on the set of positive roots $\Phi^+$ that lead to the definitions of small roots and low elements.

\subsection{Dominance order on roots}\label{sse:Domin}   

 Introduced by Brink and Howlett in~\cite{BrHo93}, the \emph{dominance order} is the partial order~$\preceq$ on $\Phi^+$ defined by:
$$
\alpha\preceq \beta \iff (\forall w\in W,\, \beta\in N(w)\implies \alpha\in N(w)).
$$
(we say in this case that {\em $\beta$ dominates~$\alpha$}). Related to the dominance order, there is a notion of dominance depth, called $\infty$-depth, of a positive root.

\begin{defi} Let $\beta\in \Phi^+$ and $n\in\mathbb N$.
\begin{enumerate}[(i)]  

\item The {\em dominance set of $\beta$ is} $\dom(\beta)=\{\alpha\in\Phi^+\,|\,\alpha\prec \beta\}$, the set of positive roots strictly dominated by $\beta$. 
 
 \item The {\em $\infty$-depth on $\Phi^+$} is  defined by 
$
\dep_\infty(\beta)=|\dom(\beta)|.
$

\end{enumerate}
\end{defi}

\begin{remark}
\begin{enumerate}[(a)]

\item In the definition of $\dom(\beta)$, the inequality  is strict: $\beta\notin \dom(\beta)$. 

\item The dominance order together with the $\infty$-depth $\dep_\infty(\beta)$ should not be confused with the {\em root poset~\cite[\S4.6]{BjBr05}} defined by using the usual depth of a root $\dep(\beta)$ from \S\ref{sse:GeoRep},  see~\S\ref{se:BruhatRoot} for further discussions on this subject.

\end{enumerate}
 \end{remark}

 The following result, which is analogous  to Proposition~\ref{prop:Dep}, gives a recurrence formula for $\infty$-depth of roots; see Proposition~\ref{1.6} for a common generalization.

\begin{prop}\label{prop:NSmall} Let $s\in S$ and $\beta\in\Phi^+$  with $\beta\neq \alpha_{s}$.  We have $\dep_\infty(\alpha_s)=0$ and 
 $$
  \dep_\infty(s(\beta))=
   \left\{
  \begin{array}{cl}
  \dep_\infty(\beta)-1&\textrm{if }\ B(\alpha_s,\beta)\geq 1\\
  \dep_\infty(\beta)&\textrm{if }\ B(\alpha_s,\beta)\in\,  ]-1,1[ \\
  \dep_\infty(\beta)+1&\textrm{if }\ B(\alpha_s,\beta)\leq  -1.\\
  \end{array}
  \right.
  $$
  In particular, $\dep_\infty(\alpha) =0$ for any simple root $\alpha\in \Delta$.
\end{prop}

\begin{proof} 

 The first statement follows from definition: if $\alpha_s$, which is element of $N(s)$,  dominates a positive root $\beta$, then $\beta\in N(s)=\{\alpha_s\}$. The second statement is a restatement of~\cite[Proposition~3.14]{Fu12}.
\end{proof}

The following lemma recollects useful properties of dominance order regarding restriction to root subsystem.

\begin{lem}\label{lem:Domin} 
\begin{enumerate} 
\item For any reflection subroup $W'$ of $W$, dominance order $\prec_{W'}$ on $\Phi^+_{W'}$ is the restriction to $\Phi_{W'}$ of the dominance order $\prec$ on $\Phi$.
\item If $I\subseteq S$, then the $\infty$ depth on $\Phi_I^+$ is the restriction of the $\infty$-depth on $\Phi^+$ to $\Phi_I^+$.
\end{enumerate}
\end{lem}
\begin{proof} Part (1) appears in \cite{HRT97}; see  \cite[Corollary 3.3]{Fu13}{\em (2)} for a more detailed proof. Let us now prove {\em (2)}: since any reduced word for $w\in W_I$ has its letters in $I$, we have $N_I(w)=N(w)$ for any $w\in W_I$. So  by definition of dominance  for $\beta\in\Phi_I^+$ and $\alpha\in \Phi^+$: $\alpha\preceq \beta\implies \alpha\in N(s_\beta)\subseteq \Phi^+_I$. So $\dom(\beta)\subseteq \Phi_I^+$  if $\beta\in\Phi_I^+$ and therefore the $\infty$-depth on $\Phi^+_I$ is the restriction of the $\infty$-depth on $\Phi^+$ to $\Phi_I^+$. 
\end{proof}

\begin{remark}\label{rem:Proj3} The $\infty$-depth has a nice geometric interpretation in the context of normalized roots (see~Remarks~\ref{rem:Proj1}~and~\ref{rem:Proj2}). Following~\cite{DyHoRi13} we say that $\h\beta\in \h \Phi$ is visible from $\h\alpha\in\h\Phi$  looking at $\h Q$, where $\h\alpha\neq \h\beta$,  if the segment  $[\h\alpha,\h\beta]$  has empty intersection with $Q$  and if the half-line $[\h\alpha,\beta)$ starting at $\h\alpha$ and passing through $\h \beta$ intersects $\h Q$. Then   $\alpha\prec \beta$ if and only if~$\h\beta$ is visible from $\h\alpha$ looking at $\h Q$; see \cite[Proposition~5.7]{DyHoRi13}. For instance, in Figure~\ref{fig:AffineA2}, $\beta$ is visible from $\alpha_3$ looking at the point $\h Q$ (in red). For a normalized root $\h\beta\in\h\Phi$, define the {\em blind cone} of $\h\beta$ to be the cone $Bl(\beta)$ pointed in $\h\beta$ and constituted of the points $a$ such that the line $(a,\h\beta)$ cuts $\h Q$ and such that the half-line $[\h\beta,a)$ starting at $\h\beta$ and passing through $a$ does not intersect $\h Q$. So $$\dom(\beta)=\{\alpha\in\Phi^+\,|\, \h\alpha\in Bl(\beta)\setminus \{\h\beta\}\}.$$ In particular $\dep_\infty(\beta)$ is the number of normalized roots in the blind cone of $\beta$ without counting $\beta$:
$$
\dep_\infty(\beta)=|Bl(\beta)\cap \h\Phi\setminus\{\h\beta\}|=|Bl(\h\beta)\cap\h\Phi|-1
$$
An example of the blind cone of a positive root $\beta$ such that $\dep_\infty(\beta)=2$ is given in Figure~\ref{fig:Ex237}; note that in Figure~\ref{fig:AffineG2}  the blind cone starting at the root of $\infty$-depth $2$ is a half-line passing through $\gamma$. 
 \end{remark}

Using Remark~\ref{rem:Proj3}, one gives the following examples. 

\begin{ex}\label{ex:Finite} If $W$ is finite, then $\dep_\infty(\alpha)=0$ for all $\alpha\in\Phi^+$. In particular any positive root dominates only itself.
\end{ex}

\begin{ex}\label{ex:Inf2} Assume, as in Example~\ref{ex:Inf1}, that $W$ is the infinite dihedral group generated by $S=\{s,t\}$. Choose a geometric representation with $B(\alpha_s,\alpha_t)\leq -1$, the associated projective representation is illustrated in Figure~\ref{fig:Dihedral}. Denote by $[s,t]_k$ the reduced word $st\dots$ with $k$ letters; this word ends with a $s$ if $k$ is odd and with a $t$ if $k$ is even. Denote for $k\in\mathbb N$:
$$
\alpha_{s,k} =\left\{\begin{array}{ll}%
 [s,t]_k (\alpha_s)&\textrm{if $k$ is even}.\\%
 {}[s,t]_k (\alpha_t)&\textrm{if $k$ is odd}.%
\end{array}
\right.
$$ 
In particular $\alpha_{s,0}=\alpha_s$ and the $\h\alpha_{s,k}$ are the roots on the left-hand side of $\h Q$ in Figure~\ref{fig:Dihedral}. We have therefore 
\begin{eqnarray*}
\dom(\alpha_{s,k})&=&\{\alpha\in\Phi^+\,|\, \h\alpha\in Bl(\alpha_{s,k})\setminus \{\h\beta\}\}\\
&=&\{\alpha_{s,j}\,|\, 0\leq j <  k\}
\end{eqnarray*}
Hence $\dep_\infty(\alpha_{s,k})=k$. By symmetry we can define $\alpha_{t,k}$ using the words $[t,s]_k$. In particular the unique roots of $\infty$-depth equals to $0$ are the simple roots $\alpha_s$ and $\alpha_t$.
\end{ex}

\begin{ex}\label{ex:AffineG2} In Figure~\ref{fig:AffineG2}, we give the $\infty$-depth on the first roots for the affine Coxeter system of type $\tilde G_2$. 
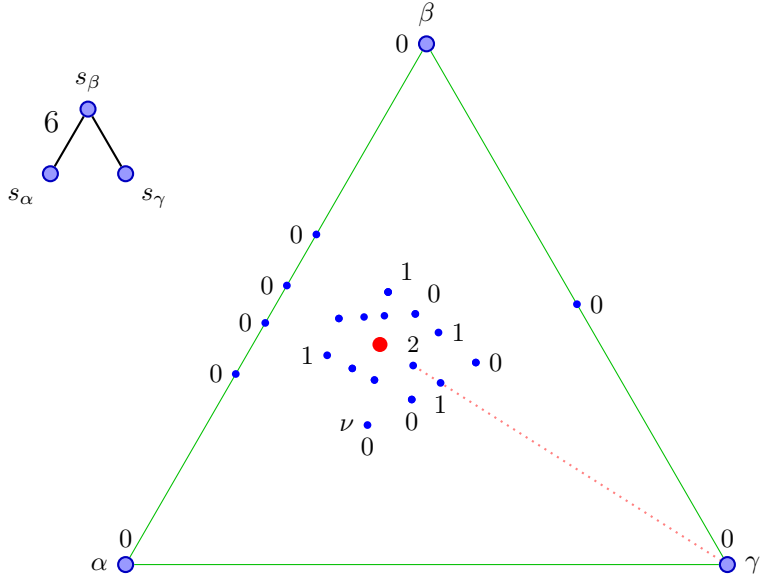
\begin{figure}[!h]
\scalebox{1}{
\begin{tikzpicture}
	[scale=2,
	 q/.style={red,line join=round,thick},
	 racine/.style={blue},
	 racinesimple/.style={inner sep=2pt,circle,draw=blue!75!black,fill=blue!40,thick},
	 racines/.style={inner sep=0.75pt,circle,draw=blue,fill=blue,thick},
	 rotate=0]
	 
\def\grosseur{0.025}
\def\grosseursimple{0.05}

\draw[dotted,thick,red!50] (1.91129642424986,1.32418500605019) -- (b);
\fill[red] (1.69059892324150,1.46410161513775) circle (\grosseursimple);

\node[racinesimple, label=left :{$\alpha$}, label=above:{$0$}] (a) at (0,0) {};
\node[racinesimple, label=right :{$\gamma$}, label=above:{$0$}] (b) at (4,0) {};
\node[racinesimple, label=above:{$\beta$}, label=left:{$0$}] (g) at (2.00000000000000,3.46410161513775) {};
\draw[green!75!black] (a) -- (b) -- (g) -- (a);

\node[racines, label=right :{$0$}] at (3.00000000000000,1.73205080756888) {};
\node[racines,label=left :{$0$}] at (1.26794919243112,2.19615242270663) {};
\node[racines,label=left :{$0$}] at (0.732050807568877,1.26794919243112)  {};

\node[racines, label=right :{$0$}] at (2.32797228678061,1.34405542643878) {};
\node[racines, label=below :{$0$}, label=left :{$\nu$}]  at (1.60769515458674,0.928203230275509) {};
\node[racines, label=left :{$0$}]  at (0.928203230275509,1.60769515458674) {};
\node[racines, label=left :{$0$}]  at (1.07179676972449,1.85640646055102) {};

\node[racines, label=below :{$0$}]  at (1.90192378864668,1.09807621135332) {};
\node[racines, label=north east :{$1$}]  at (1.74457630187009,1.81301687551177) {};
\node[racines, label=below :{$1$}]  at  (2.09349156224411,1.20867791700785) {};

\node[racines, label=north east :{$0$}]  at  (1.92552987299556,1.66755778575998) {};
\node[racines, label=right :{$1$}]  at  (2.07960403154708,1.54370564668483) {};
\node[racines, label=left :{$1$}]  at  (1.33974596215561,1.39230484541326) {};

\fill[racine] (1.50676128707236,1.30489355204360) circle (\grosseur);
\fill[racine] (1.65404441446703,1.22780955592818) circle (\grosseur);
\fill[racine] (1.41775235525746,1.63707940790424) circle (\grosseur);

\fill[racine] (1.50676128707236,1.30489355204360) circle (\grosseur);
\fill[racine] (1.92552987299556,1.66755778575998) circle (\grosseur);
\fill[racine] (1.74457630187009,1.81301687551177) circle (\grosseur);
\fill[racine] (1.58493649053890,1.64711431702997) circle (\grosseur);
\fill[racine] (1.72016678182488,1.65523125756274) circle (\grosseur);
\fill[racine] (1.90192378864668,1.09807621135332) circle (\grosseur);
\fill[racine] (2.32797228678061,1.34405542643877) circle (\grosseur);
\fill[racine] (1.65404441446703,1.22780955592818) circle (\grosseur);
\node[racines, label=above:{\small $2$}]  at  (1.91129642424986,1.32418500605019) {};
\fill[racine] (1.74457630187009,1.81301687551177) circle (\grosseur);
\fill[racine] (1.41775235525746,1.63707940790424) circle (\grosseur);

\coordinate (ancre) at (-0.5,2.6);
\node[racinesimple,label=below left:{$s_\alpha$}] (alpha) at (ancre) {};
\node[racinesimple,label=below right :{$s_\gamma$}] (beta) at ($(ancre)+(0.5,0)$) {};
\node[racinesimple,label=above:{$s_\beta$}] (gamma) at ($(ancre)+(0.25,0.43)$) {} edge[thick] node[auto,swap] {{\Large 6}} (alpha) edge[thick] (beta);
\coordinate (mat) at (4, 2.8);

\end{tikzpicture}
}
\caption{ The normalized isotropic cone $\h{Q}=\{\h\delta\}$ in red, the small roots (with $\infty$-depth~$0$)  and the $\infty$-depth on some other roots  for the normalized root system of
affine type $\widetilde{G_2}$. The  dotted red line forms the blind cone (c.f. Remark~\ref{rem:Proj3}) of the root of depth $2$ it emanates from. }
\label{fig:AffineG2}
\end{figure}
\end{ex}

\begin{ex}[Affine Weyl groups]\label{ex:Affine} We now give the example of an affine Coxeter system $(W,S)$, i.e., $W$ is an affine Weyl groups. 

When working with an affine Weyl group, one usually uses its {\em crystallographic root system}, which we recall now. Let $\Psi_{0}$ be a reduced, irreducible, crystallographic root system of a finite Weyl group $W_{0}$, in a real vector space $V_0$, with $W_{0}$-invariant positive definite scalar product $\mpair{ -|-}$ (see \cite{Bo68,Hu90}). Choose a positive system $\Psi^{+}_{0}$ with simple roots $\Delta_{0}$. Form a new vector space $V$ with $V_0$ as a codimension $1$ subspace, say $V=V_0\oplus \mathbb R\delta$ and extend $\mpair{-|-}$ to a symmetric invariant bilinear form $B$ on $V$ with radical $\mathbb R\delta$.  In particular the isotropic cone is $Q=\mathbb R\delta$.  Set 
$$
\Psi:=\{\mu+k\delta\,|\, \mu\in \Psi_{0},k\in\mathbb Z\}, \ \Psi^{+}=\{\mu+k\delta\in \Psi\,|\, \mu\in \Psi_{0}, k\geq 1 \textrm{ if $\mu\in \Psi_{0}^{-}$}\}
$$
and $\Pi=\Delta_{0}\cup\{\delta-\omega\}$ where $\omega$ is the highest root of $\Psi_{0}$. It is well known that $\Psi$ is a standard crystallographic root system of the affine Weyl group $W$ corresponding to $W_{0}$, with $\Psi^{+}$ as positive system and $\Delta$ as corresponding simple roots (see e.g. \cite{Kac90} or \cite{DyLe11}). For instance:
\begin{itemize}
\item in Example~\ref{ex:AffineA2}, $W$ is the affine Weyl group of type $\tilde A_2$: $W_0=S_3$, which is a finite Weyl group of type $A_2$, with  $\Delta_0=\{\alpha_1,\alpha_2\}$. Then the highest root is $\alpha_1+\alpha_2$ and $\alpha_3= \delta-\alpha_1-\alpha_2$.

\item in Example~\ref{ex:AffineG2}, $W$ is the affine Weyl group of type $\tilde G_2$: $W_0=\mathcal D_6$, which is a finite Weyl group of type $G_2$, with  $\Delta_0=\{\alpha,\beta\}$ and we have two choices for a finite crystallographic root system $\Psi_0$, which we give us the same based root system.  Choose $\Psi_0^+=\{\alpha,\alpha+\beta,2\alpha+3\beta,\alpha+2\beta,\alpha+3\beta,\beta\}$; the highest root is $2\alpha+3\beta$ and $\gamma= \delta-2\alpha-3\beta$.

\end{itemize}

The crystallographic root system   $\Psi$  is easily converted to the root system of a based root system $(\Phi,\Delta)$ in $(V,B)$   by normalizing roots to have square length $1$: since the elements of $\Psi$ are non-isotropic, i.e. $\Psi\cap \mathbb R\delta=\emptyset$,  we define $\alpha':=\frac{1}{B(\alpha,\alpha)^{1/2}}\alpha$ for any $\alpha\in \Psi$; then set 
$$
\Phi:=\{\alpha'\,|\, \alpha\in \Psi\},\quad \Delta:=\{\alpha'\,|\, \alpha\in\Pi\}\ \textrm{ and }\Phi^{+}=\cone_\Phi(\Delta)=\{\beta'\,|\,\beta\in \Psi^+\}.
$$
Note that the reflection $s_u=s_{ku}$ for any $k\in\mathbb R^{*}$ and $u\in V\setminus \mathbb R\delta$ so $s_{\alpha'}=s_{\alpha}\in W$ for $\alpha'\in \Phi$, i.e., $\alpha\in \Psi$. So we may  extend the notion of $\infty$-depth and dominance to the crystallographic root system $\Psi$ by  setting $\alpha \preceq  \beta$ if $\alpha' \preceq \beta'$ and 
 $$
 \dep_\infty(\alpha):=\dep_{\infty}(\alpha'),\  \textrm{ for any } \alpha,\beta\in \Psi.
 $$ 
It is well known and easily seen that the positive root systems of {\em infinite} maximal dihedral reflection subgroups of $W$ (see Example~\ref{ex:MaxDi}) are precisely the  sets 
$$
\{\mu+k\delta, -\mu+(k+1)\delta\,|\, \mu\in \Psi_0^{+}, k\in \mathbb N\}.
$$
 Those correspond to all the  sets of normalized roots on  segments in Figure~\ref{fig:AffineA2} and Figure~\ref{fig:AffineG2} passing through the roots on the face corresponding to the roots in $\Psi_0$ and the red dot $\delta$. It is clear by the definition that $\Psi^+$ is the union of all positive root systems of {\em infinite} maximal dihedral reflection subgroups of $W$.  The only dominances in $\Psi$ are thus:
 $$
 \mu+k\delta\preceq \mu+l\delta, \quad \textrm{for $k\leq l$ in $\mathbb Z$ and $\mu\in \Psi_0$}.
 $$
 It follows that for $\beta=\mu+k\delta\in\Psi^+$ we have:
$$
(\Diamond)\qquad \dep_{\infty}(\beta)=\dep_{\infty}(\mu+k\delta)=
\left\{
\begin{array}{ll}
k&\textrm{if $\mu\in \Psi_0^{+}$ and $k\in \mathbb N$}\\
k-1&\textrm{if $\mu \in \Psi_0^{-}$ and $k\in \mathbb N^*$}
\end{array}
\right. 
$$
 where $\mathbb N^{*}=\mathbb{N}\setminus\{0\}$.
In other words, on the segments in Figure~\ref{fig:AffineA2} and Figure~\ref{fig:AffineG2} representing infinite maximal dihedral reflection subgroups, $\dep_\infty$ is strictly increasing from $0$ to $\infty$ starting from a root $\mu\in \Psi_0^+$ to $\delta$ (the red dot), then decreasing from $\delta$ to  $-\mu+\delta$.

\end{ex}
\begin{ex} In Figure~\ref{fig:Ex237}, we give the $\infty$-depth on the first roots for the affine Coxeter system of rank $3$ whose graph is in the top right corner of the figure. 
\begin{figure}[h!]
\scalebox{0.8}{
\begin{tikzpicture}
	[scale=2,
	 q/.style={red,thin,line join=round},
	 racine/.style={blue},
	 racinesimple/.style={inner sep=2pt,circle,draw=blue!75!black,fill=blue!40,thick},
	 racines/.style={inner sep=0.75pt,circle,draw=blue,fill=blue,thick},
	 rotate=0]
	 
\def\grosseur{0.025}
\def\grosseursimple{0.05}

\coordinate (O) at (1.26495775914699,1.35583498282519);
\shade [shading=axis,bottom color=red!75!black!50,top color=white,shading angle=90] (O) -- (0.87,1.45) arc (4.4:0:10) -- (O);
\draw[red!50] (0.535,0) -- (1.41,1.63) {};
\draw[red!50] (0.87,1.45) -- (1.6,1.27)  {};

\draw[dotted,thick,red] (0.713791735784419,1.23632355240139) -- (1.54054077082828,1.40973726890995) ;
\draw[dotted,thick,red] (1.51794931097639,1.62700197939023) -- (0,0);


\draw[q] (1.95,1.48) -- (1.94,1.48) -- (1.95,1.49) -- (1.94,1.49) -- (1.94,1.5) -- (1.93,1.5) -- (1.94,1.51) --(1.93,1.51) -- (1.93,1.52) --(1.92,1.52) --
(1.93,1.53) --(1.92,1.53) -- (1.92,1.54) --(1.91,1.54) -- (1.91,1.55) --(1.9,1.55) -- (1.91,1.56) --(1.9,1.56) --
(1.9,1.57) --(1.89,1.57) -- (1.89,1.58) --(1.88,1.58) -- (1.88,1.59) --(1.87,1.59) -- (1.87,1.6) --(1.86,1.6) --
(1.85,1.6) -- (1.85,1.61) -- (1.84,1.61) -- (1.84,1.62) -- (1.83,1.62) -- (1.82,1.62) -- (1.82,1.63) -- (1.81,1.63) --
(1.81,1.64) -- (1.8,1.64) -- (1.79,1.64) -- (1.78,1.65) -- (1.77,1.65) -- (1.76,1.65) -- (1.76,1.66) -- (1.75,1.66) -- (1.74,1.66) -- (1.73,1.66) --
(1.73,1.67) -- (1.72,1.67) -- (1.71,1.67) -- (1.7,1.67) -- (1.7,1.68) -- (1.69,1.67) -- (1.69,1.68) -- (1.68,1.68) -- (1.67,1.68) -- (1.65,1.68) --
(1.64,1.68) -- (1.63,1.68) -- (1.62,1.69) -- (1.62,1.68) -- (1.61,1.69) -- (1.61,1.68) -- (1.6,1.69) -- (1.6,1.68) -- (1.59,1.68) -- (1.58,1.68) -- 
(1.57,1.68) -- (1.56,1.68) -- (1.55,1.68) -- (1.54,1.68) -- (1.53,1.68) -- (1.53,1.67) -- (1.52,1.67) -- (1.51,1.67) -- (1.5,1.67) -- (1.5,1.66) --
(1.49,1.67) -- (1.49,1.66) -- (1.48,1.66) -- (1.47,1.66) -- (1.47,1.65) --(1.46,1.65) -- (1.45,1.65) -- (1.45,1.64) -- (1.44,1.64) -- (1.44,1.63) -- 
(1.43,1.63) -- (1.42,1.63) -- (1.43,1.62) -- (1.42,1.62) -- (1.41,1.62) -- (1.41,1.61) -- (1.4,1.61) -- (1.41,1.6) -- (1.4,1.6) -- (1.39,1.6) -- (1.4,1.59) --
(1.39,1.59) -- (1.39,1.58) -- (1.38,1.58) -- (1.39,1.57) -- (1.38,1.57) -- (1.38,1.56) -- (1.37,1.56) -- (1.38,1.55) -- (1.37,1.55) -- (1.37,1.54) -- 
(1.37,1.53) -- (1.36,1.53) -- (1.37,1.52) -- (1.36,1.52) -- (1.37,1.51) -- (1.36,1.51) -- (1.37,1.5) -- (1.36,1.5) -- (1.37,1.49) -- (1.36,1.49) -- 
(1.37,1.48) -- (1.37,1.47) -- (1.38,1.47) -- (1.37,1.46) -- (1.38,1.46) -- (1.38,1.45) -- (1.38,1.44) -- (1.39,1.44) -- (1.38,1.43) -- (1.39,1.43) --
(1.39,1.42) -- (1.4,1.42) -- (1.4,1.41) -- (1.4,1.4) -- (1.41,1.4) -- (1.41,1.39) -- (1.42,1.39) -- (1.42,1.38) -- (1.43,1.38) -- (1.43,1.37) --
(1.44,1.37) -- (1.44,1.36) -- (1.45,1.36) -- (1.45,1.35) -- (1.46,1.35) -- (1.46,1.34) -- (1.47,1.34) -- (1.48,1.34) -- (1.48,1.33) --
(1.49,1.33) -- (1.49,1.32) -- (1.5,1.32) -- (1.51,1.32) -- (1.51,1.31) -- (1.52,1.31) -- (1.53,1.31) -- (1.53,1.3) -- (1.54,1.3) -- (1.55,1.3) -- 
(1.55,1.29) -- (1.56,1.29) -- (1.57,1.29) -- (1.58,1.29) -- (1.58,1.28) -- (1.59,1.28) -- (1.6,1.28) -- (1.61,1.28) -- (1.62,1.28) -- (1.62,1.27) -- 
(1.63,1.27) -- (1.64,1.27) -- (1.65,1.27) -- (1.67,1.27) -- (1.68,1.26) --(1.68,1.27) --(1.69,1.26) -- (1.69,1.27) -- (1.7,1.26) -- (1.7,1.27) --
(1.71,1.26) -- (1.71,1.27) -- (1.72,1.26) -- (1.72,1.27) -- (1.73,1.26) -- (1.73,1.27) -- (1.74,1.26) -- (1.74,1.27) -- (1.75,1.27) -- (1.76,1.27) --
(1.77,1.27) -- (1.78,1.27) -- (1.79,1.27) -- (1.79,1.28) -- (1.8,1.27) -- (1.8,1.28) -- (1.81,1.28) -- (1.82,1.28) -- (1.82,1.29) -- (1.83,1.28) -- 
(1.83,1.29) -- (1.84,1.29) -- (1.85,1.29) -- (1.85,1.3) -- (1.86,1.3) -- (1.87,1.3) -- (1.87,1.31) -- (1.88,1.31) -- (1.88,1.32) -- (1.89,1.32) --
(1.9,1.32) -- (1.89,1.33) -- (1.9,1.33) -- (1.91,1.33) -- (1.9,1.34) -- (1.91,1.34) -- (1.92,1.34) -- (1.91,1.35) -- (1.92,1.35) -- (1.92,1.36) --
(1.93,1.36) -- (1.93,1.37) -- (1.94,1.37) -- (1.93,1.38) -- (1.94,1.38) -- (1.94,1.39) -- (1.95,1.39) -- (1.94,1.4) -- (1.95,1.4) -- (1.94,1.41) -- 
(1.95,1.41) -- (1.94,1.42) -- (1.95,1.42) -- (1.95,1.43) -- (1.95,1.44) -- (1.95,1.45) -- (1.94,1.46) -- (1.95,1.46) -- (1.94,1.47) -- (1.95,1.47) -- cycle;

\node[racinesimple, label=left :{$\alpha$}, label=above :{$0$}] (a) at (0,0) {};
\node[racinesimple, label=right :{$\gamma$}, label=above:{$0$}] (b) at (4,0) {};
\node[racinesimple, label=above:{$\beta$}, label= left:{$0$}] (g) at (2.00000000000000,3.46410161513775) {};
\draw[green!75!black] (a) -- (b) -- (g) -- (a);
\node[racines, label=right :{$0$}] at  (3.00000000000000,1.73205080756888) {};
\node[racines, label=left :{$0$}] at  (0.713791735784419,1.23632355240139) {};
\node[racines, label=left :{$0$}] at  (1.28620826421558,2.22777806273636) {};

\node[racines, label=right :{$0$}] at  (1.57814262540253,0.911141069595774) {};
\node[racines, label=left :{$0$}] at  (0.890083735825258,1.54167025344006) {};
\node[racines, label=right :{$0$}] at  (2.34837507693250,1.35583498282519) {};
\node[racines, label=left:{$0$}] at  (1.10991626417474,1.92243136169769) {};

\node[racines, label=right :{$1$}] at  (2.14137520735326,1.23632355240139){};
\node[racines, label=south east :{$0$}] at (1.84787117021942,1.06686891755393) {};
\node[racines, label=above :{$1$}] at  (1.73484916695252,1.85948437682656) {};
\node[racines, label=left :{$0$}] at  (1.00000000000000,1.73205080756888) {};

\fill[racine] (2.12199784930487,1.54167025344006) circle (\grosseur);
\fill[racine] (2.00000000000000,1.15470053837925) circle (\grosseur);
\fill[racine] (1.89008373582526,1.73205080756888) circle (\grosseur);
\node[racines, label=below :{$2$}] at  (1.26495775914699,1.35583498282519) {};

\fill[racine] (2.14137520735326,1.23632355240139) circle (\grosseur);
\fill[racine] (1.84787117021942,1.06686891755393) circle (\grosseur);
\fill[racine] (1.62234914662510,1.17866633129016) circle (\grosseur);
\fill[racine] (1.39612452839032,1.73205080756888) circle (\grosseur);
\fill[racine] (2.00000000000000,1.64181947594621) circle (\grosseur);
\fill[racine] (1.73484916695252,1.85948437682656) circle (\grosseur);
\fill[racine] (1.40421148623362,1.28680311492469) circle (\grosseur);

\fill[racine] (2.34837507693250,1.35583498282519) circle (\grosseur);
\fill[racine] (1.50604079256507,1.23632355240139) circle (\grosseur);
\fill[racine] (1.71379173578442,1.73205080756888) circle (\grosseur);
\fill[racine] (1.52604754180084,1.73205080756888) circle (\grosseur);
\fill[racine] (1.27032300324898,1.57597974887493) circle (\grosseur);
\fill[racine] (1.89008373582526,1.73205080756888) circle (\grosseur);
\fill[racine] (1.26495775914699,1.35583498282519) circle (\grosseur);
 \fill[racine] (2.00000000000000,1.33036378002989) circle (\grosseur);

\coordinate (ancre) at (-0.5,2.6);
\node[racinesimple,label=below left:{$s_\alpha$}] (alpha) at (ancre) {};
\node[racinesimple,label=below right :{$s_\gamma$}] (beta) at ($(ancre)+(0.5,0)$) {};
\node[racinesimple,label=above:{$s_\beta$}] (gamma) at ($(ancre)+(0.25,0.43)$) {} edge[thick] node[auto,swap] {{\Large 7}} (alpha) edge[thick] (beta);
\coordinate (mat) at (4, 2.8);

\end{tikzpicture}
}
\caption{The normalized isotropic cone $\h{Q}$ in red, the small roots (with $\infty$-depth~$0$)  and the $\infty$-depth on some other roots  for the normalized root system associated to the Coxeter graph on the top right corner. The roots in the pointed shaded red cone form the blind cone (c.f. Remark~\ref{rem:Proj3}) of the root of depth $2$ it is pointed on.}
\label{fig:Ex237}
\end{figure}
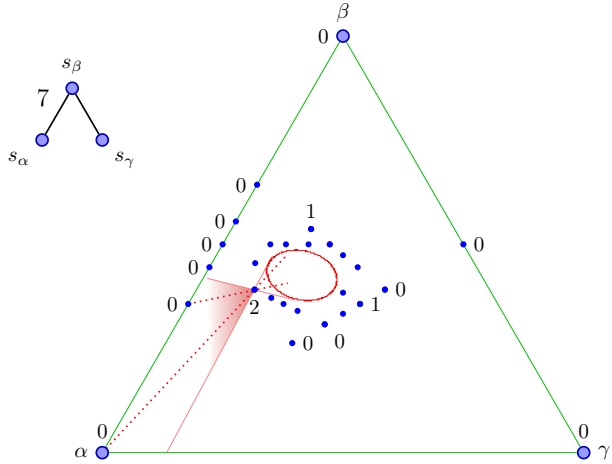

\end{ex}

\subsection{Small roots and small inversion sets}\label{sse:Small}
 
 \begin{defi} Let $\beta\in \Phi^+$ and $n\in\mathbb N$.
\begin{enumerate}[(i)]  

\item The positive root $\beta$ is  \emph{small}\footnote{These roots are also called \emph{humble} or \emph{elementary} in the literature. We adopt here the terminology of~\cite{BjBr05}. See \cite[Notes, p.130]{BjBr05} for more details.} if $\beta$ dominates no other positive root than itself, i.e., $\dep_\infty(\beta)=0$.

\item The positive root $\beta$ is {\em $n$-small} if $\dep_\infty(\beta)\leq n$. 

\item The set of $n$-small roots is denoted by 
 $
 \Sigma_n(W)$, or simply by $\Sigma_n$ if there is no possible confusion.  We denote by $\Sigma=\Sigma_0$ the set of small roots in $\Phi$.
\end{enumerate}
\end{defi}

\begin{remark} The collection $(\Sigma_n)_{n\in\mathbb N}$ is a filtration of $\Phi^+$: we have $$
\Phi^+=\bigcup_{n\in\mathbb N} \Sigma_n\quad \textrm{and}\quad\Delta\subseteq \Sigma=\Sigma_0\subseteq \Sigma_1\subseteq \dots\subseteq \Sigma_n,\  \forall n\in\mathbb N.
$$
 \end{remark}

\begin{prop}\label{prop:NSmallParab} If $I\subseteq S$ and $n\in\mathbb N$, then $\Sigma_n(W_I)\subseteq \Sigma_n(W)$.
\end{prop}
\begin{proof}  The result follows from Lemma~\ref{lem:Domin}.
\end{proof}

The next proposition gives a useful characterization of small roots, see \cite{BrHo93,BjBr05}.

\begin{prop}\label{prop:SmallMin}  Let~$\beta$ in $\Phi^{+}\setminus \Delta$ and $\alpha\in \Delta$ such that $B(\alpha,\beta)>0$, then  $\beta\in \Sigma$ if and only if $s_{\alpha}(\beta)$ lies in~$\Sigma$ and $B(\alpha,\beta)<1$ holds; if and only if $\dep(s_\alpha(\beta))<\dep(\beta)$.
\end{prop}

The case $n=0$ in the following theorem is due to Brink-Howlett~{\cite[Theorem~2.8]{BrHo93}} whereas the general case is due to Fu~{\cite[Corollary~3.9]{Fu12}~and~\cite[Theorem~3.22]{Fu13}} and Dyer (unpublished; see~\cite{Ed09}).

\begin{thm}\label{thm:NSmall}  For all $n\in\mathbb N$, the set $\Sigma_n$ is finite and does not depend of the choice of the root system. In particular $\Sigma_0=\Phi^+$ if and only if $W$ is finite 
\end{thm}

 \begin{ex}\label{ex:Inf3} If $W$ is the infinite dihedral group, as in Examples~\ref{ex:Inf1}~and~\ref{ex:Inf2}, then $\Sigma=\Delta=\{\alpha_s,\alpha_t\}$. For similar reason, $\Sigma=\Delta$ if $W$ is a universal Coxeter group as in Example~\ref{ex:Univ1}.
\end{ex}

\begin{ex}\label{ex:Small} Small roots are all represented in the cases illustrated in Figures~\ref{fig:AffineG2}~and~\ref{fig:Ex237}. In the case of the affine Coxeter group of type $\tilde A_2$ as in Example~\ref{ex:AffineA2} and Figure~\ref{fig:AffineA2},  the small roots are the roots corresponding to proper standard parabolic subgroups: 
$
\Sigma=\{\alpha_1,\alpha_2,\alpha_3,\alpha_1+\alpha_2,\alpha_1+\alpha_3,\alpha_2+\alpha_3\}. 
$
\end{ex}

 The case of small roots ($n=0$) is at the heart of the work of Brink and Howlett~\cite{BrHo93} on the automatic structure of Coxeter systems: they provide, by the mean of {\em small inversion sets}, a finite state automaton that recognizes the language of reduced words. The notion of small inversion sets will be used in \S\ref{sse:Low} and beyond.
 
 \begin{defi} Let $n\in \mathbb N$, the {\em (left) $n$-small inversion set of $w\in W$} is 
  $$
    \Sigma_n(w)=N(w)\cap \Sigma_n.
  $$ 
  We denote by $\Lambda_n(W)$ (or simply $\Lambda_n$ if there is not possible confusion) the set of all (left) $n$-small inversion sets. 
 \end{defi}
 

 \begin{remark} In \cite[\S4.8]{BjBr05}, the authors use {\em right inversion sets}: $N(w^{-1})\cap \Sigma$. 
 \end{remark}
 
 Theorem~\ref{thm:NSmall} has the following interesting direct consequence:
 
 \begin{cor}\label{cor:SmallInv} The set $\Lambda_n(W)$ is finite for all $n\in\mathbb N^*$.
 \end{cor}


We end this discussion on small descent sets with the following lemma and proposition. Note that we do not need these results  in the rest of this article; but we state them anyway since they play an important role  in relation to finite state automata associated to Coxeter groups (see \cite{Ed09} for a multi-parameter generalization). For $n=0$, this goes back to \cite{BrHo93}, see also~\cite[\S4.8]{BjBr05}.

\begin{lem}\label{lem:SmallInv} Let $w\in W$, $s\in S$ and $n\in \mathbb N$.
\begin{enumerate}
 \item If $\ell(sw)>\ell(w)$,  then $\Sigma_n(sw) =\{\alpha_s\}\sqcup (\Sigma_n\cap s(\Sigma_n(w)))$.
\item  If $n\geq 1$ and $\ell(sw)<\ell(w)$,   then $\Sigma_{n-1}(sw) =\Sigma_{n-1}\cap s(\Sigma_n(w)\setminus\{\alpha_s\})$.
\end{enumerate} 
\end{lem} 
\begin{proof}
Under the assumptions of {\em (1)}, the word $sw$ is reduced and $\alpha_s\in \Delta\subseteq \Sigma_n$. So $N(w)=\{\alpha\}\sqcup s(N(w))$ by Corollary~\ref{cor:RedInv}, which implies the right hand side of {\em (1)} is contained in the left hand side. To prove the reverse inclusion, it suffices to show that if $\beta\in \Sigma_n(sw)=N(sw)\cap\Sigma_n$ and $\beta\not = \alpha_s$, then $s(\beta)\in \Sigma_n$, since $s(\beta)\in N(w)$ from above. Assume  for a contradiction that $\dep_\infty(s(\beta))\geq n+1$.  Then Proposition~\ref{prop:NSmall} forces  $B(\alpha_s,\beta)\leq -1$ since $\dep_\infty(\beta) \leq  n$.  Since $\alpha_s,\beta\in N(sw)$, this contradicts Proposition~\ref{prop:FiniteRefl}. 

The proof of {\em (2)} is similar but simpler. Under these assumptions, we have $\alpha_s\in \Delta\subseteq\Sigma_n$ and $w=s(sw)$ is reduced. So $N(sw)=s(N(w)\setminus\{\alpha_s\})$ by Corollary~\ref{cor:RedInv}.  This implies the right hand side of {\em (2)} is contained in the left hand side.  To prove the reverse inclusion, it suffices to show that if $\beta\in \Sigma_{n-1}(sw)$, then 
$s(\beta)\in \Sigma_{n}$, since $s(\beta)\in N(w)$ from above and we cannot have $s(\beta)=\alpha$.  The result then follows from Proposition~\ref{prop:NSmall}.
\end{proof}

\begin{prop} \label{prop:SmallInv} For any $n\in \mathbb N$, the  finite set $\Lambda_n$ is the inclusion-minimum subset of $\Sigma_{n}$  such that
\begin{enumerate}
\item $\emptyset\in \Lambda_{n}$
\item If $A\in \Lambda_{n}$ and  $\alpha\in \Delta\setminus A$,
then $\{\alpha\}\cup (s_{\alpha}(A)\cap \Sigma_{n})\in \Lambda_{n}$
\end{enumerate}
\end{prop} 
 \begin{proof} Suppose $A:=\Sigma_n(w)$ for some $w\in W$.  Then for $s\in S$, one has $\ell(sw)>\ell(w)$ if and only if $\alpha_s\not\in  N(w)$, so if and only if $\alpha_s\not\in  A$. In that case $\Sigma_n(sw)=\{\alpha\}\cup (s(A)\cap \Sigma_{n})$  which is  completely determined by $A$ and $s$.  The result is clear from this using Lemma~\ref{lem:SmallInv}. 
 \end{proof}   

\begin{remark}\label{rem:Automaton} Following~\cite[p.119--120]{BjBr05}, we define for $n\in \mathbb N$ a {\em $n$-canonical automaton} that recognizes the language of reduced words: the set of states is $\Lambda_n(W)$; for each $A\in \Lambda_n(W)$ 
and $s\in S$ such that $\alpha_s\in \Delta\setminus A$ we put a transition:
$$
A\xrightarrow{s} \{\alpha\}\cup (s_{\alpha}(A)\cap \Sigma_{n}).
$$ 
Note that the $0$-canonical automaton is the canonical automaton described in \cite[p.120]{BjBr05} and in~\cite[\S4.3]{Ed09}. 
However if $n>0$, the {\em canonical automata} defined and studied in~\cite[\S4.3]{Ed09} are different than the one we define above.
\end{remark}

 \subsection{Low elements}\label{sse:Low}  We are now ready to define  low elements.

\begin{defi} Let $n\in\mathbb N$. An element $w\in W$ is {\em $n$-low} if  $N(w)=\cone_{\Phi}(A)$ for some $A\subseteq \Sigma_n$, or equivalently if $N(w)=\cone_\Phi(\Sigma_n(w))$. We denote by $L_n(W)$ the set of $n$-low elements in $W$. 
\smallskip

\noindent A $0$-low element is simply called a {\em low element} and $L_0(W)$ is denoted  by $L(W)$.
\end{defi}

\begin{remark} The collection $(L_{n}(W))_{n\in \mathbb{N}}$ is a filtration of $W$; one has 
\begin{equation*}
W=\bigcup_{n\in \mathbb{N}}L_{n}(W), \text{ \rm and } L_{0}(W)\subseteq L_{1}(W)\subseteq \ldots \subseteq L_{n}(W)\subseteq\ldots , \quad \forall n \in \mathbb{N}.
\end{equation*}
\end{remark}

The following proposition shows part of Theorem~\ref{thm:Main1}.

\begin{prop}\label{prop:Low} Let $n\in\mathbb N$. 
\begin{enumerate}
\item We have $S\cup\{e\}\subseteq L_n(W)$.
\item  The map $\Sigma_n:L_n(W)\to \Lambda_n(W)$ is injective; 
\item The set $L_n(W)$ is finite and  closed under join, i.e., if~$X\subseteq L_n(W)$ is bounded then $\bigvee X\in L_n(W)$. 
\end{enumerate}
\end{prop}
\begin{proof}  For the first statement: $e\in L_n(W)$ since $N(e)=\emptyset = \cone_\Phi (\emptyset)$; moreover $N(s)=\{\alpha_s\}=\cone_\Phi(\alpha_s)$  and  $\alpha_s\in \Delta\subseteq \Sigma_0\subseteq \Sigma_n$. For {\em (2)}, let $w,w'\in L_n(W)$ such that $\Sigma_n(w)=\Sigma_n(w')$, then 
$
 N(w)=\cone_\Phi(\Sigma_n(w))=\cone_\Phi(\Sigma_n(w))=N(w').
 $
 By Proposition~\ref{prop:Weak} $N$ is injective and so $w=w'$, which implies that $\Sigma$ is injective.  It remains to show {\em (3)}. Finiteness of $L_{n}(W)$ holds by (2) and  Corollary~\ref{cor:SmallInv}.   Now let   $X\subseteq L_n(W)$  bounded. For each $x\in X$, we have by definition  that $N(x)= \cone_\Phi(\Sigma_n(x))$ and $\bigvee X$ exists.   The fact that $L_n(W)$ is closed under join follows now from the definition and Proposition~\ref{prop:Weak}:
  $$
N\left(\bigvee X\right)=\cone_\Phi \left(\bigcup_{x\in X} \cone_\Phi(\Sigma_n(x))\right)=\cone_\Phi \left(\bigcup_{x\in X}\Sigma_n(x)\right),
$$ 
and $\bigcup_{x\in X}\Sigma_n(x)\subseteq\Lambda_n$. So $\bigvee X\in L_n(W)$.
\end{proof}

\begin{ex}\label{ex:Inf4}
If $W$ is the infinite dihedral group $\mathcal D_\infty$, as in Examples~\ref{ex:Inf1}, \ref{ex:Inf2} and~\ref{ex:Inf3}, then $\Sigma=\Delta=\{\alpha_s,\alpha_t\}$.  Therefore  the conic hull of the two small roots contains $\Phi^+$ which is infinite in this case.

Hence $\cone_\Phi(\Delta)=\Phi^+$ cannot be an inversion set. So $L(W)=S\cup\{e\}$.
\end{ex}

\begin{ex}\label{ex:Univ2} If $(W,S)$ is a universal Coxeter system, as in Examples~\ref{ex:Univ1}~and~\ref{ex:Inf3}, then $\Sigma=\Sigma_0=\Delta$ and therefore $L(W)=S\cup\{e\}$. Indeed,   the conic hull of two small roots is the positive root system for an infinite dihedral group, a standard parabolic subgroup of rank $2$ as in Examples~\ref{ex:Inf3}~and~\ref{ex:Inf4}; hence we cannot have an inversion set that arises from the conic hull of more than one small root. 
\end{ex}

\begin{ex} In the case of the affine Coxeter group of type $\tilde A_2$ as in Examples~\ref{ex:AffineA2}~and~\ref{ex:Small} and Figure~\ref{fig:AffineA2}, 
it is not difficult to see, using the same techniques as in Example~\ref{ex:AffineA2}, that
$$
 L(\tilde A_2)=\{e,1,2,3,12,21,13,31,23,32,121,131,232,1232,2313,3121\}.
 $$
\end{ex}

 In both these examples and more, the map $\Sigma$ is bijective, which leads us to state the following conjecture.

 \begin{conject}\label{conj:2} 
 The  map $\Sigma_n:L_n(W)\to \Lambda_n(W)$ is a bijection. 
 \end{conject}

\begin{remark} In the examples above, the set of low elements $L(W)$ is also the smallest Garside shadow $\tilde S$ from Remark~\ref{rem:SmallG}. But it is not true in general. If $(W,S)$ is of affine type $\tilde G_2$ (see Figure~\ref{fig:AffineG2}), then $w=s_\alpha s_\gamma s_\beta$ is a low element since its inversion set is $N(w)=\cone_\Phi(\alpha,\gamma,\nu)$,  where $s_\alpha(\gamma)=\gamma$ and $\nu=s_\alpha s_\gamma (\beta)$;  but is not in $\tilde S$ since we cannot obtain $w$ by join or suffix closure starting from $S$.
\end{remark}

\section{Low elements form a finite Garside shadow}\label{se:Suffix}

The aim of this section is to finish the  proof of Theorem~\ref{thm:Main1}. In regard to Proposition~\ref{prop:Low}{\em (3)}, we just have to show that the set $L(W)$ of low elements is closed under taking suffix. In order to do this we first give a description of the rays of the cone over the inversion set of a suffix of $w\in W$ as a function of rays of $\cone(N(w))$. Then we prove that the set of small roots  is {\em bipodal}: if a small root  is  a positive and non-simple root of a maximal dihedral reflection subgroup $W'$, then the simple system $\Delta_{W'}$ is constituted of two small roots.

\subsection{Bruhat order}\label{sse:Bruhat} The \emph{Bruhat order} on $W$
is the partial order $\leq$ arising as the reflexive, transitive closure of the relation $\to$ on $W$ 
defined by $x\to y$ if there is $\beta\in \Phi^+$ such that $y=s_\beta x$ with $\ell(x)<\ell(y)$. We say that a pair $x\leq y$ is a {\em covering in the Bruhat order}, which we denote by $x\lhd y$,  if for any $z\in W$ such that $x\leq z\leq y$ then we have $z=x$ or $z=y$. The chain property of Bruhat order implies that $x\lhd y$  is a {\em covering}  if $x<y$ and  $\ell(y)=\ell(x)+1$, i.e., if there is $\beta\in \Phi^+$ such that $y=s_\beta x$ and  $\ell(y)=\ell(x)+1$. It is known that $u\leq_R v$ implies that $u\leq v$.   We refer the reader to \cite[Chapter~2]{BjBr05} for more details. The following well-known proposition can be found for instance as \cite[Proposition 2.2.7]{BjBr05}.

\begin{prop}[Lifting property]\label{prop:Lifting} Let $s\in S$ and $u<v$ in $W$ such that $sv\lhd v$ and $u\lhd su$, then $u\leq sv$ and $su\leq sv$. 
\end{prop}

The next corollary is used in the proof of Theorem~\ref{thm:BaseSuff} below.

\begin{cor}\label{cor:Lifting}  Let $s\in S$, $x\in W$ and $\gamma\in \Phi^+$ such that $sx\lhd x$ and $s_\gamma x\lhd x$. Then either $\gamma=\alpha_s$ or $ss_\gamma x\lhd s_\gamma x$.
\end{cor} 
\begin{proof} Since $s\in S$ we have  $\ell(ss_\gamma x)=\ell(s_\gamma x)\pm 1$. So either $ss_\gamma x \lhd s_\gamma x$ and we are done,  or $s_\gamma x\lhd ss_\gamma x$. In this last case, we have $\ell(ss_\gamma x)=\ell(s_\gamma x)+1=\ell(x)$ since $s_\gamma x \lhd x$. Moreover Proposition~\ref{prop:Lifting} implies, with $u=s_\gamma x$ and $v= x$, that $ss_\gamma x \leq x$. Since $ss_\gamma x$ and $x$ have the same length, we have $x=ss_\gamma x$, forcing $s=s_\gamma$, i.e.,  $\gamma =\alpha_s$.
\end{proof}

\begin{ex}\label{ex:BruhatDi}  Let $W$ be a dihedral group $\mathcal D_m$ ($m\in\mathbb N_{\geq 2}\cup\{\infty\})$ generated by $S=\{s,t\}$ with Coxeter graph:
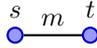
\begin{center}
\begin{tikzpicture}[sommet/.style={inner sep=2pt,circle,draw=blue!75!black,fill=blue!40,thick}]
	\node[sommet,label=above:$s$] (alpha) at (0,0) {};
	\node[sommet,label=above:$t$] (beta) at (1,0) {} edge[thick] node[auto,swap] {$m$} (alpha);
\end{tikzpicture}
\end{center}
We follow the notations introduced in Example~\ref{ex:Inf2} for $[s,t]_k$, $[t,s]_k$, $\alpha_{s,k}$ and $\alpha_{t,k}$ for $k\in\mathbb N$. We illustrate in Figure~\ref{fig:BruhatDi} the Hasse diagrams for the Bruhat order on dihedral groups with the edges representing the covering $x\lhd y$ labelled by the root $\beta$ such that $y=s_\beta x$.
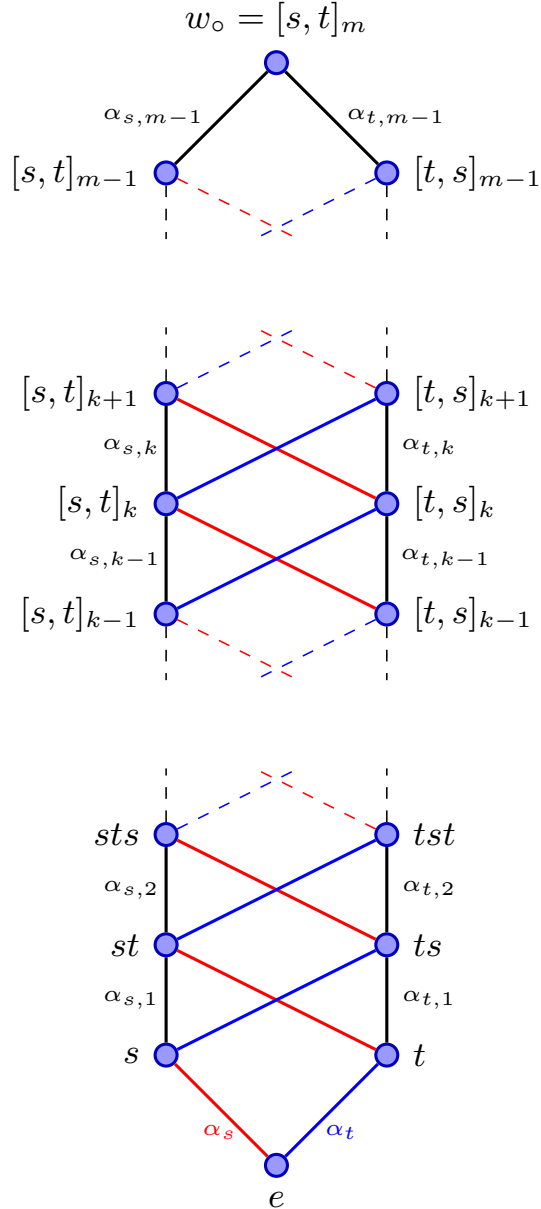
\begin{figure}[h!]
\resizebox{.6\hsize}{!}{
\begin{tikzpicture}
	[scale=1,
	 pointille/.style={dashed},
	 axe/.style={color=black, very thick},
	 sommet/.style={inner sep=2pt,circle,draw=blue!75!black,fill=blue!40,thick,anchor=west}]
	 
\node[sommet]  (id)    [label=below:{\small{$e$}}]          at (0,0)    {};
\node[sommet]  (1)    [label=left:{\small{$s$}}]         at (-1,1)   {} edge[thick, red] (id) ;
\node[sommet]  (2)    [label=right:{\small{$t$}}]        at (1,1)    {} edge[thick,blue] (id);
\node[sommet]  (12)   [label=left:{\small{$st$}}]      at (-1,2)   {} edge[thick] (1) {} edge[thick,red] (2);
\node[sommet]  (21)   [label=right:{\small{$ts$}}]     at (1,2)    {} edge[thick] (2) {} edge[thick,blue] (1);
\node[sommet]  (121)  [label=left:{\small{$sts$}}]    at (-1,3)    {} edge[thick] (12) {} edge[thick,red] (21);
\node[sommet]  (212)  [label=right:{\small{$tst$}}]    at (1,3)    {} edge[thick] (21){} edge[thick,blue] (12) ;

\draw[pointille] (121) -- +(0,0.6);
\draw[pointille] (212) -- +(0,0.6);
\draw[pointille,blue] (121) -- +(1.2,0.6);
\draw[pointille,red] (212) -- +(-1.2,0.6);

\node[sommet]  (12k)  [label=left:{\small{$[s,t]_{k-1}$}}]    at (-1,5)    {};
\node[sommet]  (21k)  [label=right:{\small{$[t,s]_{k-1}$}}]    at (1,5)    {};
\draw[pointille] (12k) -- +(0,-0.6);
\draw[pointille] (21k) -- +(0,-0.6);
\draw[pointille,red] (12k) -- +(1.2,-0.6);
\draw[pointille,blue] (21k) -- +(-1.2,-0.6);
\node[sommet]  (12kk)  [label=left:{\small{$[s,t]_{k}$}}]    at (-1,6)    {} edge[thick] (12k) {} edge[thick,red] (21k);
\node[sommet]  (21kk)  [label=right:{\small{$[t,s]_{k}$}}]    at (1,6)    {} edge[thick,blue] (12k) {} edge[thick] (21k);
\node[sommet]  (12kkk)  [label=left:{\small{$[s,t]_{k+1}$}}]    at (-1,7)    {} edge[thick] (12kk) {} edge[thick,red] (21kk);
\node[sommet]  (21kkk)  [label=right:{\small{$[t,s]_{k+1}$}}]    at (1,7)    {} edge[thick,blue] (12kk) {} edge[thick] (21kk);

\draw[pointille] (12kkk) -- +(0,0.6);
\draw[pointille] (21kkk) -- +(0,0.6);
\draw[pointille,blue] (12kkk) -- +(1.2,0.6);
\draw[pointille,red] (21kkk) -- +(-1.2,0.6);

\node[sommet]  (12f)  [label=left:{\small{$[s,t]_{m-1}$}}]    at (-1,9)    {};
\node[sommet]  (21f)  [label=right:{\small{$[t,s]_{m-1}$}}]    at (1,9)    {};
\draw[pointille] (12f) -- +(0,-0.6);
\draw[pointille] (21f) -- +(0,-0.6);
\draw[pointille,red] (12f) -- +(1.2,-0.6);
\draw[pointille,blue] (21f) -- +(-1.2,-0.6);
\node[sommet]  (wo)    [label=above:{\small{$w_\circ=[s,t]_m$}}]          at (0,10)     {} edge[thick] (12f)  {} edge[thick] (21f) ;

\draw (-0.4,0.3) node[auto,swap,red] {\tiny{$\alpha_s$}};
\draw (0.7,0.3) node[auto,swap,blue] {{\tiny $\alpha_t$}};

\draw (-1.2,1.5) node[auto,swap] {\tiny{$\alpha_{s,1}$}};
\draw (-1.2,2.5) node[auto,swap] {\tiny{$\alpha_{s,2}$}};
\draw (-1.35,5.5) node[auto,swap] {\tiny{$\alpha_{s,k-1}$}};
\draw (-1.2,6.5) node[auto,swap] {\tiny{$\alpha_{s,k}$}};
\draw (-1,9.5) node[auto,swap] {\tiny{$\alpha_{s,m-1}$}};

\draw (1.5,1.5) node[auto,swap] {\tiny{$\alpha_{t,1}$}};
\draw (1.5,2.5) node[auto,swap] {\tiny{$\alpha_{t,2}$}};
\draw (1.65,5.5) node[auto,swap] {\tiny{$\alpha_{t,k-1}$}};
\draw (1.5,6.5) node[auto,swap] {\tiny{$\alpha_{t,k}$}};
\draw (1.2,9.5) node[auto,swap] {\tiny{$\alpha_{t,m-1}$}};

\end{tikzpicture}}
\caption{ The Hasse diagram of the Bruhat order on  the finite dihedral group $\mathcal D_m$. For $\mathcal D_\infty$, the Hasse diagram is the diagram with infinite vertices obtained by not considering the top part of the above diagram. The labels for the interior edges are as follows: the `parallel' red  ones corresponds to the label $\alpha_s$ and the `parallel' blue ones to the label  $\alpha_t$. All the labels on the other edges are distinct. }
\label{fig:BruhatDi}
\end{figure}
\end{ex}

The above example of dihedral group leads to the following lemma, which will be used to prove Theorem~\ref{thm:BaseSuff} below (recall the definition of maximal dihedral reflection subgroup in Example~\ref{ex:MaxDi}).

\begin{lem}\label{lem:BruhatDi} Let $s\in S$, $\beta\in \Phi^+$ and $x\in W$ such that $sx\lhd x$.  Denote $W'=\langle s,s_\beta\rangle$, a reflection subgroup  with simple system $\Delta_{W'}$ and positive root system $\Phi^+_{W'}$. The following assertion are equivalent:
\begin{enumerate}
\item  $ss_\beta sx\lhd s_\beta sx \lhd sx\lhd x$;

\item $W'$ is a maximal dihedral reflection subgroup,~$\Delta_{W'}=\{\alpha_s,\beta\}$, $\ell(s_{s(\beta)}x)=\ell(x)-3$ and there is $\gamma \in\Phi_{W'}^+\setminus \{\alpha_s\}$ such that~$s_\gamma x\lhd x$.

\end{enumerate}
\end{lem}
\begin{proof} Assume {\em (1)} to be true. Note that $W'$ is of rank $2$ since $\alpha_s\not = \beta$ by  {\em (1)}. Let $W''$ be the maximal dihedral reflection subgroup containing $W'$.  Write $x=uv$ with $u\in W''$ and $v\in X_{W''}$ as in Proposition~\ref{prop:CosetRep}. Proposition~\ref{prop:CosetRep}{\em (2)} shows that the map $\psi_v:(W'',\leq_{W''})\mapsto (W''v,\leq)$, $z\mapsto zv$, is an order preserving  bijection. Therefore, even though the inverse map may  not be order preserving,  {\em (1)} forces 
$$
ss_\beta su\lhd_{W''} s_\beta su \lhd_{W''} su\lhd_{W''} u.
$$
But since the graph of the Bruhat order on $W''$ is   as in Example~\ref{ex:BruhatDi} and Figure~\ref{fig:BruhatDi}, the only possible chains of covering with repetition of the same reflection is one involving the two canonical generators of $W''$.  So $S''=\{s,s_\beta\}$ must be the canonical generators of $W''$. In other words, $W'=W''$ and $\Delta_{W'}=\{\alpha_s,\beta\}$. 

Moreover, since each Bruhat interval of length $2$ (or more) in dihedral groups have two coatoms (see Figure~\ref{fig:BruhatDi}), there is  $\gamma \in\Phi_{W'}^+\setminus \{\alpha_s\}$ such that $s_\beta s u \lhd_{W'} s_\gamma u\lhd_{W'} u$. Applying the order preserving map $\psi_v$ we obtain $s_\beta s x < s_\gamma x< x$. Since by assumption $\ell(s_\beta s x)=\ell(x)-2$, we have 
$s_\beta s  x \lhd s_\gamma x \lhd x$.

\smallskip

Assume now {\em (2)} to be true. As in {\em (1)} we use the decomposition $x=uv$ and the map $\psi_v$.  Observe that,  and $\ell(ss_{\beta}sx)=\ell(x)-3=l(sx)-2$ since $sx\lhd x$. We must have $s_{\beta}sx<sx$, since otherwise $\ell(ss_{\beta}sx)\geq \ell(s_{\beta}sx)-1\geq \ell(sx)$. So 
$s_{\beta}s u\lhd_{W'} su\lhd_{W'} u$, since $S_{W'}=\{s,s_\beta\}$,  and $s_{\gamma}u\lhd_{W'} u$. This implies that  $s_{\beta}s u\lhd_{W'} s_{\gamma} u\lhd_{W'} u$ since $s_\gamma u$ and $su$ are the two coatoms of the length $2$ interval $[s_{\beta}s u,u]$ in $(W',\leq_{W'})$. 

If $ss_{\beta}su >_{W'}s_{\beta}su$, this would force $s_\beta su\lhd_{W'} ss_\beta su$ since $S_{W'}=\{s,s_\beta\}$. Therefore, $ss_\beta su\in\{su,s_\gamma u\}$ since there is only two atoms above $s_{\beta}su$ in $(W',\leq_{W'})$. Since $s_\beta\not = s$, we have $ss_\beta su=s_\gamma u$ and hence $ss_{\beta}sx=s_{\gamma}x$ contrary to $s_{\gamma}x\lhd x$ and $\ell(ss_{\beta}sx)=l(x)-3$.

Hence $ss_{\beta}su <_{W'}s_{\beta}su$ and therefore $ss_{\beta}su <_{W'}s_{\beta}su<_{W'}su <_{W'} u $. Applying $\psi_v$ shows   $ss_{\beta}sx<s_{\beta}sx<sx<x$. 
Since  $\ell(ss_{\beta}sx)=\ell(x)-3$, it follows that $ss_{\beta}sx\lhd s_{\beta}sx\lhd sx\lhd x$.
\end{proof}

\subsection{Bases of inversion sets}\label{sse:BaseInv}

 Using the Bruhat order, the first author (M.D.)~\cite{Dy93,Dy94-1,Dy11} describes the rays of the cone over an inversion set. If $C$ is a cone such that $C\cap -C=\{0\}$,  an  {\em extreme ray of $C$} is a ray $\mathbb R^+\alpha$ where $\alpha\in C\setminus\{0\}$  such that if $\beta,\gamma\in C$ and  $\mathbb R^+\alpha\subseteq  \cone(\beta,\gamma)$, then $\beta\in \mathbb R^+\alpha$ or $\gamma\in \mathbb R^+\alpha$.  In that case, any non-zero $\alpha\in F$  is called a {\em representative of $F$}. 
 
 \begin{defi} Let $w\in W$, {\em the base $N^1(w)$ of the inversion set $N(w)$} is the set of representatives  of $\cone(N(w))$ constituted of roots:
 $$
 N^1(w)=\{\beta\in \Phi^+\,|\, \mathbb R^+\beta\ \textrm{ is an extreme ray of} \cone(N(w))\}.
 $$
\end{defi}
 
 The following proposition is~\cite[Lemma~1.7(a)]{Dy11}.
 
 \begin{prop}\label{prop:BaseInv} Let $w\in W$, then the base of $N(w)$ is:
$$
 N^1(w)=\{\beta\in \Phi^+\,|\,\ell(s_\beta w)= \ell(w)-1\}=\{\beta\in \Phi^+\,|\,s_\beta w\lhd w\}.
 $$
 \end{prop}
 
 The concept of base of inversion sets  has the following interesting consequences for  the weak order and $n$-low elements that are worth mentioning, but are not necessary to prove Theorem~\ref{thm:Main1}. 
 
 \begin{cor}\label{cor:BInv}  Let $x,y\in W$ and $\beta\in \Phi$.
\begin{enumerate}
\item If $\ell(s_{\beta}(x\wedge y))=\ell(x\wedge y)+1$, then either
$\ell(s_{\beta}x)=\ell(x)+1$ or $\ell(s_{\beta}y)=\ell(y)+1$.
\item If $x\vee y$ exists and  $\ell(s_{\beta}(x\vee y))=\ell(x\vee y)-1$, then either
$\ell(s_{\beta}x)=\ell(x)-1$ or $\ell(s_{\beta}y)=l(y)-1$.
\end{enumerate}
\end{cor}
\begin{proof} We prove only {\em (2)}, the proof of {\em (1)} being similar using~\cite[Lemma~1.7(b)]{Dy11}.
 Proposition~\ref{prop:Weak}  and   Proposition~\ref{prop:BaseInv} with $w=x$ and  $w=y$ respectively give
$$
N(x\vee y)=\cone_\Phi(N^1(x)\cup N^1(y)). 
$$
So by definition of the base of an inversion set, we have   $N^1(x \vee y)\subseteq N^1(x)\cup N^1(y)$. This is equivalent to the desired conclusion, by Proposition~\ref{prop:BaseInv}.
 \end{proof}

\begin{prop}\label{prop:BInv} Let $u,v\in W$ both be  $n$-low elements.
\begin{enumerate}
\item  For    $x\in W$, one has  $u\leq_R x$ if and only if $N^1(u)\subseteq \Sigma_n(x)$.
\item The  join of $u$ and $v$ exists (i.e.  $\{u,v\}$ is bounded in $W$) if and only if  there is some  $A\in \Lambda_{n}$ with $N^1(u)\cup N^1(v)\subseteq A$.
\end{enumerate}
\end{prop}
\begin{proof}
We prove {\em (1)}.  If  $v\leq_R x$, then  $N(u)\subseteq N(x)$. Since $u$ is $n$-low we have  $N^1(u)\subseteq \Sigma_n(u)=N(u)\cap\Sigma_n\subseteq N(x)\cap \Sigma_{n}=\Sigma_n(x)$.  Conversely, if $N^1(u)\subseteq\Sigma_n(x)$, then  $N^1(u)\subseteq N(x)$ and so $N(u)=\cone_{\Phi}(N^1(u))\subseteq N(x)$ i.e. $u\leq_R x$ by Proposition~\ref{prop:Weak}.  Part {\em (2)} follows from {\em (1)}  and the definition of $\Lambda_{n}$.
\end{proof}
 
 \begin{remark}\label{rem:ComputeJoin} If $S$ is finite, the  proposition gives rise to an  algorithm to determine  whether 
two elements $v,w$ of $W$ have a join in weak order, since $\Sigma_{n}$ and $\Lambda_{n}$ can be effectively computed. In general (for possibly infinite $S$) one may find a finite subset $I\subseteq S$ such that $v,w\in W_{I}=\langle I\rangle$. Using  Proposition~\ref{prop:Weak},  for instance, one easily sees that  $v,w$ have a join in $W$ if and only if they have one in $W_{I}$, and that  if so, the joins of $v$ and $w$ in $W_{I}$ and $W$ coincide.
\end{remark}

\subsection{Inversion sets of suffixes}\label{sse:InvSuff} We give here a useful description of the base of the inversion set of a suffix of $w\in W$ obtained from the description of the base of~$N(w)$. We deduce from there a sufficient condition for a suffix of a $n$-low element to be $n$-low.

For each $s\in S$, we define a function $f_{s}: \Phi^{+}\setminus\{\alpha_s\}\to \Phi^{+}$ as follows.
Let $\beta\in \Phi^{+}\setminus \{\alpha_s\}$. Then the dihedral reflection subgroup $\langle s,s_{\beta}\rangle$ is contained in a unique maximal dihedral reflection subgroup $W'$, as defined in~Example~\ref{ex:MaxDi}. We have necessarily $\alpha_s\in\Delta_{W'}$ since $\alpha_s\in \Delta$, which is a basis of $\cone(\Phi^+)$. Define $f_{s}(\beta)$ to be the other element of $\Delta_{W'}$ i.e. by
$\Delta_{W'}=\{\alpha_{s},f_{s}(\beta)\}$. Equivalently, $f_{s}(\beta)$ is defined by the conditions $f_{s}(\beta)\in \Phi^{+}$ and 
$(\mathbb R\alpha_{s}+\mathbb R\beta)\cap \Phi^{+}=\cone(\alpha_s, f_{s}(\beta))\cap \Phi^{+}$.

\begin{thm}\label{thm:BaseSuff} Let $s\in S$ and $x\in W$.
\begin{enumerate}
\item If $sx$ is a suffix of $x$, i.e.,   $sx\lhd x$, then 
$$
N^1(sx)\subseteq s\left(N^1(x)\setminus\{\alpha_s\}\right)\cup f_{s}\left(N^1(x)\setminus\{\alpha_s\}\right).
$$
\item If $x$ is a suffix of $sx$, i.e.,   $x\lhd sx$, then 
$$
N^1(sx)=\{\alpha_s\}\sqcup s\left(\{\beta\in N^1(x)\,|\, s_{\beta}x<ss_\beta x\}\right).
$$
\end{enumerate}
\end{thm}

The proof of this theorem uses the following lemma.

\begin{lem}\label{lem:Suff} Let $s\in S$ and $x\in W$ such that $sx\lhd x$, then:
\begin{eqnarray}\label{eq:Suff1}
  \{\beta \in N^1(sx)\,|\, s_\beta sx < ss_\beta s x\}&=&  \{\beta \in \Phi^+\,|\, s_\beta sx \lhd ss_\beta s x \textrm{ and } s_\beta sx \lhd sx\}\\
 & =&s(N^1(x)\setminus\{\alpha_s\});\nonumber
  \end{eqnarray}
\begin{multline}\label{eq:Suff2}
 \{\beta \in N^1(sx)\,|\, s_\beta sx > ss_\beta s x\}=  \{\beta \in \Phi^+\,|\, ss_\beta sx \lhd s_\beta s x\lhd sx\lhd x\}\\
      \begin{array}{cl}
      =&\left\{\beta \in f_s\left(N^1(x)\setminus\{\alpha_s\}\right)\,|\,  ss_\beta sx \lhd s_\beta s x\lhd sx\lhd x  \right\}\\
      =&\left\{\beta \in f_s\left(N^1(x)\setminus\{\alpha_s\}\right)\,|\,  \ell(s_{s(\beta)}x) =\ell(x)-3  \right\}.
      \end{array}
\end{multline}
\end{lem}
\begin{proof} Recall that, for all $w\in W$,  $\ell(sw)=\ell(w)\pm 1$ since $s\in S$; so either $sw\lhd w$ or $w\lhd sw$ and $sw\lhd w$ if and only if $\ell(sw)=\ell(w)-1$.  This shows the first equalities in both (1) and (2) by taking $w=s_\beta s x$.
\smallskip

For the proof of (1),  suppose first that $\gamma\in N^1(x)\setminus\{\alpha_s\}$ and write $\beta=s(\gamma)$.   Observe that $s_\beta sx= s_{s(\gamma)} sx= ss_\gamma x$ since $s_\beta=s_{s(\gamma)}=ss_\gamma s$. We show that $\beta$ is in the left hand side of (1), that is:
$$
(\star)\qquad s_\beta sx=ss_\gamma x \lhd sx\ \textrm{ and }\ s_\beta sx =ss_\gamma x \lhd ss_\beta sx=s_\gamma x,
$$
by Proposition~\ref{prop:BaseInv}. By Proposition~\ref{prop:BaseInv} again we have $s_{\gamma}x\lhd x$. Since $sx\lhd x$ and $\gamma \not = \alpha_s$ we have $ss_\gamma x \lhd s_\gamma x$ by Corollary~\ref{cor:Lifting}, proving the first statement of $(\star)$. The second statement follows from the following computation:
$$
\ell(s_\beta sx)=\ell(ss_\gamma x)=\ell(s_\gamma x) -1=\ell(x)-2=\ell(sx)-1.
$$
This shows the right hand side of (1) is contained in the left hand side. For the reverse inclusion, let $\beta$ be in the left hand side of (1) and set $\gamma:=s(\beta)$. Since $\gamma\not = \alpha$, it will suffice to show that $\gamma\in N^1(x)$. But we have by assumption on $\beta$ that
$\ell(s_{\gamma}x)=\ell(ss_{\beta}sx)=1+\ell(s_{\beta}sx)=\ell(sx)=\ell(x)-1$, and the result follows.

\smallskip
Now for the proof of (2), the first equality follows from the discussion at the beginning of this proof.  The second and third equalities follows from  Lemma~\ref{lem:BruhatDi} and the definition of $f_s$, since there is $\gamma\in N^1(x)\setminus\{\alpha_s\}$, i.e., $s_\gamma x\lhd x$, and $\gamma$ is in the positive root system spanned by $\Delta_{W'}=\{\alpha_s,\beta\}$.
\end{proof}
\begin{proof}[Proof of Theorem~\ref{thm:BaseSuff}] Part~{\em (1)} is a direct consequence of Lemma~\ref{lem:Suff} since 
$$
N^1(sx) = \{\beta \in N^1(sx)\,|\, s_\beta sx > ss_\beta s x\}\sqcup\{\beta \in N^1(sx)\,|\, s_\beta sx < ss_\beta s x\}.
$$
Item~{\em (2)} follows easily  by replacing $x$ by $sx$ in Lemma~\ref{lem:Suff}(1).
\end{proof}
\begin{cor}\label{cor:Suff} Let $n\in \mathbb N$ and $x\in L_n(W)$ be a $n$-low element.  Let $s\in S$ such that $sx$ is a suffix of $x$. If $f_s\left(N^1(x)\setminus\{\alpha_s\}\right)\subseteq \Sigma_n$,  then $sx\in L_n(W)$.
\end{cor}
\begin{proof} We have, by definition of $n$-low elements, to show that $N^1(sx)\subseteq \Sigma_n$.  Let $\gamma\in N^1(sx)$. By Theorem~\ref{thm:BaseSuff}{\em (1)}, either $\gamma\in f_s\left(N^1(x)\setminus\{\alpha_s\}\right)$ or $\gamma=s(\beta)$ where $\beta\in N^1(x)\setminus\{\alpha_s\}$. In the first case, $\gamma\in \Sigma_n$ by assumption.  In the second case, if $\gamma\not\in \Sigma_{n}$,  we have
$d_{\infty}(\beta)\leq n$ and $d_{\infty}(s(\beta))>n$. By Proposition \ref{prop:NSmall}, this implies that $B(\alpha_s,\beta)\leq -1$. But $\alpha,\beta\in N^1(x)\subseteq N(x)$, so this contradicts Proposition~\ref{prop:FiniteRefl}{\em (2)}. 
\end{proof}

The next result is immediate from Corollary~\ref{cor:Suff} and Proposition \ref{prop:Low}{\em (3)}. 

\begin{cor} \label{cor:SuffMain} Let $n\in \mathbb N$. Suppose that  for every $s\in S$ and $\beta\in \Phi^{+}\setminus\{\alpha_s\}$,
one has $f_{s}(\beta)\in \Sigma_n$ whenever  $\beta\in \Sigma_n$. Then the set $L_n(W)$ of $n$-low elements of $W$  is a finite Garside shadow.
\end{cor}
%

\subsection{Bipodality}\label{sse:Bip} We present here a property called {\em bipodality} and show that if the set $\Sigma_n$ of $n$-small roots is bipodal, then $\Sigma_n$ meets the hypothesis of   Corollary~\ref{cor:SuffMain} and so $L_n(W)$ is a finite Garside shadow. 

\begin{defi} A subset $A\subseteq \Phi^+$ of positive roots is {\em bipodal} if for any maximal dihedral reflection subgroup $W'$  with canonical simple system $\Delta_{W'}$ and positive root system $\Phi^+_{W'}$ we have the following property:
$$
A\cap (\Phi_{W'}^+\setminus\Delta_{W'})\not = \emptyset \implies \Delta_{W'}\subseteq A.
$$
 \end{defi}
 
 \begin{remark}
 In the context of normalized roots, as in Remark~\ref{rem:Proj1}, $A$ is bipodal if for any maximal segment $[\h \alpha, \h\beta]=\h\Phi\cap P$ ($P$ is some line containing at least two normalized roots) such that $\h A\cap (\h \alpha, \h\beta)\not = \emptyset$ the endpoints $\h\alpha,\h\beta$ are elements of~$A$. 
 \end{remark}

\begin{prop}\label{prop:BipNSmall} Let $n\in \mathbb N$. If the set $\Sigma_n$ of $n$-small  roots  is bipodal, then the set $L_n(W)$ of $n$-low elements  is a finite Garside shadow. 
\end{prop}
\begin{proof} By Proposition~\ref{prop:Low}{\em (3)} it is enough to show that if $s\in S$ and $x\in L_n(W)$ such that $sx$ is a suffix of $x$, then $sx\in L_n(W)$. By Corollary~\ref{cor:SuffMain},  we have to show that if $\beta\in N^1(x)\setminus\{\alpha_s\}$, then $f_s(\beta)\in \Sigma_n$. Since $x\in L_n(W)$ and $\beta$ is in the base of its inversion set, we have $\beta\in \Sigma_n$. Let $W'$ be the dihedral reflection subgroup generated by the reflection in the simple system $\Delta_{W'}=\{\alpha_s,f_s(\beta)\}$; by definition of $f_s$, this reflection subgroup is maximal. So either $f_s(\beta)=\beta$ and we are done, or $\beta\in \Phi^+_{W'}\setminus \Delta_{W'}$. This forces $f_s(\beta)\in \Sigma_n$ since $\Sigma_n$ is bipodal by assumption. 
\end{proof}

 In regards of the above proposition, it is now time to state our second conjecture, that would imply, if true,  Conjecture~\ref{conj:1}.

\begin{conject}\label{conj:3} The set of $n$-small roots $\Sigma_n$ is bipodal for any $n\in\mathbb N$. 
\end{conject}

This conjecture is obvioulsy  true for finite Coxeter groups, since all roots are small (see Example~\ref{ex:Finite}); it is also true for infinite dihedral groups (see Example~\ref{ex:Inf2}) and  for affine Coxeter systems as shown in the following theorem. 

\begin{thm}\label{thm:AffineBip} Let $n\in\mathbb N$ and assume $(W,S)$ is an affine Weyl group. Then:
\begin{enumerate}
\item $\Sigma_n$  is bipodal;
\item $L_n(W)$ is a finite  Garside shadow in $(W,S)$.
\end{enumerate}
\end{thm}
\begin{proof} By Proposition~\ref{prop:BipNSmall} one just has to show {\em (1)}. 
Assume without loss of generality that $(W,S)$ is irreducible. We use the notations and results of Example~\ref{ex:Affine}. Let $\gamma$ be a root in the positive crystallographic root system $\Psi$ such that $\gamma'\in\Sigma_n$. Let  $W'$ be a maximal  dihedral reflection subgroup with simple system $\Delta_{W'}=\{\alpha',\beta'\}$, with $\alpha',\beta'\in\Psi^+$, and positive root subsystem $\Phi_{W'}^{+}$ such that $\gamma'\in\Phi_{W'}^+\setminus\Delta_{W'}$.  By symmetry between $\alpha$ and $\beta$, one just has to show that $\dep_{\infty}(\alpha)\leq \dep_\infty(\gamma)$, which would imply $\alpha'\in\Sigma_n$.

Since $\Phi_{W'}^+=\cone_\Phi(\Delta_{W'})$, $\Psi$ is crystallographic and $\gamma'\notin\Delta_{W'}$, there is $a,b\in\mathbb N^*$ such that $\gamma=a\alpha+b\beta$. By the description of $\Psi^+$ given in Example~\ref{ex:Affine} we have $\alpha=\alpha_0+k\delta$ and $\beta=\beta_0+l\delta$, for some $k,l\in\mathbb N$ and $\alpha_0,\beta_0\in\Psi_0$. So
 $$
 \gamma=a\alpha_0+b\beta_0 + (ak+bl)\delta.
 $$ 
 By the description of $\Psi^+$ again, we must have $\gamma_0:=a\alpha_0+b\beta_0\in \Psi_0$ since $\delta$ is not in the linear span of $\Psi_0$ but $\gamma_0=a\alpha_0+b\beta_0$ is.  So by Equation~$(\Diamond)$ in Example~\ref{ex:Affine} we have:
$$
\dep_{\infty}(\gamma)=\dep_{\infty}(\gamma_0+(ak+bl)\delta)=
\left\{
\begin{array}{ll}
ak+bl&\textrm{if $\gamma_0\in \Psi_0^{+}$}\\
ak+bl-1&\textrm{if $\gamma_0\in \Psi_0^{-}$}
\end{array}
\right. .
$$
and
$$
\dep_{\infty}(\alpha)=\dep_{\infty}(\alpha_0+n\delta)=
\left\{
\begin{array}{ll}
k&\textrm{if $\alpha_0\in \Psi_0^{+}$}\\
k-1&\textrm{if $\alpha_0\in \Psi_0^{-}$}
\end{array}
\right. .
$$
So $\dep_\infty(\alpha)\leq \dep_\infty(\gamma)$ unless $\gamma_0\in\Psi_0^-$, $\alpha_0\in\Psi^+_0$ and $l=0$. In this case $\beta=\beta_0$ must be a positive root in $\Psi_0^+$ and therefore $\gamma_0=a\alpha_0+b\beta_0\in\Psi_0^+$ contradicting ~$\gamma_0\in\Psi_0^-$. In conclusion, $\dep_\infty(\alpha)\leq \dep_\infty(\gamma)$, $\alpha'\in \Sigma_n$ and $\Sigma_n$ is bipodal.
\end{proof}

\subsection{Bipodality of small roots and proof of Theorem~\ref{thm:Main1}}\label{sse:ProofMain1} As seen in Proposition~\ref{prop:BipNSmall}, in order to show Theorem~\ref{thm:Main1} it is enough to show that the set of small roots $\Sigma$ is bipodal.

 \begin{thm}\label{thm:Main3} The set of small roots $\Sigma$ is bipodal. 
\end{thm}

This theorem is the direct consequence of the  three following technical  lemmas

\begin{lem} \label{lem:Bip1}The set $\Sigma$ is bipodal if and only if the following condition  holds:
\begin{enumerate}
\item[$\textnormal{($\heartsuit$)}$] for any  maximal dihedral reflection subgroup $W'$ of $W$, $\gamma\in \Sigma \cap (\Phi_{W'}^{+}\setminus \Delta_{W'})$ and simple root $\alpha\in \Delta\setminus \Delta_{W'}$ with $B(\alpha,\gamma)>0$, one has $-1<B(\alpha,\beta)<1$ for all $\beta\in \Delta_{W'}$.
\end{enumerate}
\end{lem}
\begin{remark} Observe that the condition $B(\alpha,\gamma)>0$ is equivalent to say that $\gamma$ is on the positive side of the hyperplane $H_\alpha=\ker(v\mapsto B(\alpha,v))$ and that the condition $-1<B(\alpha,\beta)<1$ is equivalent to say that the dihedral reflection subgroup generated by $s_\alpha,s_\beta$ is finite, which is equivalent to $(\mathbb{R} \alpha+\mathbb{R}\beta)\cap \Phi$ is finite,  or equivalently,  $(\mathbb{R} \alpha+\mathbb{R}\beta)\cap Q=\{0\}$; see \cite{DyHoRi13,HoLaRi14} for more details. So the condition $(\heartsuit)$ has the following geometric interpretation: for any  small root $\gamma\in \Sigma$, simple root $\alpha\in \Delta$ such that $\gamma$ is on the positive side of the hyperplane $H_\alpha$ and  maximal dihedral reflection subgroup $W'$ of $W$ such that $\gamma\in (\Phi_{W'}^{+}\setminus \Delta_{W'})$, 
one has  $(\mathbb{R} \alpha+\mathbb{R}\beta)\cap \Phi$  finite  (which translates in the language of normalized roots to: the line passing through $\h\alpha$ and $\h\beta$ contains a finite number of normalized roots).
\end{remark}
\begin{proof} Suppose $\Sigma $ is bipodal. Let  $W'$ be  a  maximal dihedral reflection subgroup  of $W$,  and  $\gamma\in \Sigma \cap (\Phi_{W'}^{+}\setminus \Pi_{W'})$. Since $\Sigma $ is bipodal, $\Delta_{W'}\subseteq \Sigma $. Now let $\alpha\in \Delta\setminus \Delta_{W'}$ with $B(\alpha,\gamma)>0$. Abbreviate $s=s_{\alpha}$ and  $W'':=sW's$ (a maximal dihedral reflection subgroup of $W$).  
Note that $\alpha\not \in \Phi_{W'}$, for otherwise  $\alpha\in \Pi_{W'}$ since $\alpha\in \Delta$.
Hence   $\Delta_{W''}=s(\Delta_{W'})$ and $s(\gamma)\in s(\Phi_{W'}^+)=\Phi_{W''}^+\subseteq \Phi^{+}$. Proposition~\ref{prop:NSmall} forces  $0\leq \dep_{\infty}(s(\gamma))\leq \dep_{\infty}(\gamma)=0$, so $s(\gamma)\in \Sigma \cap (\Phi_{W''}^{+}\setminus \Delta_{W''})$. Since $\Sigma $ is bipodal, this implies $\Delta_{W''}\subseteq \Sigma$.
 Now let $\beta\in \Delta_{W'}$, so $s(\beta)\in \Delta_{W''}$ and $\dep_{\infty}(\beta)=\dep_{\infty}(s(\beta))=0$ by above. By Proposition~\ref{prop:NSmall}: if $B(\alpha,\beta)\leq -1$, then $\dep_{\infty}(s(\beta))=\dep_{\infty }(\beta)+1$, whereas if  $B(\alpha,\beta)\geq 1$, then $\dep_{\infty}(\beta)=\dep_{\infty }(s(\beta))+1$.   In either case, we have  a contradiction, so $-1<B(\alpha,\beta)<1$.  Hence  ($\heartsuit$) holds, completing the proof of the ``only if'' direction.

\smallskip
Now we assume  ($\heartsuit$). We  show  that for any maximal dihedral reflection subgroup $W'$ of $W$ and any $\gamma\in \Sigma \cap (\Phi_{W'}^{+}\setminus \Delta_{W'})$, one has $\Delta_{W'}\subseteq \Sigma$, by induction on $\dep(\gamma)$ using Proposition~\ref{prop:Dep}. This will obviously imply that $\Sigma$ is bipodal.

Choose $\alpha\in \Delta$ with $B(\alpha,\gamma)>0$.  Since $B(\alpha,\beta)\leq 0$ for all $\alpha,\beta\in \Delta$, we must have $\gamma\not\in \Delta$. So $\dep(\gamma)\geq 3$ by Proposition~\ref{prop:Dep}. Set as above $s=s_{\alpha}$ and  $W'':=sW's$.  By Proposition~\ref{prop:Dep},   we have $\dep(s(\gamma))=\dep(\gamma)-1$, so  $s(\gamma)\in \Sigma$ by Proposition~\ref{prop:SmallMin}.  We distinguish the cases $\alpha\in \Delta_{W'}$ and $\alpha\not\in \Delta_{W'}$. 

First assume $\alpha\in \Delta_{W''}$, so $W''=W'$. If $s(\gamma)\in \Delta_{W'}$, then $\Delta_{W'}=\{\alpha, \gamma\}\subseteq \Sigma $ as required. On the other hand, if $s(\gamma)\not\in \Delta_{W'}$, then $s(\gamma)\in \Sigma \cap (\Phi_{W'}^{+}\setminus \Delta_{W'})$ with $\dep(s(\gamma))<\dep(\gamma)$. Hence   $\Delta_{W'}\subseteq \Sigma $ by induction. 

Now consider the case $\alpha\not\in \Delta_{W''}$. Then $\Delta_{W''}=s(\Delta_{W'})$,
  $\Phi_{W''}^{+}=s(\Phi_{W'}^{+})$ and  $s(\gamma)\in \Sigma \cap (\Phi_{W''}^{+}\setminus \Delta_{W''})$ with $\dep(s(\gamma))<\dep(\gamma)$. By induction, $\Delta_{W''}\subseteq \Sigma $. But by ($\heartsuit$), for any $\beta\in \Delta_{W'}$, we have 
 $-1<B(\alpha,\beta)<1$ and hence $\dep_{\infty}(\beta)=d_{\infty}(s_{\alpha}(\beta))=0$ since $s_{\alpha}(\beta)\in \Delta_{W''}$. Hence $\Delta_{W'}\subseteq \Sigma$ as required to complete the proof of the ``if'' direction.
\end{proof}

\begin{lem} \label{lem:Bip2} Assume that $\Sigma$ is bipodal whenever $(W,S)$ is of rank $\vert S\vert =3$.  Then $\Sigma $ is bipodal for all $(W,S)$.
\end{lem}
\begin{proof} We regard  the condition ($\heartsuit$) in Lemma~\ref{lem:Bip1}  as a condition on the based root system $(\Phi,\Delta)$:  
assume that ($\heartsuit$) holds for each rank three root subsystem $(\Psi,\Pi)$ of $(\Phi,\Delta)$.  Let $W'$ be any rank $2$ maximal dihedral reflection subgroup  of~$W$, $\gamma\in \Sigma \cap (\Phi_{W'}^{+}\setminus \Delta_{W'})$ and $\alpha\in \Delta\setminus \Delta_{W'}$ with $B(\alpha,\gamma)>0$ and $\beta\in \Delta_{W'}$.  We have to show that $-1<B(\alpha,\beta)<1$. 

Define the reflection subgroup  $G:=\langle W',s_{\alpha}\rangle$ of $W$. Since $\Phi_{W'}\sqcup \{\alpha\}$ spans a subspace of $V$ of dimension three and $\Phi_{W'}\sqcup \{\alpha\}\subseteq \Phi_G$, the rank of $G$ has to be at least $3$. Since $G$  is generated by three reflections, it is of rank at most three by~\cite[Corollary~3.11]{Dy90}. So the rank of $G$ is $3$. 

Set $\Psi:=\Phi_{G}$ and let $\Pi:=\Delta_{G}$, so $(\Psi, \Pi)$ is a rank three based root subsystem of $(\Phi,\Delta)$
(but \emph{not} a standard parabolic root subsystem in general).  Note that $W'$ is a maximal dihedral reflection subgroup of $G$ and that  
 the set $\Psi^{+}_{W'}$ of positive roots of $W'$ in $\Psi$ is   $\Psi^{+}_{W'}=\Phi^{+}_{W'}$ and  the canonical  set of simple roots 
 $\Delta_{W'}$ of  $\Psi^{+}_{W'}$ with respect to $(\Psi,\Pi)$ is $\Pi_{W'}=\Delta_{W'}$. Further, 
 $\gamma\in \Sigma(G)\cap (\Psi_{W'}^{+}\setminus \Pi_{W'})$ by Lemma~\ref{lem:Domin}, and   $\alpha\in \Pi\setminus \Pi_{W'}$ since  any simple root lies in the canonical simple system  
 of any root subsystem containing it.  Hence by  ($\heartsuit$) for $(\Psi,\Pi)$, which holds by assumption since $\Psi$ is of rank three,  it follows  that $-1<B(\alpha,\beta)<1$ for $\beta\in \Pi_{W'}= \Delta_{W'}$.
 \end{proof}

\begin{lem} \label{lem:Bip3} If $(W,S)$ has rank $3$, then  $\Sigma $ is bipodal.
\end{lem}
\begin{proof}  Assume that $(W,S)$ is of rank $3$. Let $\Gamma$ be the Coxeter graph of $(W,S)$, with set of vertices $\Delta=\{\alpha_1,\alpha_2,\alpha_3\}$. We denote $S=\{s_1,s_2,s_3\}$ with $s_i=s_{\alpha_i}$. The {\em support of a root} $\beta=a_1\alpha_1+a_2\alpha_2+a_3\alpha_3$ is  the  set  $\supp(\beta)=\{\alpha_i\in \Delta\,|\, a_{i}\not = 0\}$. We denote by
 $\Gamma(\beta)$ the full subgraph of $\Gamma$ on the vertex set $\supp(\beta)$.
 \smallskip

 The proof proceeds essentially by systematically listing all possible $(\Gamma,\gamma,W')$ satisfying the  conditions
 $\gamma\in \Sigma$ and  $W'$ is a maximal dihedral reflection subgroup such that $\gamma\in  \Phi^{+}_{W'}\setminus \Delta_{W'}$ 
  and checking that $\Delta_{W'}\subseteq \Sigma$ in each case.
\smallskip

If $\Gamma$ is not connected, i.e.,  $\Phi$ is reducible, then   $\Phi$ is the disjoint union of a simple root and of the root system of a dihedral standard parabolic subgroup, for which the set of small roots is bipodal. It is then easy to  see that $\Sigma$ is bipodal.  

\smallskip
So from now on, we suppose that  $\Gamma$ is a connected graph.

\smallskip

Assume that $\gamma$ is not of full support, i.e.,  $\supp(\gamma)=I\subsetneq \Delta$.  So $\gamma\in \Phi_I^{+}$ and therefore $\gamma$ is in a facet of $\cone(\Delta)$.  Since $\gamma\in  \Phi^{+}_{W'}\setminus \Delta_{W'}$, then $\gamma$ is in the relative interior of $\cone(\Delta_{W'})$. This forces  $\Delta_{W'}\subseteq \Phi_{W'}^{+}\subseteq \Phi_I^{+}$, since $\Delta_{W'} \subseteq \cone(\Delta_{W'})\subseteq \cone(\Delta)$. So we are in the case of a dihedral standard parabolic subgroup. But in this case we know that $\Delta_{W'}\subseteq \Sigma(W_I)$. So $\Delta_{W'}\subseteq \Sigma$ by Proposition~\ref{prop:NSmallParab}.

\smallskip
So from now on $\gamma$ is assumed to be of full support, i.e., $\Gamma(\gamma)=\Gamma$

\smallskip

By Brink's characterization of the support of small roots~\cite[Lemma~4.1]{Br98}, the support of any small  root contains no cycle and no edge with infinite label.  We may therefore  assume that the Coxeter graph is  of the form
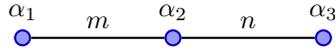
\begin{center}
 \begin{tikzpicture}[sommet/.style={inner sep=2pt,circle,draw=blue!75!black,fill=blue!40,thick},]
\coordinate (ancre) at (0,0);
\node[sommet,label=above:$\alpha_1$] (alpha) at (ancre) {};
\node[sommet,label=above:$\alpha_2$] (beta) at ($(ancre)+(2,0)$) {} edge[thick] node[auto,swap] {$m$} (alpha);
\node[sommet,label=above:$\alpha_3$] (gamma) at ($(ancre)+(4,0)$) {}  edge[thick] node[auto,swap] {$n$} (beta);
\end{tikzpicture}
\end{center}
where $m,n\in \mathbb N$ and $m\geq n\geq 3$. 

\smallskip
We already know $\Sigma $ is bipodal if $W$ is finite (by Example~\ref{ex:Finite}) or affine (by Theorem~\ref{thm:AffineBip}). So the classification of affine and finite Coxeter groups (see for instance \cite{Hu90}) forces  $m> n\geq 4$, or $m\geq 7$ and $n=3$.
\smallskip

 It is easy now to list for each of these $\Gamma$ the small  roots $\gamma$ of full support. This may be done 
 either by simple direct calculation using Proposition~\ref{prop:NSmall} or by using~\cite{Br98}, where the small  roots are recursively determined for all finite rank Coxeter systems; see especially \cite[Propositions 4.7 and 6.7]{Br98}.  For each $\gamma$, all possible maximal dihedral reflection subgroups $W'$ are obtained by specifying the canonical simple system $\Delta_{W}=\{\mu_1,\mu_2\}$, which is obtained by inspection since $\gamma=a\mu_1+b\mu_2$ with $a,b\geq 1$. Observe that $\supp(\mu_i)\not = \{\alpha_1,\alpha_3\}$ since there is no root with  full  support in the subgroup generated by $s_1,s_3$.

 To complete the proof, we list below all the possible $\gamma$ and $W'$ (by specifying $\Delta_{W'}$), writing $c_{p}:=2\cos\frac{\pi}{p}$ for all $p\in \mathbb N_{\geq 2}$. In each case, each element of $\Delta_{W'}$ lies in a dihedral finite standard parabolic root subsystem of $\Phi$ and so is small  as required.
\begin{itemize}
\item If $m> n\geq 4$, the only possible $\gamma$ is $c_{m}\alpha_{1}+\alpha_{2}+c_{n}\alpha_{3}=s_{1}s_{3}(\alpha_{2})$ and then $\Delta_{W'}$ is either $\{\alpha_{1},s_{3}(\alpha_{2})=\alpha_2+c_n\alpha_3\}$ or  $\{\alpha_{3}, s_{1}(\alpha_{2})=\alpha_2+c_m\alpha_1\}$.
\item If $m\geq 7$ and $n=3$, there are three possible choices for $\gamma$. 
\begin{enumerate}[{\em (i)}]
\item  $\gamma=\alpha_{1}+c_{m}\alpha_{2}+c_{m}\alpha_{3}=s_{3}s_{2}(\alpha_{1})$; in this case $\Delta_{W'}$ is:
$$
\textrm{either $\{\alpha_{3},s_{2}(\alpha_{1})=\alpha_{1}+c_{m}\alpha_{2}\}$ or $\{\alpha_{1},s_{3}(\alpha_{2})  =\alpha_2+\alpha_3\}$.}
$$

\item  $\gamma=c_{m}\alpha_{1}+\alpha_{2}+ \alpha_{3}=s_{1}s_{2}(\alpha_{3})$; in this case $\Delta_{W'}$ is:
$$
\textrm{either $\{\alpha_{3},s_{1}(\alpha_{2})=\alpha_2+c_m\alpha_1\}$ or $\{\alpha_{1},s_{2}(\alpha_{3}) =\alpha_{2}+\alpha_{3}\}$.}
$$

\item  $\gamma=(c_{m}^{2}-1)\alpha_{1}+c_{m}\alpha_{2}+ c_{m}\alpha_{3}=s_{1}s_{3}s_{2}(\alpha_{1})$;  in this case $\Delta_{W'}$ is:
$$
\textrm{either $\{\alpha_{3},s_{1}s_{2}(\alpha_{1})=(c_m^2-1)\alpha_1+c_m\alpha_2\}$ or $\{\alpha_{1},s_{3}(\alpha_{2})=\alpha_2+\alpha_3\}$.}
$$
\end{enumerate}
\end{itemize}
\end{proof}

\section{Weak and Bruhat order on root systems and bipodality}\label{se:BruhatRoot}

Let $(\Phi,\Delta)$ be a based root system in the quadratic space $(V,B)$ with Coxeter system $(W,S)$. In the preceding sections we conjectured that:

\begin{itemize}
\item The set $L_n(W)$ of $n$-low elements in $W$ is a finite Garside shadow in $(W,S)$ for all $n\in \mathbb N$ (Conjecture~\ref{conj:1}).  

\item The set $\Sigma_n$ of $n$-small roots is bipodal for all $n\in \mathbb N$ (Conjecture~\ref{conj:3}).
\end{itemize}
In view of Proposition~\ref{prop:BipNSmall},  we know that Conjecture~\ref{conj:1} is implied by Conjecture~\ref{conj:3}. We know therefore that Conjectures~\ref{conj:1}~and~\ref{conj:3} are true in the following cases:
\begin{itemize}
\item if $n=0$ by Theorem~\ref{thm:Main1}~and Theorem~\ref{thm:Main3}; 
\item if $W$ is finite, dihedral or an affine Weyl group by Theorem~\ref{thm:AffineBip} and Example~\ref{ex:Finite}~\and~\ref{ex:Inf2};
\item if the Coxeter graph of $(W,S)$ has labels $3$ or $\infty$, by Theorem~\ref{thm:W3Infty}  in~\S\ref{se:BruhatRoot}. 
\end{itemize}

The aim of this section is to show Conjecture~\ref{conj:1} and Conjecture~\ref{conj:3} hold also in the following cases. 

\begin{thm}\label{thm:W3Infty}  Suppose that $(W,S)$ is a Coxeter system such that its Coxeter graph has all  edges labelled by $3$ or $\infty$, i.e.,  all entries of the Coxeter matrix of $(W,S)$ lie in $\{1,2,3,\infty\}$. Then for each $n\in \mathbb N$: 
\begin{enumerate}
\item $\Sigma_n(W)$ is a bipodal subset of $\Phi^{+}$;
\item $L_n(W)$ is a finite Garside shadow in $(W,S)$.  
\end{enumerate}
\end{thm}
In order to show bipodality, we  will show that $\Sigma_n$ enjoys a stronger property that we define now.

\begin{defi} A subset $A\subseteq\Phi^+$ is {\em balanced} if  for all $\gamma \in A$ and all maximal dihedral reflection subgroups $W'$ of $W$ with $\gamma\in \Phi_{W'}$ the following holds: if $\beta\in \Phi_{W'}^+$ such that $\ell_{W'}(s_\beta)< \ell_{W'}(s_\gamma)$ then $\beta\in A$.
\end{defi}

This definition goes back to Edgar's thesis~\cite{Ed09} (in  the case of the  standard length function). Since the canonical simple reflections $s_\alpha,s_\beta$ of $W'$ are of length $\ell_{W'}(s_\alpha)=\ell_{W'}(s_\beta)=1$, it is easy to see that a balanced set of root is necessarily a bipodal set of roots. 

 \begin{remark}\label{rem:balance}
\begin{enumerate}[(a)]
\item In the language of normalized roots,   as in~Remark~\ref{rem:Proj1},  a maximal dihedral subgroup $W'$ with canonical simple system $\Delta_{W'}=\{\alpha,\beta\}$ corresponds to  a ``maximal'' segment $[\h\alpha,\h\beta]\cap \h\Phi$; maximality holds  in the sense that  there are no other roots on the line $(\h\alpha,\h\beta)$ but the ones in $[\h\alpha,\h\beta]\cap \h\Phi =\h\Phi_{W'}$.  If $W'$ is finite (infinite), one may index $\Phi_{W'}^{+}$ as
$\{ \alpha_{1}=\beta_{n},\alpha_{2}=\beta_{n-1}, \ldots, \alpha_{n}=\beta_{1}\}$ (resp.,
$\{\alpha_{1},\alpha_{2},\ldots, \alpha_{n},\ldots,  \beta_{n},\ldots, \beta_{2},\beta_{1}\}$; see Figure~\ref{fig:Dihedral}) in the order corresponding to that in which  a  point moving along the line segment  $[\h \alpha,\h\beta]$ from $\h\alpha$ to $\h\beta$ passes through the associated normalized roots. One has $l_{W'}(\alpha_{i})=l_{W'}(\beta_{i})=2i-1$ provided $1\leq 2i-1\leq \vert \Phi_{W'}^{+}\vert$.  
Set $a_{i}:=\dep_{\infty}(\alpha_{i})$ and $b_{i}=\dep_{\infty} (\beta_{i})$. 
Then 
$\Sigma_{n}$ is balanced if and only if  (for each  $W'$), $1\leq 2i-1\leq \vert \Phi_{W'}^{+}\vert$ and $\min(a_{i},b_{i})\leq n$ imply that $a_{j},b_{j}\leq n$ for all $j$ with $1\leq j<i$.

\item It is not true that a bipodal set of roots is balanced. Assume $(W,S)$ to be of type $\tilde G_2$ as in Figure~\ref{fig:AffineG2}, then $\Sigma$ is bipodal but not balanced. For the maximal (finite) segment $[\h\gamma,s_\beta\cdot \h\alpha]\cap \Phi$ has $\infty$-depth values $(0,0,1,0,1,0)$, where, in the notation of (a), $b_{3}=0$ but $b_{2}=1\not \leq 0$.
 Hence the conjecture, mentioned in \cite{Ed09}, that $\Sigma_{n}$ is always balanced is false (even for $n=0$).
\end{enumerate}
\end{remark}

Theorem~\ref{thm:W3Infty} is  a consequence of the following proposition, as we show now.

 \begin{prop} \label{prop:Increase} Let $W'$ be a dihedral reflection subgroup of $W$, and $\gamma,\delta\in \Phi^{+}_{W'}$.\begin{enumerate}
\item If $W'$ is infinite and  $\ell_{W'}(s_{\gamma})<\ell_{W'}(s_{\delta})$, then $\dep_{\infty}(\gamma)<\dep_{\infty}(\delta)$.
\item If $W'$ is finite and there exists $x\in W'$ such that $\delta=x(\gamma)$ and $\ell(s_{\delta})=\ell(s_{\gamma})+2\ell_{W'}(x)$, then $\dep_{\infty}(\gamma)\leq \dep_{\infty}(\delta)$.\end{enumerate}
\end{prop}
\begin{remark}\label{rem:Increase}  In the notation of Remark \ref{rem:balance}, (1) says that
(for infinite $W'$), $a_{j},b_{j}<a_{i},b_{i}$ if $1\leq j<i$, while (2) says that (for finite $W$),
$a_{i-1}<b_{i}$ and $b_{i-1}<a_{i}$ if $3\leq2i-1\leq \vert \Phi_{W'}^{+}\vert$. 
\end{remark}

 The proof of the proposition involves a  study of an analogue for $\Phi$ of the weak  order of $W$, and is postponed to the end of this section. 
 For now,  assuming the above proposition, let us give the proof of the theorem.

\begin{proof}[Proof of Theorem~\ref{thm:W3Infty}] We shall show that $\Sigma_n$ is balanced, and hence it is bipodal. We begin the argument assuming just that $(W,S)$ is a finite rank Coxeter system.  Let $\gamma\in \Sigma_n$ and $W'$ be a (maximal) dihedral reflection subgroup of $W$ such that $\gamma\in \Phi_{W'}\setminus \Delta_{W'}$.  Write $\Delta_{W'}=\{\alpha,\beta\}$. We wish to show that $\alpha,\beta\in \Sigma_n$.

In any case, $\ell_{W'}(s_{\alpha}),\ell_{W'}(s_{\beta})<\ell_{W'}(s_{\gamma})$. If $W'$ is infinite, it follows from Proposition~\ref{prop:Increase}(1) that $\dep_{\infty}(\alpha), \dep_{\infty}(\beta)<\dep_{\infty}(\gamma)\leq n$ and $\alpha,\beta\in \Sigma_n$ as required.
 Otherwise, $W'$ is finite.  It follows from Proposition~\ref{prop:Increase}(2) and Remark~\ref{rem:Increase} that $\dep_\infty(s_{\beta}(\alpha))\geq \dep_\infty(s_{\alpha})$, 
 $\dep_\infty(s_{\alpha}(\beta))\geq \dep_\infty(s_{\beta})$ and either 
 $\dep_\infty(s_{\gamma})\geq \dep_\infty(s_{\beta}(\alpha))$ or $\dep_\infty(s_{\gamma})\geq\dep_\infty(s_{\alpha}(\beta))$. Hence it would suffice to show that $\dep_\infty(s_{\beta}(\alpha))\geq \dep_\infty(s_{\beta})$ and $\dep_\infty(s_{\alpha}(\beta))\geq \dep_\infty(s_{\alpha})$. We do not know if this holds in general. However,  we now make the special assumptions on the Coxeter matrix  in Theorem~\ref{thm:W3Infty}  and   show instead that    one has $s_{\gamma}=s_{\alpha}(\beta)=s_{\beta}(\alpha)$.

 An argument due to Tits (see \cite[Ch IV, \S4, Ex 4(d)]{Bo68}, \cite[Proposition 1.3]{BrHo93} or \cite[Theorem 4.5.3]{BjBr05}) shows that a finite subgroup $H$ of $W$  must be conjugate to a reflection subgroup of a finite standard parabolic subgroup of $W$. A minor elaboration of Tits' argument shows that if $H$ is a reflection subgroup, and its roots span a subspace of dimension $r$, then the standard parabolic subgroup may be taken to be of rank $r$. This also follows from the well known  fact that, in a finite Coxeter group,   the parabolic closure $H''$  of a rank $r$ reflection subgroup $H'$ (i.e. the inclusion-minimal parabolic subgroup $H''$ containing $H'$) has as its  root system the set of all roots in the linear span of the root system of $H'$. 

  Applying the previous paragraph with $H=W'$ shows that  $W'$ is conjugate to a subgroup of a finite dihedral standard parabolic subgroup of $W$. By assumption, this subgroup must be of order $4$ or $6$ with two or three positive roots respectively. In particular, either $B(\alpha,\beta)=0$ (which is impossible since then $\Phi_{W'}\setminus\Delta_{W'}=\emptyset$ contradicting $\gamma\in\Phi_{W'}\setminus\Delta_{W'}$) or $B(\alpha,\beta)=-\frac{1}{2}$ and $\gamma=\alpha+\beta=s_{\alpha}(\beta)=s_{\beta}(\alpha)$ as required.
\end{proof} 

Proposition~\ref{prop:Increase} has also the following noteworthy consequence, which was observed in the proof above.

\begin{cor}\label{cor:Increase} Let $n\in \mathbb N$.  The set $\Sigma_n$ of $n$-small  roots  is bipodal  if and only for any pair of  roots $\alpha,\beta\in \Phi^{+}$ such that $W':=\langle s_{\alpha},s_{\beta}\rangle$ is a finite  maximal dihedral reflection subgroup with canonical simple system $\Delta_{W'}=\{\alpha,\beta\}$, $B(\alpha,\beta)\not = 0$, and  $s_{\alpha}(\beta)\in\Sigma_n$, one has $\alpha\in \Sigma_n$. 
\end{cor}

\subsection{Length and depth functions on $\Phi$}\label{sse:Length}  

\begin{defi} Fix any subset $X\subseteq \mathbb R^+:=\{\lambda\in \mathbb R\mid \lambda >0\}$. The {\em $X$-length} on $\Phi$ is the function $d_X:\Phi\to \mathbb Z$ defined as follows:  let  $\beta\in\Phi^{+}$, then  
\begin{equation*}
d_{X}( \pm\beta)=\pm\vert \{\alpha\in N(s_{\beta})\,|\, B(\alpha,\beta)\in X\}\vert.
\end{equation*} 
\end{defi}

\begin{ex}\label{ex:Length1}
Note that if $\beta\in \Phi^{+}  $ and $\alpha\in N(s_{\beta})$, then $B(\alpha,\beta)>0$, since otherwise $s_{\beta}(\alpha)=\alpha-2B(\alpha,\beta) \beta\in \Phi^{+}$ contrary to $\alpha\in N(s_{\beta})$. So if $X=\mathbb R^+$ the $\mathbb R^+$-length on $\beta$ corresponds to the usual length of the associated reflection $s_\beta$:
$$
d_{\mathbb R^+}(\beta)=|N(s_\beta)|=\ell(s_\beta)=2\dep(\beta)-1,\qquad \beta\in \Phi^{+}.
$$
\end{ex}

\begin{remark}\label{rem:Ideal}
\begin{enumerate}[(a)]
\item For a reflection subgroup $W'$ of $W$,  $d_{W',X}\colon \Phi_{W'}\to \mathbb N$ denotes the function attached to $(\Phi_{W'},\Delta_{W'},X)$ in the same way as $d_{X}$ is attached to $(\Phi,\Delta,X)$.

\item  
The function $d_{X}$ depends only on $X\cap \{B(\gamma,\beta)\,|\, \gamma,\beta\in \Phi\}$. Also, $d_{\sqcup_{i} X_{i}}(\alpha)=\sum_{i }d_{X_{i}}(\alpha)$ (where the right hand side has only finitely many non-zero terms for fixed $\alpha\in \Phi$) if the subsets $X_{i}\subseteq \mathbb R^+ $ for $i\in I$ are pairwise disjoint. It follows that  the functions $d_{X}$ for arbitrary subsets of $\mathbb R$ are determined by the functions $d_{X}$ for order coideals of $\mathbb R$: such an order coideal $X$ is either $\emptyset$, $\mathbb R^+ $ or an open or closed ray $[a,\infty[$ or $]a,\infty[$ for some $a>0$.
\end{enumerate}
\end{remark}

\begin{prop} \label{1.6} Let $X\subseteq \mathbb R^+$ and set $-X:=\{-c\,|\, c\in X\}$. Let $\beta\in \Phi$ and $\alpha\in \Delta$. Then $d_{X}(\alpha)$ is equal to $1$ if $1$ is in $X$, and to $0$ otherwise. Further, 
\begin{equation*}
d_{X}(s_{\alpha}(\beta))=
\begin{cases}
d_{X}(\beta)-2,&\text{if $B(\alpha,\beta)\in X$}\\
d_{X}(\beta),&\text{if $B(\alpha,\beta)\in \mathbb R\setminus(X\cup -X)$}\\
d_{X}(\beta)+2,&\text{if $B(\alpha,\beta)\in -X$.}
\end{cases} 
\end{equation*} 
\end{prop}
\begin{proof} The first claim is immediate from the definitions. It is easy to check that the displayed formula  is equivalent to any of its variants given by replacing $\beta$ by $\pm \beta$ or $\pm s_{\alpha}(\beta)$.  Hence, without loss of generality, we may (and do) assume that $\beta\in \Phi^{+}$ and $B(\beta,\alpha)\geq 0$.  If $\beta=\alpha$ or $B(\beta,\alpha)=0$, the conclusion follows readily. Otherwise, $\gamma:=s_{\alpha}(\beta)\in \Phi^{+} $ and $s_{\beta}=s_{\alpha}s_{\gamma}s_{\alpha}$ with $\ell(s_{\beta})=\ell(s_{\gamma})+2$.  Hence  $N(s_\beta)=\{\alpha\}\sqcup s_{\alpha}(N(s_{\gamma}))\sqcup \{s_{\alpha}s_{\gamma}(\alpha)\}$ by Proposition~\ref{prop:Weak}.  One has
$$
 B(s_{\alpha}s_{\gamma}(\alpha),\beta)=B(s_{\alpha}s_{\gamma}s_{\alpha}(\alpha),-\beta)=B(s_{\beta}(\alpha),-\beta)=B(\alpha,\beta).
 $$  
 For any $\gamma'\in N(s_{\gamma})$ we have $B(\gamma',\gamma)\in X$ if and only if $B(s_\alpha(\gamma'),\beta)\in X$, since $$
 B(s_{\alpha}(\gamma'),\beta)=B(\gamma',s_{\alpha}(\beta))=B(\gamma',\gamma).
 $$
  Let $c$ denote $1$ if $B(\alpha,\beta)\in X$ and $0$ otherwise.  Now, since $B(\alpha,\beta)>0$,  the desired conclusion follows: \begin{equation*}\begin{split}
 d_{X}(\beta)&=\vert \{\rho\in N(s_{\beta})\mid B(\rho,\beta)\in X\}\vert =\vert \{\gamma'\in N(s_{\gamma})\mid B(s_{\alpha}(\gamma'),\beta)\in X\}\vert +2c\\ &=\vert \{\gamma'\in N(s_{\gamma})\mid B(\gamma',\gamma)\in X\}\vert +2c=d_{X}(\gamma)+2c.\end{split}\end{equation*}
  \end{proof}

\begin{cor}
\begin{enumerate}
\item Let  $w\in W$ and $\beta\in \Phi$. Then  $d_{X}(w(\beta))$ and $d_{X}(\beta) $ are integers of  the same parity, and $-2\ell(w)\leq d_{X}(w(\beta))-d_{X}(\beta)\leq 2\ell(w)$
\item $d_{X}(\beta)$ is  an odd integer or  an even integer for all $\beta\in \Phi$ according as to whether $1\in X$ or $1\not \in X$.
\end{enumerate} 
\end{cor}
\begin{proof} Part (1)  is proved by induction on $\ell(w)$ using Proposition~\ref{1.6}. Part (2) follows from (1) and the description of $d_{X}(\alpha)$ for $\alpha\in \Delta$ in Propositilon~\ref{1.6},  since every root  $\beta$ is in the $W$-orbit of some simple root $\alpha$. 
\end{proof}

Thanks to Proposition~\ref{1.6}, we can define $X$-depth on $\Phi$.

 \begin{defi}\label{1.6a} Let $X\subseteq \mathbb R^+$, the {$X$-depth on $\Phi^{+}$} is the function $\dep_X:\Phi^{+}\to \mathbb N$  defined as follows: 
 \begin{equation*}
\dep_{X}(\beta)=
\begin{cases}
\frac{d_{X}(\beta)-1}{2},&\text{if $1\in X$}\\
\frac{d_{X}(\beta)}{2},&\text{if $1\notin X$.}
\end{cases} 
\end{equation*}
\end{defi}

The following proposition follows from Proposition~\ref{1.6} together with Proposition~\ref{prop:Dep} and Proposition~\ref{prop:NSmall}.

\begin{prop}\label{prop:DepthLength}
\begin{enumerate}
\item If $X=\mathbb R^+$ then $\dep_X(\beta)=\dep(\beta)-1$ for all $\beta\in \Phi^+$.
\item If $X=[1,+\infty)$ then $\dep_X(\beta)=\dep_\infty(\beta)$ for all $\beta\in \Phi^+$.
\end{enumerate}
\end{prop}

\begin{remark} Sometimes, the depth of a positive root is defined so that the simple roots equal have depth $0$ and not $1$ as in this text. In this case, then the equality would be $\dep_X(\beta)=\dep(\beta)$ in~{\em (1)} above.
\end{remark}

\begin{ex} \label{e1} We discuss some examples of the functions $d_{X}: \Phi\to \mathbb R$,  restricting to the case in which $X$ is an order ideal of $\mathbb R^+$. 
The claims are readily checked using Proposition \ref{1.6}.
\begin{enumerate}[(a)]  
\item Fix $\alpha\in \Phi^{+}$. If  $W$ is finite and $X=]1,\infty[$, then $d_{X}(\alpha)=0$. If $X=[1,\infty[$, then $d_{X}(\alpha)=1$.  If $X=\mathbb{R}^{+}$, then $d_{X}(\alpha)=\ell(s_{\alpha})$.
\item If $W$ is a  finite  simply laced Weyl group  then
for any order coideal $X$ of $\mathbb{R}^{+}$, one has $d_{X}=d_{Y}$ where $Y$ is either $]1,\infty[$, $[1,\infty[$ or $\mathbb{R}^{+}$.

\item Suppose $W$ is of type $B/C_{n}$, with positive roots
\begin{equation*}
\Phi^{+}=\{\frac{1}{\sqrt{2}}(e_{i}\pm e_{j})\,|\, 1\leq i<j\leq n\}\cup\{e_{i}\,|\, 1\leq i\leq n\}. 
\end{equation*} 
Let $\alpha\in \Phi^{+}$. If $1\not \in X$, then $d_{X}(\alpha)=0$. If $1\in X$ but $\frac{1}{\sqrt{2}}\not\in X$, then $d_{X}(\alpha)=1$. If $\frac{1}{2}\in X$, then $d_{X}=d_{\mathbb R^+}$,  which   is given by  $d_{\mathbb R^+}((e_{i}-e_{j})/\sqrt{2})=2j-2i-1$ and 
 $d_{\mathbb R^+}((e_{i}+e_{j})/\sqrt{2})=4n-2i-2j+1$ if $1\leq i<j\leq n$ and  $d_{\mathbb R^+}(e_{i})=2n-2i+1$ for $1\leq i\leq n$. In the remaining case ($\frac{1}{2}\not\in X$ but $\frac{1}{\sqrt{2}}\in X$), one  checks that $d_{X}((e_{i}-e_{j})/\sqrt{2})=1$ and 
$d_{X}((e_{i}+e_{j})/\sqrt{2})=3$ if $1\leq i<j\leq n$, while $d_{X}(e_{i})=2n-2i+1$ for $1\leq i\leq n$.
\end{enumerate}
\end{ex}

\subsection{The weak order on $\Phi$}\label{1.3}  If $X=\mathbb R^+ $, we abbreviate $d_{\mathbb R^+}$ by $d$ and call it the \emph{standard length function on $\Phi$}. We saw in Example~\ref{ex:Length1} that for any
$\beta\in \Phi^{+}$, we have 
\begin{equation}\label{eq1.3a.1}
d(\pm{\beta})=\pm\ell(s_\beta)=\pm(\dep(\beta)-1).
\end{equation} 
Proposition \ref{1.6} implies that if $\beta\in \Phi$ and $\alpha\in \Delta$, then 
\begin{equation}\label{eq1.3.1}
d(s_{\alpha}(\beta))>d(\beta)\iff d(s_{\alpha}(\beta))=d(\beta)+2\iff B(\alpha,\beta)<0.
\end{equation}

\begin{defi}
Let $\lessdot$ be the relation on $\Phi$ defined by $\alpha\lessdot \beta$ if there is $\gamma\in \Delta$ such that $B(\gamma,\alpha)<0$ and $\beta=s_{\gamma}(\alpha)$ both hold.  We  define the \emph{weak order on $\Phi$} to be the partial order  $\leq$ on $\Phi$ obtained as   the reflexive transitive closure of the relation $\lessdot$. 
\end{defi}

Since $d(\beta)$ is odd for all $\beta\in \Phi$,  \eqref{eq1.3.1} implies that  $\lessdot $ is the covering relation of $\leq$.  Thus, $\alpha\leq \beta$ holds in $\Phi$ if and only if there exist $n\in \mathbb N$ and $\beta_{0},\ldots, \beta_{n}\in \beta$ with
\begin{equation}\label{eq1.3.2}
\alpha=\beta_{0}\lessdot \beta_{1}\lessdot \ldots\lessdot \beta_{n}=\beta.
\end{equation}
Since $d(\beta_{i})=d(\beta_{i-1})+2$ for all $i$ from $1$ to $n$, this implies that
\begin{equation}\label{eq1.3.3}
d(\beta)=d(\alpha)+2n,\qquad n=\frac{1}{2}(d(\beta)-d(\alpha)).
\end{equation}

\begin{remark}\label{1.3h}
\begin{enumerate}[(a)]

\item  If $\alpha\lessdot \beta$, then $\beta=\alpha+c \gamma$ where $c=-2B(\alpha,\gamma)>0$, so $\gamma$ and $c$ are uniquely determined by $\alpha$ and $\beta$.  

\item   In \cite{BjBr05}, the set $\Phi^{+}  $ partially ordered by restriction of $\leq$ is called the \emph{root poset}. Note that  this partial order on  $\Phi^{+}$ is \emph{not}  in general   the order obtained by restricting   weak  order on $W$ to $T=\{s_\beta\,|\,\beta\in \Phi\}$ and transferring to $\Phi^{+}$ via the standard bijection $\beta\mapsto s_{\beta}$.  It is easy to see that if $\beta\in \Phi^{+}  $, then $-\beta\leq \beta$. The map $\beta\mapsto -\beta$ for $\beta\in \Phi$ is an order-reversing bijection of $\Phi$ with itself in  weak order. Note also  that if $\beta\lessdot \alpha$ where $\alpha\in \Delta$, then $\beta=-\alpha$.

\end{enumerate}
\end{remark}


\subsection{The root category}\label{1.3d} 
The relationship between the weak orders on $\Phi$ and $W$ is clarified by the introduction  of a certain category $\mathcal C$, which we shall call the \emph{root category}.

First, let $\mathcal C'=G$ denote  the transformation groupoid of $W$ on $\Phi$; this is the  category (in fact, a groupoid)   such that:
$$
\textrm{$\ob(G)= \Phi$ and  $\mor_{G}(\alpha,\beta)=\{(\beta,w,\alpha)\,|\, w\in W,  w(\alpha)=\beta\}$,}
$$  
with composition defined  by $(\gamma,v,\beta)(\beta,w,\alpha)=(\gamma,vw,\alpha)$. 
 For $\beta\in V$, $c\in \mathbb R$ and $\bullet$ denoting one of the symbols  $=$, $<$ ,$\leq$, $\geq$ or  $>$, define $V_{\beta}^{\bullet c}:=\{\alpha\in V\,|\,B(\alpha,\beta) \bullet c\}$.

\begin{prop}\label{1.3e} 
\begin{enumerate}
\item There is a subcategory $\mathcal C$ of $\mathcal C'$ with all objects of $\mathcal C'$, and only those morphisms $(\beta,w,\alpha)$ in $\mathcal C'$ with $N(w)\subseteq V_{\beta}^{>0}$.
\item There is a duality (contravariant involutive automorphism) of $\mathcal C$ given by $\alpha\mapsto -\alpha$ on objects and $(\beta,w,\alpha)\mapsto (-\alpha,w^{-1},-\beta)$ on morphisms.
\item For  two   morphisms $(\gamma,v,\beta)$ and  $(\beta,w,\alpha)$ of $\mathcal C$, one has 
$\ell(vw)=\ell(v)+\ell(w)$ and hence  $v\leq_R vw$ in right weak order  $\leq_R $  on $W$.
\item Let $(\gamma,x,\alpha)$ be a morphism in $\mathcal C$ and $v\in W$ with $v\leq_R x$. Set $w:=v^{-1}x$ and $\beta:=w(\alpha)=v^{-1}(\gamma)$. Then $(\gamma,v,\beta)$ and $(\beta,w,\alpha)$ are morphisms in $\mathcal C$ and $(\gamma,v,\beta)(\beta,w,\alpha)=(\gamma,x,\alpha)$.
\end{enumerate}
\end{prop}
\begin{proof}
Consider two composable morphisms $(\gamma,v,\beta)$ and  $(\beta,w,\alpha)$ of $\mathcal C'$ with both $N(v)\subseteq V^{>0}_{\gamma}$ and $N(w)\subseteq V^{>0}_{\beta}$. Then $v(N(w))\subseteq v( V^{>0}_{\beta})=V^{>0}_{v(\beta)}=V^{>0}_{\gamma}$. Therefore $v(N(w))\cap -N(v)=\emptyset$. By Corollary~\ref{cor:RedInv}, we get $\ell(vw)=\ell(v)+\ell(w)$, so $v\leq_R vw$, and $N(vw)=N(v)\sqcup v(N(w))\subseteq V^{>0}_{\gamma}$. This implies (1) and (3). Part (2) is readily proved using the fact  that $N(w^{-1})=-w^{-1}(N(w))$. For (4), note that  we have  $N(x)=N(v)\sqcup v(N(w))$ by Corollary~\ref{cor:RedInv} again,  so $N(v)\subseteq N(x)\subseteq V^{>0}_{\gamma}$ and  $N(w)\subseteq v^{-1}(N(x))\subseteq v^{-1}(V^{>0}_{\gamma}) =V^{>0}_{v^{-1}(\gamma)}=V^{>0}_{\beta}$.  
\end{proof}

\begin{cor}\label{1.4}  Let $\alpha,\beta\in \Phi $. Then there is    a natural  bijective correspondence  between the set of maximal chains  from $\alpha$ to $\beta$ in $(\Phi ,\leq)$ and  the set of reduced expressions of elements $w\in W$ such that $(\beta,w,\alpha)\in \mor_{\mathcal C}(\alpha,\beta)$.
\end{cor}
\begin{proof} Consider $\gamma\lessdot \gamma'$ in $(\Phi ,\leq)$. Write $\gamma'=s_{\beta}(\gamma)$ with $\beta\in \Delta$. Then $(\gamma',s_{\beta},\gamma)$ is in~$\mor(\mathcal C)$, by the definitions. On the other hand, if $(\gamma',s_{\beta},\gamma)$ is in~$\mor(\mathcal C)$ with $\beta\in \Delta$, then $\gamma\lessdot \gamma'$ in $(\Phi ,\leq)$.

Given a maximal chain $c: \alpha=\alpha_{0}\lessdot \ldots \lessdot \alpha_{n}=\beta$  in $(\Phi ,\leq)$, let $\beta_{i}\in \Delta$ with $\alpha_{i}=s_{\beta_{i}}(\alpha_{i-1})$. By the above,  $(\alpha_{i},s_{\beta_{i}},\alpha_{i-1})$ is in~$\mor(\mathcal C)$ for $i=1,\ldots, n$ with composite $(\alpha_{n},s_{\beta_{n}},\alpha_{n-1})\cdots (\alpha_{1},s_{\beta_{1}},\alpha_{0})=(\beta,w,\alpha)\in\mor(\mathcal C)$ where $w:=s_{\beta_{n}}\cdots s_{\beta_{1}}\in W$. We have $n=\ell(w)$ by Proposition \ref{1.3e}, and we attach to $c$ the above reduced expression defining $w$.

On the other hand, consider a morphism $(\beta,w,\alpha)$ in $\mathcal C$ and a reduced expression $w=s_{\beta_{n}}\cdots s_{\beta_{1}}$, where all $\beta_{i}\in \Delta$. Let $w_{i}:=s_{\beta_{i}}\cdots s_{\beta_{1}}$ and $\alpha_{i}:=w_{i}(\alpha)$ for $i=0,\ldots, n$. From Proposition \ref{1.3e} again, we see that $c: \alpha=\alpha_{0}\lessdot \ldots \lessdot \alpha_{n}=\beta$ is a maximal chain from $\alpha$ to $\beta$ in $(\Phi ,\leq)$.

 It is clear that the maps of the  two previous paragraphs define inverse bijections as  required.
\end{proof}

\begin{prop} \label{1.3f} For all $\alpha,\beta\in \Phi$ and $w\in W$, one has $(\beta,w,\alpha)\in \mor (\mathcal C)$ if and only if  $\beta=w(\alpha)$ and $d(\beta)=d(\alpha)+2\ell(w)$ hold.
\end{prop}
\begin{proof} Suppose $(\beta,w,\alpha)\in \mor(\mathcal C)$. Then by the proof of Corollary \ref{1.4}, there is a maximal chain from $\alpha$ to $\beta$ in $(\Phi,\leq)$ of length $l(w)$. By  \eqref{eq1.3.3}, this maximal chain necessarily has length $\frac{1}{2}(d(\beta)-d(\alpha))$, proving that $d(\beta)=d(\alpha)+2\ell(w)$.

On the other hand, suppose that $\beta=w(\alpha)$ where $d(\beta)=d(\alpha)+2\ell(w)$. Choose a reduced expression $w=s_{\beta_{n}}\cdots s_{\beta_{1}}$ where each $\beta_{i}\in \Delta$.  Let $w_{i}:=s_{\beta_{i}}\cdots s_{\beta_{1}}$ and $\alpha_{i}:=w_{i}\alpha$ for $i:=0,\ldots, n$. Then $d(\alpha_{i})\leq d(\alpha)+2\ell(w_{i})=d(\alpha)+2i$ and  $d(\alpha_{i})=d(w_{i}w^{-1}\beta)\geq d(\beta)-2\ell(w_{i}w^{-{1}})=d(\beta)-2(\ell(w)-\ell(w_{i})=d(\alpha)+2i$. Hence $d(\alpha_{i})=d(\alpha)+2i$ and it follows from  \S\ref{1.3} that $\alpha=\alpha_{0}\lessdot \alpha_{1}\lessdot\ldots \lessdot \alpha_{n}=\beta$ is  a maximal chain from $\alpha$ to $\beta$ in $(V,\leq)$. By the proof of  Corollary \ref{1.4}, $(\beta,w,\alpha)\in \mor(\mathcal C)$.
\end{proof}

\subsection{Paths in weak order  and  length functions}\label{1.5a}   We introduce more systematic notation for certain invariants attached to paths  in the poset $(\Phi,\leq)$. 

Let  $p:\alpha=\alpha_{0}\lessdot \alpha_{1}\lessdot\ldots\lessdot  \alpha_{n}=\beta$ be a maximal chain in the interval 
$$
[\alpha,\beta]:=\{\gamma\in \Phi\,|\, \alpha\leq \gamma\leq \beta\}
$$ 
in $(\Phi,\leq)$. Let $\beta_{i}\in \Delta$ with $\alpha_{i}=s_{\beta_{i}}(\alpha_{i-1})$.  
Then  $(\alpha_{i},s_{\beta_{i}},\alpha_{i-1})$ is in~$\mor(\mathcal C)$ for $i=1,\ldots, n$ with composite 
$$
(\alpha_{n},s_{\beta_{n}},\alpha_{n-1})\cdots (\alpha_{1},s_{\beta_{1}},\alpha_{0})=(\beta,w,\alpha)\in\mor(\mathcal C)
$$
 where $w:=s_{\beta_{n}}\cdots s_{\beta_{1}}\in W$. We have $n=\ell(w)$ by Proposition~\ref{1.3e}. 
 
 For $i:=1,\ldots, n$, let   $c_{i}:=B(\beta_{i},\alpha_{i})\in \mathbb R^+$,  
$$
w_{i}:=s_{\beta_{i}}\cdots s_{\beta_{1}}\quad \textrm{ and }\quad 
 \gamma_{i}:=ww_{i}^{-1}(\beta_{i})=s_{\beta_{n}}\ldots s_{\beta_{i+1}}(\beta_{i})\in \Phi^{+} .
 $$ 
 We call $\mathcal L(p):=n$ the \emph{length} of $p$ and $\mathcal L_{W}(p):=w$ the \emph{$W$-length} of $p$.
Define the \emph{simple root label} $\Delta(p):=(\beta_{n},\ldots, \beta_{1})\in \Delta^{n}$, the \emph{root label} $\Phi(p)=(\gamma_{n},\ldots, \gamma_{1})\in \Phi^{n}$ and the  \emph{numerical label} $c(p):=(c_{n},\ldots, c_{1})\in (\mathbb R^+ )^{n}$. We record the following simple relationships amongst these invariants for reference.
\begin{equation}\label{eq2.12.1} w=s_{\beta_{n}}\ldots s_{\beta_{1}}=s_{\gamma_{1}}\cdots s_{\gamma_{n}}, \qquad n=\ell(w),
\qquad N(w)=\{\gamma_{1},\ldots, \gamma_{n}\}.
\end{equation}
\begin{equation} \label{eq2.12.2} 
 \alpha_{i}=w_{i}(\alpha)=w_{i}w^{-1} (\beta), \quad i=0,\ldots, n
\end{equation}
\begin{equation}\label{eq2.12.3}
 \beta_{i}=w_{i}w^{-1}(\gamma_{i}), \quad 
c_{i}=B(\alpha_{i},\beta_{i})=B(\beta,\gamma_{i}), \quad i=1,\ldots, n.
\end{equation}
\begin{equation}\label{eq2.12.4} 
s_{\beta}=ws_{\alpha}w^{-1},\qquad  d(\beta)=d(\alpha)+2\ell(w).
\end{equation}
\begin{equation}\label{eq2.12.5} 
\text{\rm  $ N(s_{\beta})=N(w)\sqcup w(N(s_{\alpha}))\sqcup -s_{\beta} (N(w))$ if $\alpha\in \Phi^{+}  $.}
\end{equation}
\begin{equation}\label{eq2.12.6} 
\text{\rm $w (N(s_{a}))=\{\beta\}$ if  $\alpha\in \Delta$.}
\end{equation}

\begin{cor} \label{1.7} Consider $\alpha\leq \beta$ in  $(\Phi,\leq)$. Fix  a  path $p$ from $\alpha$ to $\beta$  and write its numerical label as   $\Delta(p)=(c_{1},\ldots, c_{n})$.  Also let $(\beta,w,\alpha)$, where $w\in W$,  be in $\mor_{\mathcal C}(\alpha,\beta)$ (e.g. $w=\mathcal L_{W}(p)$ is the $W$-length of $p$). Then for any $X\subseteq \mathbb R^+$, 
\begin{equation*}
\vert  \{ \gamma\in N(w)\,|\, B(\gamma,\beta)\in X\}\vert =\vert\{i\,|\, 1\leq i\leq n, c_{i}\in X\}\vert =\frac{d_{X}(\beta)-d_{X}(\alpha)}{2}.
\end{equation*}  
In particular, the two cardinalities  depend only on $\alpha$ and $\beta$, and 
for any  $\alpha\leq \beta$ in~$\Phi$, one has $d_{X}(\alpha)\leq d_{X}(\beta)$.
\end{cor}
\begin{proof} The second equality in the displayed equation follows from Proposition~\ref{1.6} by induction on $\mathcal L(p)$. To prove the first equality in that equation, one may, by  the proof of Corollary \ref{1.4}, assume without loss of generality that  $w=\mathcal L_{W }(p)$. Then the first equality is a consequence of~\eqref{eq2.12.1} and \eqref{eq2.12.3}.  The final statement of the corollary follows from the displayed equation.
\end{proof}

\begin{remark} \label{2.18} 
\begin{enumerate}[(a)] 
\item The corollary refines (with simpler proof)  unpublished results of the first author (M.~D) used  in his proof of \ref{thm:NSmall} and related results.

\item The corollary shows that, for all $X\subseteq \mathbb R^+ $, the function $d_{X}$ is monotonic non-decreasing with respect to weak order on $\Phi$. 
\end{enumerate}
\end{remark}

\subsection{Length functions and Bruhat order on root systems}\label{s3}
In view of  Remark~\ref{rem:Ideal}  , we assume now that $X$ is an order coideal  of $\mathbb R^+$: $X$ is either $\emptyset$, $\mathbb R^+ $ or an open or closed ray $[a,\infty[$ or $]a,\infty[$ for some $a>0$.
 
  The  following  lemma (in the case $X=[1,\infty[$) is needed for the proof of Proposition~\ref{prop:Increase}.

 \begin{lem}\label{1.9} 
\begin{enumerate}
\item Let $(\alpha,w,\alpha')$ be a morphism in $\mathcal C$. Then
\begin{equation*}
d_{X}(s_{\alpha}(\beta))-d_{X}(\beta)\geq d_{X}(s_{\alpha'}(\beta'))-d_{X}(\beta')
\end{equation*}  
for all  $\beta,\beta'\in \Phi$ such that  $\beta=w(\beta')$ and   $B(\alpha,\beta)\leq 0$.
\item Let  $\beta\in \Phi$, $\alpha\in \Phi^{+}  $  with $B(\alpha,\beta)\leq 0$.  Then $d_{X}(s_{\alpha}(\beta))\geq  d_{X}(\beta)$ with strict inequality if $B(\alpha,\beta)\in -X$. 
\end{enumerate} 
\end{lem}

\begin{proof} Note that in {\em (1)}, $\alpha=w(\alpha')$  so $B(\alpha,\beta)=B(\alpha',\beta')$ for $\beta,\beta'$ as there.  By a  predicate, we mean here  a function which attaches to each element of its domain  a truth value. Consider the predicate which attaches to any morphism $(\alpha,w,\alpha')$ in $\mathcal C$ the truth value of the second sentence  in {\em (1)}.    It is easy to check that this predicate   is true on identity morphisms and that if it is true on each of two composable morphisms in $\mathcal C$, it is true on their composite. Since every (non-identity) morphism in $\mathcal C$ is a composite of morphisms $(\alpha,s,\alpha')$ with $s\in S$, we may assume without loss of generality that $w=s_{\gamma}$ for some $\gamma\in \Delta$. 

 Rearranging the equation in  {\em (1)}, we see now that it  holds  if and only if    
\begin{equation}\label{1.9.1}
d_{X}(s_{\alpha}(\beta))-d_{X}(s_{\gamma}s_{\alpha}(\beta))\geq d_{X}(\beta)-d_{X}(s_{\gamma}(\beta))
\end{equation}
for all $\alpha,\beta\in \Phi$ and $\gamma\in \Delta$ with $B(\alpha,\beta)\leq 0$ and $B(\gamma,\alpha)>0$. We have $s_{\alpha}(\beta)=\beta+c\alpha$ where $c:=-2B(\alpha,\beta)\geq 0$, so $B(\gamma,s_{\alpha}(\beta))=B(\gamma,\beta)+cB(\gamma,\alpha)\geq B(\gamma,\beta)$.  Then~Equation~\eqref{1.9.1} follows by computing both sides  by~Proposition \ref{1.6} using $B(\gamma,s_{\alpha}(\beta))\geq B(\gamma,\beta)$. This completes the proof of {\em (1)}. 
\smallskip

Now let $\alpha,\beta$ be as in {\em (2)}. Since $\alpha\in \Phi^{+}  $, we may choose a morphism $(\alpha,w,\alpha')$ in $\mathcal C$ with $\alpha'\in \Delta$, and set $\beta'=w^{-1}(\beta)$. By {\em (1)}, $d_{X}(s_{\alpha}(\beta))-d_{X}(\beta)\geq d_{X}(s_{\alpha'}(\beta'))-d_{X}(\beta')$. But since $B(\alpha',\beta')\leq 0$, Proposition~\ref{1.6} implies that  $d_{X}(s_{\alpha'}(\beta'))-d_{X}(\beta')\geq 0$  with strict inequality if  $B(\alpha',\beta')\in -X$ i.e. if $B(\alpha,\beta)\in -X$. This proves {\em (2)}.
 \end{proof}

 The remaining results of this subsection are not required in the proof of  Proposition \ref{prop:Increase}. They provide a more conceptual interpretation of Lemma \ref{1.9}(2) 
as a monotonicity property of the length functions $d_{X}$ on $\Phi$, in terms  of  a second  partial order $\leq'$ on $\Phi$ which we now define.  

\begin{defi}
\begin{enumerate}[(i)]
\item  Let $\lessdot'$ be the relation on $\Phi$ defined by $\alpha\lessdot' \beta$ if there is $\gamma\in \Phi^{+}  $ such that $c:=B(\gamma,\alpha)<0$ and $\beta=s_{\gamma}(\alpha)$.  This implies that $\beta-\alpha=-c\gamma$,  so $\gamma$ and $c$ are uniquely determined by $\alpha$ and $\beta$.   
\item Define a preorder (reflexive transitve relation)   $\leq'$ on $\Phi $ as the reflexive, transitive closure of the relation $\lessdot'$.  
\end{enumerate}
\noindent If $\alpha\leq '\beta$, then $\beta-\alpha\in \cone(\Phi)$. Since $\cone(\Phi)\cap -\cone(\Phi)=\{0\}$, it follows that~$\leq'$ is anti-symmetric i.e. it is a partial order, which we call  the \emph{Bruhat order} on $\Phi $.
\end{defi}

\begin{remark} 
\begin{enumerate}[(a)]
\item Note that $\lessdot'$ is not the covering relation of $\leq'$ in general. 

\item The Bruhat order is stronger than weak order, in the sense that $\alpha\leq \beta$ implies  $\alpha\leq'\beta$, for any $\alpha,\beta\in \Phi $.  

\item The assertions of \ref{1.3h}(2), excluding those in the  last sentence there, still hold with ``weak order'' replaced by ``Bruhat order'' and ``$\leq$'' replaced by ``$\leq'$''. 

\item  Analogues of weak and Bruhat orders on $W$-orbits on $V$ will be  discussed more systematically in relation to orders on $W$   by the first author (M.D.) elsewhere.
\end{enumerate}
\end{remark}

Length functions have the following  interesting behavior in relation to the Bruhat order on $\Phi$.

\begin{thm} \label{1.10} Let  $W'$ be a reflection subgroup of $W$, and  $\leq'$ (resp., $\leq'_{W'}$) denote Bruhat order on $\Phi$ (resp., $\Phi_{W'}$).
\begin{enumerate}
\item If $\beta,\beta'\in \Phi$ and  $\beta\leq'\beta'$, then $d_{X}(\beta)\leq d_{X}(\beta')$.
\item If $\beta,\beta'\in \Phi_{W'}$ and $\beta\leq'_{W'}\beta'$, then $ \beta\leq'\beta'$ and    
$$
0\leq d_{W',X}(\beta')-d_{W',X}(\beta)\leq d_X(\beta')-d_{X}(\beta).
$$ 
\item  If $\beta'\in \Phi_{W'}^{+}$, then $0\leq d_{W',X}(\beta')\leq d_{X}(\beta')$.
\end{enumerate}
\end{thm}



\begin{proof}[Proof of Theorem~\ref{1.10}] Part {\em (1)} follows from Lemma \ref{1.9}{\em (2)} and the definition of Bruhat order $\leq'$ on $\Phi$.  Part {\em (3)} follows from {\em (2)} by taking $\beta:=-\beta'$, so $\beta\leq'_{W'}\beta'$ since  $\beta'\in \Phi_{W'}^{+}$, and  dividing the resultant inequalities in {\em (2)} by $2$.

It remains to prove (b). By definition of $\leq'_{W'}$, it is sufficient to do this in the case that $\beta\lessdot'_{W'}\beta'$ i.e. when  $\beta'=s_\alpha(\beta)$ for some $\alpha\in \Phi_{W'}^{+}$ with $\mpair{\alpha,\beta}<0$. This immediately implies $\beta\lessdot'\beta'$, so $\beta\leq' \beta'$. One has $0\leq d_{W',X}(\beta')-d_{W',X}(\beta)$ by {\em (1)} applied to $W'$. To complete the proof of {\em (2)}, we shall  show  by induction on $n\in \mathbb N$ that  for all reflection subgroups $W'$ of $W$ and all $\beta\in \Phi_{W'}$ and $\alpha\in \Phi_{W'}^{+}$ with $B(\alpha,\beta)<0$ and $d(\alpha)=2n+1$, one has  \begin{equation}\label{1.10.1}
 d_{W',X}(s_{\alpha}(\beta))-d_{W',X}(\beta)\leq d_X(s_{\alpha}(\beta))-d_{X}(\beta).
\end{equation}
 
If $n=0$, then $\alpha\in \Delta$. Since also $\alpha\in \Phi_{W'}^{+}$, we have $\alpha\in \Delta_{W'}^{+}$. Equality holds in~\eqref{1.10.1}, by computing both sides using Proposition~\ref{1.6} applied to $W'$ and $W$.

Assume now that $n>0$. Choose $\gamma\in \Delta$ so $\alpha':=s_{\gamma}(\alpha)\lessdot \alpha$. Then $d(\alpha')=2n-1>0$ so $\alpha'\in \Phi^{+}$. Define the reflection subgroup $W'':=s_{\gamma}W' s_{\gamma}$ of $W$  and the root $\beta':=s_{\gamma}(\beta)\in \Phi$.  One has $\alpha'\in \Phi_{W''}^{+}$ and $\beta'\in \Phi_{W''}$  with $B(\alpha',\beta')<0$ and  $d(\alpha')=2n-1$.  By induction, we have 
 $d_{W'',X}(s_{\alpha'}(\beta'))-d_{W'',X}(\beta')\leq d_X(s_{\alpha'}(\beta'))-d_{X}(\beta')$.
That is, 
 \begin{equation} \label{1.10.2} 
 d_{W'',X}(s_{\gamma}s_{\alpha}(\beta))-d_{W'',X}(s_{\gamma}(\beta) )\leq d_X(s_{\gamma}s_{\alpha}(\beta))-d_{X}(s_{\gamma}(\beta)).
\end{equation} 

We now consider two cases as follows. First, suppose that $\gamma\in \Phi_{W'}$. Then
$W''=W'$ and  $\gamma\in \Delta_{W'}$.  By Proposition~\ref{1.6} applied to both $W$ and $W'$, one has $d_{W',X}(s_{\gamma}(\beta''))-d_{W',X}(\beta'')=d_{X}(s_{\gamma}(\beta''))-d_{X}(\beta'')$ for all $\beta''\in \Phi_{W'}$.
Taking $\beta'':=s_{\alpha}(\beta)$ and $\beta'':=\beta$ in turn in this, we see that~\eqref{1.10.1} follows from \eqref{1.10.2}.

The second and final case is that in which $\gamma\not\in \Phi_{W'}$. Since $\gamma\in \Delta\setminus \Delta_{W'}$, we have $\Delta_{W''}=s_{\gamma}(\Delta_{W'})$. It follows that the map $\beta''\mapsto s_{\gamma}(\beta'')$ defines  a bijection  $\Phi_{W'}\to \Phi_{W''}$.  This bijection is  an order isomorphism $(\Phi_{W'},\leq_{W'})\to (\Phi_{W''},\leq_{W''})$ in the corresponding  weak orders,
 restricts to   bijections $\Delta_{W'}\to \Delta_{W''}$ and $\Phi_{W'}^{+}\to \Phi_{W''}^{+}$,
 and preserves bilinear forms in the sense $B(s_{\gamma}(\beta''),s_{\gamma}(\beta'''))=
B(\beta'',\beta''')$ for all $\beta'',\beta'''\in \Phi_{W'}$. It follows that $d_{W'',X}(s_{\gamma}(\beta''))=d_{W',X}(\beta'')$ for all $\beta''\in \Phi_{W'}$. Taking $\beta''$ equal to $s_{\alpha}(\beta)$ and to $\beta$ in turn shows that the left hand sides of  \eqref{1.10.2} and \eqref{1.10.1} are equal. On the other hand, by \eqref{1.9.1}, the right hand side of \eqref{1.10.2} is less than or equal to  the right hand side of \eqref{1.10.1}. Hence \eqref{1.10.1} also follows from \eqref{1.10.2} in this case. This completes the proof of {\em (2)} and of the theorem. 
 \end{proof}

\subsection{Consequences for $\infty$-depth and  dominance order}\label{s4} Henceforward, we take $X=[1,\infty)$. Recall that for $\alpha\in \Phi^{+}$ we have:
$$
d_{[1,\infty)}(\alpha)=2\dep_{\infty}(\alpha)+1=2\vert\dom(\alpha)\vert+1 ,
$$ 
where $\dom(\alpha)=\{\beta\in \Phi^+\,|\, \beta\prec \alpha\}$.  Dominance order and $\infty$-depth are related to~$(\Phi,\leq)$ as follows. 

\begin{prop}\label{2.60} Let $\beta\in \Phi^{+}$. Choose a path  $p:\alpha_{0}\lessdot \alpha_{1}\lessdot \ldots\lessdot \alpha_{n}=\beta$ in $(\Phi,\leq)$ where $\alpha_{0}\in \Delta$. Write $c(p):=(c_{n},\ldots, c_{1})$ and $\Phi(p)=(\gamma_{n},\dots, \gamma_{1})$. 
Then \begin{enumerate}
\item $\dom(\beta)=\{\gamma_{i}\,|\, 1\leq i\leq n,  c_{i}\geq 1\}$. 
\item $\dep_{\infty}(\beta)=\vert \{i\,|\, 1\leq i\leq n,  c_{i}\geq 1\}\vert$.
\item If $(\beta,y,\alpha)$ is in $\mor(\mathcal C)$ and $\alpha\in \Delta$, then 
$$
\dom(\beta)=\{\gamma\in N(y)\,|\, B(\gamma,\beta)\geq 1\}.
$$
\item   $\dom(\beta)=y(\dom( \alpha))\sqcup \{\gamma\in N(y)\,|\, B(\gamma,\beta)\geq 1\}$ if $(\beta,y,\alpha)\in \mor(\mathcal C)$ and $\alpha\in \Phi^{+}$.
 \item $\dom(\beta)=\{\gamma\in N(s_{\beta})\,|\, B(\gamma,\beta)\geq 1, \ell(s_{\gamma})<\ell(s_{\beta})\}$.
\end{enumerate}
\end{prop} 
\begin{proof} We  prove {\em (1)} using the facts listed in \S\ref{1.5a}. Write $\Delta(p)=(\beta_{n},\ldots, \beta_{1})$.  Suppose $1\leq i\leq n$ and $c_{i}\geq 1$. Then  $B(\gamma_{i},\beta)=c_{i}\geq 1$. Also, since $\gamma_{i}=s_{\beta_{n}}\cdots s_{\beta_{i+1}}(\beta_{i})$, we have 
\begin{equation*}
\ell(s_{\gamma_{i}})\leq 2(n-i)+1<2n+1=\ell(s_{\beta}).
\end{equation*} 
This implies that $\gamma_{i}\prec \beta$.
On the other hand, let $\gamma\in \Phi^{+}  $ with $\gamma\prec \beta$. Set $w:=\mathcal L_{W}(p)$, so $w\alpha_{0}=\beta$. Hence $s_{\alpha_{0}}w^{-1}(\beta)=-\alpha_{0}$,  we have $\gamma\in N(ws_{\alpha_{0}})=N(w)\sqcup\{\beta\}=\{\gamma_{n},\ldots, \gamma_{1},\beta\}$. Therefore $\gamma=\gamma_{i}$ for some $1\leq i\leq n$, with $1\leq B(\gamma_{i},\beta)=c_{i}$.  Part {\em (2)} follows by taking cardinalities in the equality in {\em (1)}. 

To prove {\em (3)} and {\em (4)}, we may suppose without loss of generality that $\alpha=\alpha_{0}$ and $y=\mathcal L_{W}(p)$;  so $\beta=y(\alpha_{0})$ . Then  for {\em (3)}  $N(y)=\{s_{\gamma_{1}},\ldots, s_{\gamma_{n}}\}$ and $B(\gamma_{i},\beta)=c_{i}$ for $i=1,\ldots, n$, so {\em (3)} follows from {\em (1)}. For {\em (4)} choose a morphism $(\alpha,z,\delta)$ in $\mathcal C$ with $\delta\in \Delta$. Then $(\beta,yz,\delta) = (\beta,y,\alpha)(\alpha,z,\delta)$ in $\mor(\mathcal C)$. By {\em (3)}, 
$$
\dom(\beta)=\{\gamma\in N(yz)\,|\, B(\gamma,\beta)\geq 1\}\ \textrm{ and  }\ \dom(\alpha)=\{\gamma\in N(z)\,|\, B(\gamma, \alpha)\geq 1\}.
$$
 But $N(yz)=N(y)\sqcup y(N(z))$ and for $\gamma\in N(z)$, one has $B(\gamma,\alpha)=B(y(\gamma),\beta)$, so {\em (4)} follows.

Finally, we prove {\em (5)}. Let $y=\mathcal L_{W}(p)$. Then $\beta=y(\alpha_{0})$ and  $s_{\beta}=ys_{\alpha_{0}}y^{-1}$ where $\ell(s_{\beta})=2\ell(y)+1$. Hence
$N(s_{\beta})$ is the disjoint union $N(s_\beta)=N(y)\sqcup \{\beta\}\sqcup -s_{\beta }(N(y))$. By {\em (3)}, it will suffice to show that if $\gamma\in \{\beta\}\cup -s_{\beta }(N(y))$, then either $\ell(s_{\gamma})\geq \ell(s_{\beta})$ or $B(\gamma,\beta)<1$. This is trivial if $\gamma=\beta$. Otherwise, write $\gamma=-s_{\beta}(\gamma')$ where $\gamma'\in N(y)$. Then $\ell(s_{\gamma'})<\ell(s_{\beta})$ and 
 $B(\gamma,\beta)=B(\gamma',\beta)$. If $B(\gamma',\beta)<1$, then $B(\gamma,\beta)<1$ also.  If $B(\gamma',\beta)\geq 1$, then $\gamma'\prec \beta$ by {\em (3)}. 
  
  The following facts can be easily deduced from the definition of dominance order, see \cite[Lemma~2.2]{BrHo93}:  the action of~$W$ preserves the dominance order so   $-\gamma=s_{\beta}(\gamma')\prec s_{\beta}(\beta)=-\beta$;   multiplication by $-1$ reverses the dominance order so $\beta\prec \gamma$ and in this case $\dep(\beta)<\dep(\gamma)$. Therefore by Example~\ref{ex:Length1} we obtain $\ell(s_{\beta})<\ell(s_{\gamma})$ as required. 
 
\end{proof}

\subsection{Proof of Proposition~\ref{prop:Increase}}\label{2.24}

 Assume first for the proof of \ref{prop:Increase}(1) that $W$ is infinite.  We use here Proposition~\ref{2.60}. Write $\Delta_{W'}=\{\alpha,\beta\}$. Note that the inner product of any two roots in 
$\Phi_{W'}$ is greater than or equal to $1$ in absolute value, since $W'$ is infinite.  
 There is some $\rho\in \Delta_{W'}$ with $B(\rho,\delta)>0$. Interchanging 
$\alpha$ and $\beta$ if necessary, we assume without loss of generality  that $B(\alpha,\delta)>0$. 
If also $B(\alpha,\gamma)>0$, then $\alpha\preceq \gamma\prec \delta$ since 
$1=\ell_{W'}(s_{\alpha})\leq \ell_{W'}(s_{\gamma})<\ell_{W'}(s_{\delta})$. Hence 
 $\dom(\delta)\supseteq \dom(\gamma)\sqcup \{\gamma\}$ and $\dep_{\infty }(\delta)=\vert \dom(\delta)\vert \geq \vert \dom(\gamma)\vert+1>\dep_{\infty}(\gamma)$. The other case is that $B(\alpha,\gamma)\leq -1$. Let $\gamma':=s_{\alpha}\gamma$. Then $\dep_{\infty}(\gamma)<\dep_{\infty}(\gamma')$ by Lemma~\ref{1.9}(2). We have $B(\alpha,\gamma')>0$, 
$B(\alpha, \delta)>0$ and  $\ell_{W'}(s_{\gamma'})=\ell_{W'}(s_{\gamma})+2\leq \ell_{W'}(s_{\delta})$, so
 $\alpha\prec \gamma'\preceq  \delta$. This gives either $\gamma'=\delta$ or else 
$\dep_{\infty}(\gamma')<\dep_{\infty}(\delta)$, by arguing as before but with $\gamma$ replaced by $\gamma'$. In either case,  $\dep_{\infty}(\gamma)<\dep_{\infty}(\gamma')\leq \dep_{\infty}(\delta)$ as required. 

Now assume  for the proof of \ref{prop:Increase}(2) that $W'$ is finite.
By induction on $\ell_{W'}(x)$, it is sufficient to prove the assertion there in the special
case that $x=s_{\alpha}$  where $\alpha\in \Delta_{W'}$. But then the hypotheses $\delta=s_{\alpha}\gamma$ with $\ell_{W'}(s_{\delta})=\ell_{W'}(s_{\gamma)}+2$ imply that $B(\gamma,\alpha)<0$ and  so $\dep_{\infty}(\gamma)\leq \dep_{\infty}(s_{\alpha}(\gamma))=\dep_{\infty}(s_{\delta})$ as required, by 
Lemma~\ref{1.9}(2) again.  This finishes the proof of Proposition~\ref{prop:Increase}.

\subsection*{Acknowledgment} 

We deeply thank Patrick Dehornoy for all the interesting and exciting discussions we had on this subject, and for many suggestions on a preliminary version of this article that improved the present text. 
  The second author (CH) wishes to thank Christian Kassel and the Institut de Recherche Math\'ematique Avanc\'ee (IRMA), Universit\'e de Strasbourg, to have hosted him during his sabbatical leave from October 2013 to June 2014 and where a significant part of this work was done.   The second author wishes also to warmly thank  Nathan Williams for many interesting conversations on the subject, in particular for inspiring the counterexample in type $\tilde G_2$ that the smallest Garside shadow is not the set of low elements.



\end{document}